\documentclass[invmat,numbook]{svjour}

\usepackage[adobe-utopia]{mathdesign}
\usepackage{isomath}
\SetMathAlphabet{\mathbf}{normal}{OML}{mdput}{b}{n}

\usepackage{graphics}
\usepackage[all]{xy}
\usepackage{amsmath}
\usepackage[colorlinks=true,
urlcolor=blue, bookmarks=true, bookmarksopen=true, citecolor=blue]{hyperref}

\def\bos#1{{\mathbf{#1}}}

\newcommand{\f}[2]{\frac{#1}{#2}}

 \newcommand{\pa}[1]{\frac{\partial}{\partial{#1}}}
  \newcommand{\prt}[2]{\frac{\partial{#1}}{\partial{#2}}}
 \newcommand{\Sch}{\operatorname{Sch}}
  \newcommand{\Wey}{\operatorname{Wey}}
  \newcommand{\pr}{\operatorname{prim}}
  \newcommand{\frC}{\mathfrak{C}}
  \newcommand{\frL}{\mathfrak{L}}
  \newcommand{\frH}{\mathfrak{H}}
  \newcommand{\Des}{\operatorname{\mathfrak{Des}}}

  \newcommand{\nc}{\newcommand}
  \newcommand{\be}{\begin{eqnarray*}}
  \newcommand{\ee}{\end{eqnarray*}}
  \newcommand{\bea}{\begin{eqnarray}}
  \newcommand{\eea}{\end{eqnarray}}

   \nc{\bei}{\begin{itemize}}
   \nc{\eei}{\end{itemize}}
   \nc{\bee}{\begin{enumerate}}
   \nc{\eee}{\end{enumerate}}
   \nc{\bet}{\begin{theorem}}
   \nc{\eet}{\end{theorem}}
   \nc{\bed}{\begin{definition}}
   \nc{\eed}{\end{definition}}
   \nc{\bel}{\begin{lemma}}
   \nc{\eel}{\end{lemma}}
   \nc{\bep}{\begin{proposition}}
   \nc{\eep}{\end{proposition}}
   \nc{\bec}{\begin{corollary}}
   \nc{\eec}{\end{corollary}}
   \nc{\ber}{\begin{remark}}
   \nc{\eer}{\end{remark}}
   \nc{\beex}{\begin{example}}
   \nc{\eeex}{\end{example}}

   \nc{\bpm}{\begin{pmatrix}}
   \nc{\epm}{\end{pmatrix}}
   \nc{\bspm}{\left(\begin{smallmatrix}}
   \nc{\espm}{\end{smallmatrix}\right)}

\newcommand{\cA}{\mathcal{A}}
\newcommand{\cB}{\mathcal{B}}
\newcommand{\cC}{\mathscr{C}}

\newcommand{\cF}{\mathcal{F}}
\newcommand{\cG}{\mathcal{G}}
\newcommand{\cH}{\mathcal{H}}
\newcommand{\cI}{\mathcal{I}}

\newcommand{\cM}{\mathcal{M}}
\newcommand{\cN}{\mathcal{N}}
\newcommand{\cO}{\mathcal{O}}
\newcommand{\cP}{\mathcal{P}}

\newcommand{\cR}{\mathcal{R}}
\newcommand{\cS}{\mathcal{S}}

\newcommand{\cX}{\mathcal{X}}

\newcommand{\cZ}{\mathcal{Z}}

\newcommand{\bC}{\mathbb{C}}

\newcommand{\bH}{\mathbb{H}}

\newcommand{\bN}{\mathbb{N}}

\newcommand{\bQ}{\mathbb{Q}}
\newcommand{\bR}{\mathbb{R}}

\newcommand{\bZ}{\mathbb{Z}}

\newcommand{\BA}{\mathbf{A}}

\newcommand{\BL}{\mathbf{L}}

\newcommand{\BP}{\mathbf{P}}

\nc{\frf}{\mathfrak{f}} 

\newcommand{\frg}{\mathfrak g}

\nc{\frs}{\mathfrak{s}}  
\nc{\frt}{\mathfrak{t}} 
\nc{\fru}{\mathfrak{u}}
  
\nc{\lsl}{\mathfrak{sl}}
\nc{\lgl}{\mathfrak{gl}}

\nc{\upsi}{\underline{\psi}}
\nc{\uchi}{\underline{\chi}}

\DeclareMathOperator{\rk}{rk}

\DeclareMathOperator{\Hom}{Hom}

\DeclareMathOperator{\coker}{coker}
\DeclareMathOperator{\im}{im}

\DeclareMathOperator{\CH}{CH}
\DeclareMathOperator{\CT}{CT}

\DeclareMathOperator{\id}{id}

\DeclareMathOperator{\End}{End}

\newcommand{\lra}{\longrightarrow}    

\nc{\surjto}{\twoheadrightarrow}
\nc{\ts}{\times}
\nc{\ds}{\displaystyle}
\nc{\nd}{\noindent}  
\nc{\ud}{\underline}
\nc{\ov}{\overline}
\nc{\maplra}[1]{\buildrel #1 \over \lra}
\nc{\mapto}[1]{\buildrel #1 \over \to}
\nc{\setb}[1]{\{  #1\}}

 \nc{\cHom}{\mathcal{H}om}

\newcommand{\CC}{\mathbb{C}}

\newcommand{\GG}{\mathbb{G}}

\begin{document}

\nc{\cdruur}[8] {\begin{CD} 
#1 @>#2>> #3\\ 
@AA#4A @AA#5A\\ 
#6 @>#7>> #8 
\end{CD} }
\nc{\cdrddr}[8] {\begin{CD} 
#1 @>#2>> #3\\ 
@VV#4V @VV#5V\\ 
#6 @>#7>> #8 
\end{CD} }


\nc{\dia}[8]{\xymatrix{ 
&#1 \ar@{-}[ld]_{#2} \ar@{-}[rd]^{#3} \\
#4 \ar@{-}[rd]_{#6} & &#5 \ar@{-}[ld]^{#7}\\ 
&#8} }

\nc{\diam}[9]{\xymatrix{ 
&#1 \ar@{-}[ld]_{#2}  \ar@{-}[d]^{#3} \ar@{-}[rd]^{#4} \\
#5 \ar@{-}[rd]_{#8}     & #6 \ar@{-}[d]_{#9}      & #7   \ar@{-}[ld]^{2} \\
& \bQ} } 

\nc{\sumn}[2][n]{#2_{1} +#2_{2}+ \cdots + #2_{#1}}
\nc{\poly}[3][n]{#2_{#1}#3^{#1} +#2_{#1-1}#3^{#1-1}  \cdots + #2_{1} #3+ #2_0}
\nc{\dpoly}[3][n]{#1#2_{#1}#3^{#1-1} +(#1-1)#2_{#1-1}#3^{#1-1}  \cdots +2 #2_{2} #3+ #2_1}
\nc{\mpoly}[3][n]{#3^{#1} +#2_{#1-1}#3^{#1-1}  \cdots + #2_{1} #3+ #2_0}

\nc{\vpar}[4]{    \left \{ \begin{array}{cc} #1 & \textrm{if } #2, \\
&\\
#3 & \textrm{if } #4. 
\end{array}\right. }

\nc{\vparr}[4]{    \left \{ \begin{array}{cc} #1 & \textrm{if } #2, \\
&\\
#3 & \textrm{if } #4, 
\end{array}\right. }

\nc{\ary}[5]{#1: \left\{ \begin{array}{ll} #2 &\mapsto #3 \\ #4 &\mapsto #5 \end{array} \right.}
 \nc{\bedm}{\begin{displaymath}}
 \nc{\eedm}{\end{displaymath}}
 \nc{\art}{\hbox{\bf Art}^\Z}
 \nc{\bvx}{\bos{B\!\!V}_{\! \!X}}

\newcommand{\pmat}{\left(\begin{matrix}}   
\newcommand{\epmat}{\end{matrix}\right)}   
\newcommand{\psmat}{\left(\begin{smallmatrix}}    
\newcommand{\epsmat}{\end{smallmatrix}\right)}
\nc{\twotwo}[4]{\pmat #1 & #2 \\ #3 & #4 \epmat}
\nc{\thrthr}[9]{\pmat #1 & #2 & #3 \\ #4 & #5 & #6 \\ #7 & #8 & #9 \epmat}
\nc{\stwotwo}[4]{\psmat #1 & #2 \\ #3 & #4 \epsmat}
\nc{\sthrthr}[9]{\psmat #1 & #2 & #3 \\ #4 & #5 & #6 \\ #7 & #8 & #9 \epsmat}

\newcommand{\fixme}[1]{\footnote{#1}}


\def\eqalign#1{\null\,\vcenter{\openup\jot\m@th
\ialign{\strut\hfil$\displaystyle{##}$&$\displaystyle{{}##}$\hfil
\crcr#1\crcr}}\,}

\def\eqn#1#2{
\xdef #1{(\nsecsym\the\meqno)}
\global\advance\meqno by1
$$#2\eqno#1\eqlabeL#1
$$}


%
\def\CA{{\mathcal A}}
\def\CB{{\cal B}}
\def\CC{{\cal C}}
\def\CD{{\cal D}}
\def\CE{{\cal E}}
\def\CF{{\cal F}}
\def\CG{{\mathcal G}}
\def\CH{{\cal H}}
\def\CI{{\cal I}}
\def\CJ{{\cal J}}
\def\CK{{\cal K}}
\def\CL{{\cal L}}
\def\CM{{\cal M}}
\def\CN{{\mathcal N}}
\def\CO{{\cal O}} 
\def\CP{{\cal P}}
\def\CQ{{\cal Q}}
\def\CR{{\cal R}}
\def\CS{{\cal S}}
\def\CT{{\cal T}}
\def\CU{{\cal U}}
\def\CV{{\cal V}}
\def\CW{{\cal W}}
\def\CX{{\cal X}}
\def\CY{{\cal Y}}
\def\CZ{{\cal Z}}
%

\def\V{\mathbb{V}}
\def\E{\mathbb{E}}
\def\R{\mathbb{R}}
\def\C{\mathbb{C}}
\def\H{\mathbb{H}}
\def\Z{\mathbb{Z}}
\def\A{\mathbb{A}}
\def\T{\mathbb{T}}
\def\L{\mathbb{L}}
\def\D{\mathbb{D}}
\def\Q{\mathbb{Q}}


\def\mJ{\mathfrak{J}}
\def\mq{\mathfrak{q}}
\def\mQ{\mathfrak{Q}}
\def\mP{\mathfrak{P}}
\def\mp{\mathfrak{p}}
\def\mH{\mathfrak{H}}
\def\mh{\mathfrak{h}}
\def\ma{\mathfrak{a}}
\def\mA{\mathfrak{A}}
\def\mC{\mathfrak{C}}
\def\mc{\mathfrak{c}}
\def\ms{\mathfrak{s}}
\def\mS{\mathfrak{S}}
\def\mm{\mathfrak{m}}
\def\mM{\mathfrak{M}}
\def\mn{\mathfrak{n}}
\def\mN{\mathfrak{N}}
\def\mt{\mathfrak{t}}
\def\ml{\mathfrak{l}}
\def\mT{\mathfrak{T}}
\def\mL{\mathfrak{L}}
\def\mo{\mathfrak{o}}
\def\mg{\mathfrak{g}}
\def\mG{\mathfrak{G}}
\def\mf{\mathfrak{f}}
\def\mF{\mathfrak{F}}
\def\md{\mathfrak{d}}
\def\mD{\mathfrak{D}}
\def\mO{\mathfrak{O}}
\def\mk{\mathfrak{k}}
\def\mK{\mathfrak{K}}
\def\mR{\mathfrak{R}}
\def\sA{\mathscr{A}}
\def\sB{\mathscr{B}}
\def\sC{\mathscr{C}}
\def\sD{\mathscr{D}}
\def\sE{\mathscr{E}}
\def\sF{\mathscr{F}}
\def\sG{\mathscr{G}}
\def\sL{\mathscr{L}}
\def\sM{\mathscr{M}}
\def\sN{\mathscr{N}}
\def\sO{\mathscr{O}}
\def\sP{\mathscr{P}}
\def\sQ{\mathscr{Q}}
\def\sR{\mathscr{R}}
\def\sS{\mathscr{S}}
\def\sT{\mathscr{T}}
\def\sU{\mathscr{U}}
\def\sV{\mathscr{V}}
\def\sW{\mathscr{W}}
\def\sX{\mathscr{X}}
\def\sY{\mathscr{Y}}
\def\sY{\mathscr{Z}}



\def\rd{\partial}

\def\ppr{{\color{blue}\prime\prime\color{black}}}

\def\Ker{\hbox{Ker}\;}
\def\Im{\hbox{Im}\;}


\def\mod{\hbox{ }mod\hbox{ }}
\def\de{\mathfrak{Def}_{(\cA,\underline \ell^K)}}
\def\TM{T\!\!\cM}
\def\CTM{T^*\!\!\cM}

\def\wt{\hbox{\it wt}}
\def\ch{\hbox{\it ch}}
\def\gh{\hbox{\it gh}}


\catcode`\@=11 
\global\newcount\nsecno \global\nsecno=0
\global\newcount\meqno \global\meqno=1
\def\newsec#1{\global\advance\nsecno by1
\eqnres@t
\section{#1}}
\def\eqnres@t{\xdef\nsecsym{\the\nsecno.}\global\meqno=1}
\def\sequentialequations{\def\eqnres@t{\bigbreak}}\xdef\nsecsym{}

\def\draftmode{\message{ DRAFTMODE }

{\count255=\time\divide\count255 by 60 \xdef\hourmin{\number\count255}
\multiply\count255 by-60\advance\count255 by\time
\xdef\hourmin{\hourmin:\ifnum\count255<10 0\fi\the\count255}}}
\def\nolabels{\def\wrlabeL##1{}\def\eqlabeL##1{}\def\reflabeL##1{}}
\def\writelabels{\def\wrlabeL##1{\leavevmode\vadjust{\rlap{\smash%
{\line{{\escapechar=` \hfill\rlap{\tt\hskip.03in\string##1}}}}}}}%
\def\eqlabeL##1{{\escapechar-1\rlap{\tt\hskip.05in\string##1}}}%
\def\reflabeL##1{\noexpand\llap{\noexpand\sevenrm\string\string\string##1}
}}

\nolabels

\def\eqn#1#2{
\xdef #1{(\nsecsym\the\meqno)}
\global\advance\meqno by1
$$#2\eqno#1\eqlabeL#1
$$}

\def\eqalign#1{\null\,\vcenter{\openup\jot\m@th
\ialign{\strut\hfil$\displaystyle{##}$&$\displaystyle{{}##}$\hfil
\crcr#1\crcr}}\,}

\def\ket#1{\left|\bos{ #1}\right>}\vspace{.2in}
   \def\bra#1{\left<\bos{ #1}\right|}
\def\oket#1{\left.\bos{ #1}\right>}
\def\obra#1{\left<\bos{ #1}\right.}
\def\epv#1#2#3{\left<\bos{#1}\left|\bos{#2}\right|\bos{#3}\right>}
\def\qbvk#1#2{\bos{\left(\bos{#1},\bos{#2}\right)}}
\def\Hoch{{\tt Hoch}}
\def\rrd{\up{\rightarrow}{\rd}}
\def\lrd{\up{\leftarrow}{\rd}}
   \nc{\hr}{[\![\hbar]\!]}
\def\foot#1{\footnote{#1}}

\catcode`\@=12 

\def\fr#1#2{{\textstyle{#1\over#2}}}
\def\Fr#1#2{{#1\over#2}}
\nc{\bt}{\mathbf{t}}

\draftmode

\def\a{\alpha}
\def\b{\beta}
\def\c{\chi}
\def\d{\delta}  \def\D{\Delta}
\def\e{\varepsilon} \def\ep{\epsilon}
\def\f{\phi}  \def\F{\Phi}
\def\g{\gamma}  \def\G{\Gamma}
\def\k{\kappa}
\def\l{\lambda}  \def\La{\Lambda}
\def\m{\mu}
\def\n{\nu}
\def\r{\rho}
\def\vr{\varrho}
\def\o{\omega}  \def\O{\Omega}
\def\p{\psi}  \def\P{\Psi}
\def\s{\sigma}  \def\S{\Sigma}
\def\th{\theta}  \def\vt{\vartheta}
\def\t{\tau}
\def\w{\varphi}
\def\x{\xi}
\def\z{\zeta}
\def\U{\Upsilon}

%
%
%
%
%
%
%

\title{Enhanced Homotopy Theory for Period Integrals of Smooth Projective Hypersurfaces}

\author{Jae-Suk Park\inst{1,2}
\thanks{
\email{jaesuk@ibs.re.kr}.
The work of Jae-Suk Park was supported by the IBS (CA1305-01).
 }
\and Jeehoon Park\inst{2}
\thanks{\email{jeehoonpark@postech.ac.kr}.
The work of Jeehoon Park was partially supported by Basic Science Research Program through the National Research Foundation of Korea(NRF) funded by the Ministry of Education, Science and Technology(2013023108) and 
 was supported by Basic Science Research Program through the National Research Foundation of Korea(NRF) funded by the Ministry of Education(2013053914).
}
}

\institute{
Center for Geometry and Physics, Institute for Basic Science (IBS), 77 Cheongam-ro, Nam-gu, Pohang, Gyeongbuk, Korea 790-784
 \and 
 Department of Mathematics, Pohang University of Science and Technology (POSTECH), 77 Cheongam-ro, Nam-gu, Pohang, Gyeongbuk, Korea 790-784
  }

\date{}


\maketitle

\begin{abstract}
The goal of this paper is to reveal hidden structures on the singular cohomology 
and the Griffiths period integral of a smooth projective hypersurface
in terms of 
BV(Batalin-Vilkovisky) algebras and homotopy Lie theory (so called, $L_\infty$-homotopy theory).

Let $X_G$ be a smooth projective hypersurface in the complex projective space $\BP^n$ defined by 
a homogeneous polynomial $G(\ud x)$ of degree $d \geq 1$.
Let $\H=H^{n-1}_{\pr}(X_G, \C)$ be the middle dimensional primitive cohomology of $X_G$.
We explicitly construct a BV algebra $\bvx=(\cA_X,Q_X, K_X)$ such that its $0$-th cohomology
$H^0_{K_X}(\cA_X)$ is canonically isomorphic to $\H$.
We also equip $\bvx$ with a decreasing filtration and a bilinear pairing which
 realize the Hodge filtration and the cup product polarization on $\H$ under the canonical isomorphism.
Moreover, we lift $C_{[\g]}:\H \to \bC$ to a cochain map $\cC_\g:(\cA_X, K_X)  \to (\bC,0)$, where 
$C_{[\g]}$ is the Griffiths period integral given by $\omega \mapsto \int_\g \omega$ for $[\g]\in H_{n-1}(X_G,\bZ)$.

We use this enhanced homotopy structure on $\H$ to study
an extended formal deformation of $X_G$ and the correlation of its period integrals.
If $X_G$ is in a formal family of Calabi-Yau hypersurfaces $X_{G_{\ud T}}$, 
we provide an explicit formula and algorithm (based on a Gr\"obner basis) 
to compute the period matrix of $X_{G_{\ud T}}$ in
terms of the period matrix of $X_G$ and an $L_\infty$-morphism $\ud \kappa$ which enhances $C_{[\g]}$ and governs deformations of period matrices.

\end{abstract}

\tableofcontents



\newsec{Introduction}

The purpose of this paper is to attempt to establish certain  correspondences between the Griffiths period integrals of
smooth algebraic varieties and period integrals, defined in section \ref{section2}, attached to representations
of a finite dimensional Lie algebra  on a polynomial algebra.  
Such correspondences allow us to reveal  hidden BV(Batalin-Vilkovisky) algebra structures and $L_\infty$-homotopy structures in period integrals, leading
to their higher generalization. We work out the correspondence in detail when the variety
is a smooth projective hypersurface and
when the representation is the Schr\"odinger representation of the Heisenberg Lie algebra twisted by 
Dwork's polynomial associated with the hypersurface. The period integrals of this kind of example
 have been studied extensively by Griffiths in \cite{Gr69}. 
 We will enhance the Griffiths period integral into an $L_\infty$-morphism $\ud \kappa$ which governs  correlations and deformations of period matrices.

\subsection{Main theorems} \label{int1.1}

Let $n$ be a positive integer. Let $X_G$ be a smooth hypersurface in the complex projective $n$-space $\BP^n$ defined by a homogeneous polynomial $G(\underline x)=G(x_0, \cdots, x_n)$ of degree $d$ in $\bC[x_0, \cdots, x_n]$. 
Let $H_{n-1}(X_G, \bZ)_0$ be the subgroup of the singular homology group $H_{n-1}(X_G, \bZ)$ of $X_G$ of degree $n-1$, which consists of vanishing $(n-1)$-cycles, and let
$H_{\pr}^{n-1}(X_G,\bC)$ be the primitive part  of the middle dimensional cohomology group 
$H^{n-1}(X_G,\bC)$, i.e. $H_{n-1}(X_G, \bZ)_0=\ker (H_{n-1}(X_G,\bZ) \to H_{n-1}(\BP^n,\bZ))$, and $H_{\pr}^{n-1}(X_G,\bC)$ $=\coker(H^{n-1}(\BP^n,\bC) \to H^{n-1}(X_G,\bC))$. Then we are interested in the following period integrals
\eqn\ytr{
C_{[\g]}: H_{\pr}^{n-1}(X_G,\bC) \longrightarrow \C,\qquad [\varpi]\mapsto  \int_\g \varpi,
}
where $\g$ and $\varpi$ are representatives of the homology class $[\g] \in H_{n-1}(X_G, \bZ)_0$
and the cohomology class $[\varpi]\in H_{\pr}^{n-1}(X_G,\bC)$, respectively.
We shall often use the shorthand notation $\H=H^{n-1}_{\pr}(X_G,\C)$ from now on.
We also use the notation that $$
\cF_{j}\H=
H_{\pr}^{n-1,0}(X_G,\bC)\oplus H_{\pr}^{n-2,1}(X_G,\bC)\oplus\cdots\oplus H_{\pr}^{n-1-j,j}(X_G,\bC),
\quad 0\leq j\leq n-1,
$$
where $\H^{p,q}=H_{\pr}^{p,q}(X_G, \bC)$ is the $(p,q)$-th Hodge component of $\H$.

Let $\cH(X_G)$ be the rational de Rham cohomology group defined as the quotient of the group of rational $n$-forms on $\BP^n$ regular outside $X_G$ by the group of the forms $d \psi$ where $\psi$ is a rational $n-1$ form regular outside $X_G$.
For each $k\geq 1$, let $\cH_k(X_G) \subset \cH(X_G)$ be the cohomology group defined as the quotient of the group of rational $n$-forms on $\BP^n$
with a pole of order $\leq k$ along $X_G$ by the group of exact rational $n$-forms on $\BP^n$ with a pole
of order $\leq k$ along $X_G$. Griffiths showed that 
any rational $n$-form on $\BP^n$
with a pole of order $\leq k$ along $X_G$ can be written as a rational  differential $n$-form
${ F(\ud x) \Omega_n\over G(\ud x)^k}$, where  $\Omega_n=\sum_{j}(-1)^j x_j d\!x_0\wedge\cdots\wedge\widehat{d\!x_j}\wedge \cdots \wedge d\!x_n$ and $F(\ud x)$ is a  homogeneous polynomial of degree $kd -(n+1)$. He also showed that $\cH_n(X_G)
= \cH(X_G)$ and there is a natural injection $\cH_k(X_G) \subset
\cH_{k+1}(X_G)$ for each $k\geq 1$.
Moreover, Griffiths defined the isomorphism (the residue map) 
$$Res: \cH(X_G) \rightarrow \H$$ by
$$
\frac{1}{2\pi i} \int_{\tau(\g)} { F(\ud x)\over G(\ud x)^k} \Omega_n=\int_\g Res \left( { F(\ud x) \over G(\ud x)^k}\Omega_n \right),
$$
where $\tau(\g)$ is the tube over $\g$, as in (3.4) of \cite{Gr69},
such that $Res$ takes the pole order filtration
\eqn\pof{
\cH_1(X_G)\subset \cH_2(X_G)\subset \cdots \subset \cH_{n-1}(X_G)\subset  \cH_{n}(X_G)=\cH_{n+1}(X_G)=\cdots =\cH(X_G)
}
of $\cH(X_G)$\big($=\cH_{n+\ell}(X_G)$ for all $\ell\geq 0$\big) onto
the increasing Hodge filtration $\cF_{\bullet}\H$
\eqn\pog{
 \cF_0\H\subset\cF_1\H\subset \cdots \subset \cF_{n-1}\H=\H
}
 of the primitive middle dimensional cohomology $\H=H_{\pr}^{n-1}(X_G,\bC)$.
The Griffiths theory provides us with an effective method of studying the period integrals $C_{[\g]}$ 
on the hypersurface $X_G$ as well as an infinitesimal family of hypersurfaces.

%
%

In this article, we provide a new homotopy theoretic framework to understand $\cH(X_G)\simeq \H$ and such period integrals. 
Our main result is to construct a BV algebra $\bvx$ whose $0$-th cohomology is canonically isomorphic to
$\H$ and lift the polarized Hodge structure on $\H$ to $\bvx$. Moreover we enhance the period integral $C_{[\g]}:\H\rightarrow \C$ to a cochain map $\cC_\g:
\bvx \to (\C,0)$.
\begin{definition} \label{bvd}
 A BV (Batalin-Vilkovisky) algebra\footnote{We normalize so that the Lie bracket and the differential have degree 1, and the binary product has degree 0.} over $\Bbbk$ is a cochain complex $(\cA, K=Q+\Delta)$ with the following properties:
 
 (a) $\cA=\bigoplus_{{m\in \Z}}\cA^m$ is a unital $\Z$-graded super-commutative and associative $\Bbbk$-algebra satisfying $K(1_\cA)=0$.
 
 (b) $Q^2=\Delta^2=Q\Delta+\Delta Q=0$, and $(\cA, \cdot, Q)$ is a commutative differential graded algebra:
 $$
 Q(a\cdot b)=Q(a) \cdot b +(-1)^{|a|} a \cdot Q(b), 
 $$
 for any homogeneous elements $a, b \in \cA$.
 
 (c) $\Delta$ is a differential operator\footnote{This means that $\ell_3^{\Delta}=\ell_4^{\Delta}=\cdots=0$ in our terminology of \textit{the descendant functor.} We will explain the notion of the descendant functor later in subsection \ref{subs3.1}.} of order 2 
  and $(\cA, K, \ell_2^K)$ is a differential graded Lie algebra (DGLA),
 where $\ell_2^K(a,b) := K(a\cdot b)-K(a) \cdot b -(-1)^{|a|} a \cdot K(b)$: 
 \be
 \ell^K_2(a,b) - (-1)^{|a| |b|} \ell^K_2(b,a)&=&0, \\
 \ell^K_2(a, \ell^K_2(b,c)) + (-1)^{|a|} \ell^K_2(\ell^K_2(a,b),c) 
 +(-1)^{(|a|+1)|b|} \ell^K_2(b, \ell^K_2(a,c))&=&0,\\
K\ell^K_2(a,b) + \ell^K_2(Ka, b) + (-1)^{|a|}\ell^K_2(a, Kb) &=& 0,
\ee
 for any homogeneous elements $a, b \in \cA$.
 
(d) $(\cA, \cdot, \ell_2^K)$ is a {Gerstenhaber algebra}\footnote{In fact, this condition $(d)$ follows from (c).}: 
$$
 \quad \ell_2^{K}(a \cdot b, c)= (-1)^{|a|} a \cdot \ell_2^{K}(b, c) +(-1)^{|b|\cdot |c|} \cdot \ell_2^{K}(a,c)\cdot b, \quad a, b, c \in \cA.
$$
\end{definition}

In Section \ref{section4}, we explicitly construct a BV algebra $\bvx:= (\cA_X, \cdot, Q_X, K_X)$ associated to $X_G$. 
In the introduction, we briefly summarize the construction and some of its features. 
Let
 \eqn\qfta{
\cA_X:=\C[y, x_0, \cdots, x_n][\eta_{-1}, \eta_0, \cdots, \eta_n],
}
be the $\bZ$-graded super-commutative polynomial algebra, where $y, x_0, \cdots, x_n$ are formal variables of degree $0$ and $\eta_{-1}, \cdots, \eta_n$ are formal variables of degree $-1$. 
We define three additive gradings on $\cA_X$
with respect to the multiplication, called {\it ghost number} $\gh \in \Z$, {\it charge} ${\ch} \in \Z$ and {\it weight} $\wt \in \Z$, by the following rules:
$$
\eqalign{
&\gh(y_{})=0,  \ \ \  \gh(x_j)=0, \quad   \gh(\eta_{-1})=-1, \ \gh(\eta_j)=-1, 
\cr 
&\ch(y_{})=-d, \ch(x_j)=1, \quad \ch(\eta_{-1})=d, \ \ \ \ch(\eta_j)=-1,
\cr 
&\wt(y_{})=1,\ \ \  \wt(x_j)=0, \quad \wt(\eta_{-1})=0,  \ \ \ \wt(\eta_j)=1, \quad
}
$$
where $ j=0, \cdots, n.$ The ghost number is same as the (cohomology) degree.
Write such a decomposition as follows:
\eqn\depa{
\cA_X=\bigoplus_{\gh,{\ch}, \wt} \cA_{X,{\ch}, (\wt)}^{gh}= \bigoplus_{-n-2\leq j \leq 0}\bigoplus_{w \in \Z^{\geq 0}}\bigoplus_{{\l} \in \Z}\cA^j_{X,{\l},(w)}.
}
We define a differential $K_X$ (of degree 1)
\eqn\qftb{
K_X:=\left(G(\ud x) + \pa{y}\right)\pa{\eta_{-1}} + \sum_{i=0}^n \left( y \prt{G(\underline x)}{x_i} + \pa{x_i}\right) \pa{\eta_i} 
 }
and let
$\Delta:=\pa{y}\pa{\eta_{-1}}+ \sum_{i=0}^n \pa{x_i}\pa{\eta_i}$ and $Q_X:=K_X-\Delta$.
The $\wt$ grading turns out to give a (decreasing) filtered subcomplex $(F^\bullet \cA_X, K_X)$ of $(\cA_X, K_X)$ defined by 
$$
F^0 \cA_X = \cA_X, \quad  F^i \cA^{}_X =\bigoplus_{k \leq n-1-i} \cA^{}_{X,(k)}, \quad i \geq 1. 
$$
%
Let $\pi_0$ be the projection map from $\cA_X$ to $\cA_{X}^0$.
We define two cochain maps $\cC_\g:(\cA_X, K_X) \to (\C,0)$ and $\oint:(\cA_X, Q_X) \to (\C,0)$ as follows:
\eqn\fpint{
\cC_\g (u ) := -\frac{1}{2 \pi i} \int_{\tau(\g)}\bigg( \int_{0}^{\infty} \pi_0(u) \cdot e^{y G(\underline x)}  dy\bigg) \Omega_n,
}
for homogeneous elements $u \in \cA_X$ of charge $d-(n+1)$, $\cC_\g(u):=0$ for other homogeneous elements with respect to charge, where $[\g] \in H_{n-1}(X_G, \bZ)_0$ and $\tau: H_{n-1}(X_G,\bZ) \to H_n(\BP^n-X_G,\bZ)$ is the tubular neighborhood map (see (3.4) in \cite{Gr69}),\footnote {We will see that $\pi_0(u)$ is homogeneous of charge $d-(n+1)$ if and only if the integral $\bigg( \int_{0}^{\infty} \pi_0(u) \cdot e^{y G(\underline x)}  dy\bigg) \Omega_n$ defines a differential $n$-form on $\BP^n-X_G$.} and 
$$
\oint {u} :=
\frac{1}{(2\pi i)^{n+2}}
\int_{X(\e)}
\left(\oint_C \frac{\pi_0(u)}{  y^{n} }dy\right)
\frac{dx_0 \wedge \cdots \wedge dx_n}{  \frac{\rd G}{\rd x_0}\cdots\frac{\rd G}{\rd x_n} }, \quad u \in \cA_X,
$$
where
$C$ is a closed path on $\C$ with the standard orientation around $y=0$ 
and
$$
X(\e)=\left\{ \underline{x} \in \C^{n+1}\left| \Big|\frac{\rd G(\ud x)}{\rd x_i}\Big|=\e > 0, i=0,1,\cdots, n\right.\right\}.
$$
Then $\oint$ 
defines a symmetric bilinear pairing $\langle \cdot ,  \cdot \rangle_X$ on $\cA_X$ by
\eqn\liftingofp{
\langle u, v \rangle_X := \oint \pi_0 (u \cdot v), \quad u, v \in \cA_X.
}


 \begin{theorem} \label{firsttheorem}
The triple $\bvx:= (\cA_X, \cdot, Q_X, K_X)$ becomes a BV algebra over $\C$ with following properties:

(a) We have a decomposition $\cA_X =\bigoplus_{-(n+2) \leq m \leq 0} \cA_X^m$ and there is a canonical isomorphism 
$$
J: H_{K_X}^0 (\cA) \mapto{\sim} \H
$$
where $H_{K_X}^0 (\cA)$ is the $0$-th cohomology module.

(b) $Q_X(f)=\ell_2^{K_X}(yG(\ud x), f)$ for any $f \in \cA_X$.\foot{It is easy to see that the homogeneous coordinate ring of $X_G$ is isomorphic to the weight zero part of the cohomology ring $H_{Q_X}^0(\cA_X)_{(0)}.$
Since $(\cA_X, K_X)$ can be viewed as a quantization of $(\cA_X, Q_X)$, the singular cohomology $\H$ may be regarded as a quantization of the homogenous coordinate ring of $X_G$.}

%

Moreover, the followings hold: 

(c) The map $J$ sends the decreasing filtration $(F^\bullet \cA_X, K_X)$ on $(\cA_X, K_X)$ to the Hodge filtration on $\H$.

(d) The pairing $\langle \cdot, \cdot \rangle_X$ on $\cA_X$ induces a polarization (the cup product pairing) on $\H$ (up to sign)
under $J$.

(e) The cochain map $\cC_\g:(\cA_X, K_X) \to (\C,0)$ induces the period integral $C_{[\g]}$ under $J$. 
%
%
\end{theorem}
The novel feature here is that we are able to put an associative and super-commutative binary product $\cdot$ on the cochain complex $(\cA_X, K_X)$ which turns out to govern \textit{correlations} and \textit{deformations} of the period integral $C_{[\g]}:\H \to \C$. 
\footnote{Since $\cH(X_G)$ is defined as the cohomology of the de Rham complex with the wedge product, one might think to play a similar game to find hidden correlations.
But if one wedges two $n$-forms then the resulting differential form is a $2n$-form which can \textbf{not} be integrated against a fixed cycle $\g$.} 
In addition, we simultaneously realize the Hodge theoretic information on $\H$ (the Hodge filtration and the cup product polarization)
at the BV algebra level.

A consequence of the BV structure on $\bH$ is that the period 
integral $C_{[\g]}:\H\rightarrow \C$ can be enhanced to an $L_\infty$-morphism
$\ud {\kappa} ={\kappa}_1,{\kappa}_2,\cdots$, where ${\kappa}_1=C_{[\g]}$ and ${\kappa}_m$ is a linear map from the $m$-th symmetric
power $S^m\H$ of $\H$ into $\C$, which is a composition of two non-trivial $L_\infty$-morphisms
such that $\ud \kappa$ depends only on the $L_\infty$-homotopy types of each factor.
We recall that an $L_\infty$-algebra, or homotopy Lie algebra, 
$(V, \ud \ell)$ is a $\Z$-graded vector space $V$ with an $L_\infty$-structure 
$\ud \ell =\ell_1,\ell_2,\ell_3,\cdots$, where $\ell_1$ is a differential such that $(V,\ell_1)$ is a cochain complex, $\ell_2$ is a graded Lie bracket which satisfies the graded Jacobi identity up to homotopy $\ell_3$ etc. An $L_\infty$-morphism $\ud \phi=\phi_1,\phi_2,\cdots$ is a morphism between $L_\infty$-algebras, say $(V,\ud \ell)$
and $(V^\prime, \ud \ell^\prime)$, such that $\phi_1$ is a cochain map of the underlying cochain complex, 
which is a Lie algebra homomorphism
up to homotopy $\phi_2$, etc. An $L_\infty$-homotopy $\ud \l=\l_1,\l_2,\cdots$ is a homotopy of $L_\infty$-morphisms 
such that $\l_1$ is a cochain homotopy of the underlying cochain complex, etc.\footnote{
Any $\Z$-graded vector space may be regarded as an $L_\infty$-algebra with zero $L_\infty$-structure.
The cohomology of an $L_\infty$-algebra is defined to be the cohomology of the underlying cochain complex.
 An $L_\infty$-quasi-isomorphism is a $L_\infty$-morphism $\ud \phi=\phi_1,\phi_2,\cdots$ such that $\phi_1$ is a cochain quasi-isomorphism. 
See Appendix \ref{subs6.2} for the definitions of $L_\infty$-algebras, $L_\infty$-morphisms, and $L_\infty$-homotopies as well as 
the category and the homotopy category of $L_\infty$-algebras.
}
We use this hidden $L_\infty$-homotopy theoretic structure  to study certain extended deformations and correlations of period integrals;
we develop a new formal deformation theory of $X_G$
which leaves the realm of infinitesimal variations of Hodge structures of $X_G$. 
This new formal deformation theory has directions that do {\it not} satisfy Griffiths transversality.

\bet\label{thirdtheorem}
There is a non-trivial $L_\infty$-algebra $\left(\tilde \cA, \tilde{\ud \ell}\right)_X$ associated to $X_G$
with the following properties:

(a) The cohomology $\H:= H_{\pr}^{n-1}(X_G,\C)$, regarded as an $L_\infty$-algebra 
$(\H,\ud 0)$ concentrated in degree 0 with zero $L_\infty$-structure $\ud 0$, is quasi-isomorphic to  
the $L_\infty$-algebra $\left(\tilde \cA, \tilde{\ud \ell}\right)_X$.

(b) For each representative $\g$ of  $[\g] \in H_{n-1}(X_G, \bZ)_0$ there is an 
$L_\infty$-morphism $\ud\phi^{\!\!\cC_\g}=\phi^{\!\!\cC_\g}_1, \phi^{\!\!\cC_\g}_2,\cdots $ from 
$(\tilde \cA,\tilde{\ud \ell})_X$ into $(\C,\ud 0)$ $\mbox{--}$ the ground field $\C$ regarded as an $L_\infty$-algebra $(\C,\ud 0)$ with zero $L_\infty$-structure $\ud 0$ $\mbox{--}$ whose $L_\infty$-homotopy type $\left[\ud\phi^{\cC_\g}\right]$ is determined uniquely 
by the homology class $[\g]$ of $\g$.

(c) There is an  explicitly constructible  $L_\infty$-morphism $\ud\kappa=\kappa_1,\kappa_2,\cdots$
from $(\H,\ud 0)$ into $(\C,\ud 0)$
which is the composition $\ud \kappa := \ud \phi^{\cC_\g} \bullet \ud \varphi^\H$
of the $L_\infty$-quasi-isomorphism $\ud\varphi^\H$ 
from (a)
and the  $L_\infty$-morphism $\ud \phi^{\!\!\cC_\g}$ associated to $\g$ from (b) such that

$\quad $(i) $\ud \kappa := \ud \phi^{\cC_\g} \bullet \ud \varphi^\H$ depends only on the $L_\infty$-homotopy types
of 
$\ud \varphi^\H$ and $\ud \phi^{\cC_\g}$ and

$\quad$ (ii) $\kappa_1= C_{[\g]}= \phi^{\!\!\cC_\g}_1 \circ \varphi^\H_1$:
\eqn\Liecomp{
\xymatrixcolsep{5pc}\xymatrix{(\H, \ud 0) 
\ar@{->}@/^2pc/[rr]^{C_{[\g]}= \phi^{\cC_\g}_1 \circ \varphi^\H_1}
\ar@{..>}@/_4pc/[rr]_{\ud \kappa=\ud \phi^{\cC_\g} \bullet \ud \varphi^\H}
\ar@{..>}@/_2pc/[r]^{\ud \varphi^\H}   \ar[r]^{\varphi_1^\H} &
(\tilde \cA, \tilde{\ud \ell})_X
  \ar@{..>}@/_2pc/[r]^{\ud \phi^{\cC_\g}}  \ar[r]^{\phi_1^{\cC_\g}} & (\bC,\ud 0)}
  .
}
\eet

Note that  $\varphi^\H_1$ is a cochain quasi-isomorphism from $(\H,0)$ to $(\tilde \cA, \tilde{\ell}_1)_{X}$, $\phi^{\cC_\g}_1$ 
is a cochain map from $(\tilde \cA, \tilde{\ell}_1)_{X}$ to $(\C,0)$ and both are defined up to cochain homotopies. 
Within their own homotopy types,
a choice of $\varphi^\H_1$  corresponds to a choice of representative $\varpi$ of
the cohomology class $[\varpi]\in H^{n-1}_{\pr}(X_G, \C)$, 
while a choice of $\phi^{\cC_\g}_1$ corresponds to, after dualization,
a choice of representative $\g$ of the homology class $[\g]\in H_{n-1}(X_G, \bZ)_0$ 
in the integral $\int_\g \varpi$ in \ytr\ such
that $\kappa_1\big([\varpi]\big)=\phi^{\cC_\g}_1 \circ \varphi^\H_1\big([\varpi]\big)= \int_\g \varpi$
and $\cC_\g=\phi^{\cC_\g}_1$.

\subsection{Applications}\label{int1.2}

Note that $\bvx$
is very explicit (a super-commutative polynomial ring with derivative operators) and 
amenable to computation.
Here we explain some applications how to use the theorems in subsection \ref{int1.1} to analyze the period integral and the period matrix of $X_G$ and their deformations. More precisely, if $X_G$ is in a formal family of hypersurfaces $X_{G_{\ud T}}$, 
we provide an explicit formula and an algorithm (based on a Gr\"obner basis) 
to compute the period integral (see Theorem \ref{fourththeorem}) of $X_{G_{\ud T}}$ in
terms of the period integral of $X_G$ and the $L_\infty$-morphism $\ud \kappa=\ud \phi^{\cC_\g} \bullet \ud \varphi^\H$ in Theorem \ref{thirdtheorem}.
We also do the same work for the period matrix (see Theorem \ref{fifththeorem}) of $X_{G_{\ud T}}$.  
This explicit formula can be viewed as an explicit solution to Picard-Fuchs type differential equations for period integrals of a formal family of hypersurfaces
using the $L_\infty$-homotopy data and the initial solution. 
We assume that $X_G$ is Calabi-Yau, i.e. $d=n+1$, in Theorems \ref{fourththeorem} and \ref{fifththeorem}; our deformation theory works especially well in this case.
If $X_G$ is not Calabi-Yau, some of statements are still literally true and some can be modified to be true. But we decided to postpone the non-Calabi-Yau deformation 
in anther paper for the sake of simplicity of presentation and for technical reasons.\footnote{If $c_X:=d-(n+1)\neq 0$, then
one can find a homogeneous polynomial $g(\ud x)$ of the minimal degree and minimal $i \geq 0$ such that $\cC_\g\left( y^ig(\ud x)\right)\neq 0$ and $y^i g(\ud x) \in \cA^0_{c_X}$
and then use $\cC_\g \left(y^i g(\ud x) (e^{\G_{\ud \varphi^\H}/y^i g(\ud x)}-1)\right)$ instead of $\cC_\g(e^{\G_{\ud \varphi^\H}}-1)=\cZ_{[\g]}\left(\big[\ud \varphi^\H\big]\right)$; this amounts to twisting of "the measure $e^{y G(\ud x)} dy \Omega_n$" in \fpint\ by a new measure "$y^i g(\ud x)e^{y G(\ud x)} dy \Omega_n$" which is invariant under the classical charge symmetry $y \mapsto \l^{-d} y, x_i \mapsto \l x_i, \  \l \in \bC^\times$.}

For our deformation theory, we define the following formal power series
\eqn\mgfcy{
\cZ_{[\g]}\left(\big[\ud \varphi^\H\big]\right):=\exp\left(\sum_{n=1}^\infty \frac{1}{n!} \sum_{\alpha_1, \cdots, \alpha_n} t^{\alpha_n} \cdots t^{\alpha_1} \big(\ud \phi^{\!\!\cC_\g} \bullet \ud \varphi^\H\big)_n(e_{\alpha_1}, \cdots, e_{\alpha_n} )\right) -1  \in \C[[\ud t]],
}
which depends only on the homology class $[\g] \in H_{n-1}(X_G, \bZ)_0$ of $\g$ and
the $L_\infty$-homotopy type $\big[\ud \varphi^\H\big]$ of the $L_\infty$-quasi-isomorphism
$\xymatrix{\ud \varphi^\H:(\H,\ud 0)\ar@{..>}[r]& (\tilde \cA, \tilde{\ud \ell})_X}$.
This generating power series shall be used to determine the period integral and the period matrix of a projective hypersurface deformed from $X_G$.

Let $\{e_{\alpha}\}_{\alpha \in I }$ be a $\bC$-basis of $\H$, where $I$ is an index set, 
and denote the $\bC$-dual of $e_\alpha$ by $t^\alpha$. 
Let $X_{G_{\ud {T}}} \subset \BP^n$ be a formal family of smooth hypersurfaces defined by
$$
G_{\ud{T}} (\ud x) = G(\ud x) +F(\ud T),
$$
where $F(\ud T) \in \bC[[\ud T]] [\ud x] $ is a homogeneous polynomial of degree $d$ with 
coefficients in $\bC[[\ud T]]$ with $F(\ud 0) =0$ and $\ud T=\{T^\alpha \}_{\alpha \in I'}$ are formal variables with some index set $I' \subset I$. 

By a standard basis of $\H$ we mean a choice of basis $e_1,\cdots, e_{\delta_0}, e_{\delta_0+1},\cdots, 
 e_{\delta_1},\cdots, $ $  e_{\delta_{n-2}+1}, \cdots, 
e_{\delta_{n-1}}$ for the flag $\cF_\bullet \H$ in \pog\ such that
$e_1,\cdots, e_{\delta_0}$ gives a basis for the subspace $\H^{n-1 - 0, 0}:=H_{\pr}^{n-1,0}(X_G,\C)$
and $e_{\delta_{k-1} +1},\cdots, e_{\delta_{k}}$, $1\leq k\leq n-1$, gives a basis for the subspace 
$\H^{n-1 - k, k}=H_{\pr}^{n-1-k,k}(X_G,\C)$. We also denote such a basis by 
$\{e_\alpha\}_{\alpha \in I}$ where $I=I_{0}\sqcup I_{1}\sqcup\cdots\sqcup I_{n-1}$ with the notation
$\{e^{j}_{a}\}_{a\in I_j}= e_{\delta_{j-1} +1},\cdots, e_{\delta_{j}}$ 
and $\{t^a_{j}\}_{a \in I_{j}}= t^{\delta_{j-1} +1},\cdots, t^{\delta_{j}}$.
We assume that $I' \subset I_1$, since $\bH^{n-2,1}$ governs the deformations of complex
structures of $X$.

\begin{theorem}\label{fourththeorem}
Assume that $X$ is Calabi-Yau. For any standard basis 
$\{e_\alpha \}_{\alpha \in I}$ of $\H$ with $I=I_{0}\sqcup I_{1}\sqcup\cdots\sqcup I_{n-1}$,
 there is  an $L_\infty$-quasi-isomorphism 
$\xymatrix{\ud f:(\H,\ud 0)\ar@{..>}[r]& (\tilde \cA, \tilde{\ud \ell})_X}$ such that


(a) for each $ 0 \leq k \leq n-1$, the set $\left\{f_1\big(e^{k}_{a}\big)\right\}_{a\in I_k}$
corresponds to a set $\big\{F_{[k]a}(\ud x)\big\}_{a\in I_k}$ of homogeneous polynomials of degree
$dk=d(k+1)-(n+1)$ such that
$$
\left\{\left[Res {(-1)^k k! F_{[k]a}(\ud x)\over G(\ud x)^{k+1}}\Omega_n\right]\right\}_{a\in I_k}=
\left\{e^{k}_{a}\right\}_{a\in I_k}.
$$

(b)We have the following equation 
$$
\frac{1}{2\pi i} \int_{\tau(\g)} {\Omega_n\over G_{\ud{T}} (\ud x) }
=
\frac{1}{2 \pi i} \int_{\tau(\g)} \frac{\Omega_n}{G(\ud x)}+
\cZ_{[\g]}([\ud f])\Big|_{\substack{t^\beta=0, \beta\in I\setminus I' \\  t^\alpha=T^\alpha, \alpha \in I'}}.
$$

\end{theorem}
Note that
$\frac{1}{2 \pi i} \int_{\tau(\g)} \frac{\Omega_n}{G_{\ud{T}}(\ud x)}$ 
in $(b)$ above is a geometrically defined invariant of the formal 
family of hypersurfaces $X_{ G_{\ud{T}}}$; recall that 
the image of $\frac{\Omega_n}{G_{\ud{T}}(\ud x)}$ under the residue map represents a holomorphic $(n-1)$-form on $X_{G_{\ud {T}}}$. Thus the $L_\infty$-homotopy invariant 
$\cZ_{[\g]}([\ud f])\Big|_{\substack{t^\beta=0, \beta\in I\setminus I' \\  t^\alpha=T^\alpha, \alpha \in I'}}$
tells us how to compute the period integral of 
a deformed hypersurface $X_{ G_{\ud{T}}}$ from the period integral of $X_G$.

Now we explain how to use $L_\infty$-homotopy theory
to compute the period matrix of a deformed hypersurface. 
Let $\{\g_{\alpha} \}_{\alpha \in I}$ be a basis of $H_{n-1}(X,\bC)_0$ by noting that 
$$
\dim_\bC H_{n-1}(X,\bC)_0 = \dim_\bC \H.
$$
Let $\Omega(X)=\omega^{\alpha}_{\beta}(X)$ be the period matrix of $X$, i.e.
$$
\omega^\alpha_\beta(X) := \int_{\g_{\alpha}} e_\beta=\frac{1}{2\pi i}  \int_{\tau(\g_{\alpha})}  \frac{(-1)^k k!F_{[k]\beta}(\ud x)}{G(\ud x)^{k+1}}\Omega_n,  
\quad  \alpha, \beta \in I.
$$

Let $\Omega(X_{G_{\ud T}})=\omega^{\alpha}_{\beta}(X_{G_{\ud {T}}})$ be the period matrix of a formal hypersurface $X_{G_{\ud {T}}}$ defined
by $G_{\ud{T}} (\ud x)$.
Note that smooth projective hypersurfaces with fixed degree $d$ have same topological types and their singular homologies (consisting of vanishing cycles) and primitive cohomologies are isomorphic.

\begin{theorem}\label{fifththeorem}
Assume that $X$ is Calabi-Yau. The $L_\infty$-quasi-isomorphism $\ud f: \xymatrix{(\H,\ud 0)\ar@{..>}[r]^{}& (\tilde \cA, \tilde{\ud \ell})_X}$ in Theorem \ref{fourththeorem}
induces a 1-tensor $\{T^\alpha(\ud t)_{\ud f} \} \in \C[[\ud t]]$, which is explicitly computable and depends only on 
the $L_\infty$-homotopy type of $\ud f$,
such that 

(a) $ \cZ_{[\g]}\left(\big[\ud f\big]\right)(\ud t)=\sum_{\alpha\in I} T^\alpha(\ud t)_{\ud f} C_{[\g]}(e_{\alpha}),$

(b) $ T^{\alpha}(\ud t)_{\ud f} = t^{\alpha} +\cO(\ud t^2),\quad \forall {\alpha} \in I.$

In addition, we have the following formula between $\omega^{\alpha}_{\beta}(X_{G_{\ud {T}}})$ and $\omega_{\alpha}^{\beta}(X_G)$ via $\{T^{\alpha}(\ud t)_{\ud f}\}$;
$$
\eqalign{
\omega^{\alpha}_\beta (X_{G_{\ud {T}}}) &= 
\frac{\rd}{\rd t^\beta} \left( \cZ_{[\g_{\alpha}]}\left(\big[\ud f\big]\right)(\ud t)\right) 
\Big|_{\substack{t^\beta=0, \beta\in I\setminus I' \\  t^{\alpha}=T^{\alpha}, {\alpha} \in I'}}
\cr
&=\sum_{\rho\in I} \left( \frac{\rd}{\rd t^\beta} T^{\rho}(\ud t)_{\ud f}\right) \omega^{\alpha}_\rho(X_G)
\Big|_{\substack{t^\beta=0, \beta\in I\setminus I' \\  t^{\alpha}=T^{\alpha}, {\alpha} \in I'}},
}
$$
for each ${\alpha}, \beta \in I$.
\end{theorem}

This theorem says that $\Omega(X)$ and $\Omega(X_{G_{\ud T}})$ are ``transcendental'' invariants but their relationship is ``algebraically computable up to desired precision''; if we know the period matrix $\Omega(X)=\Omega(X_{G_{\ud 0}})$ and the polynomials $G_{\ud T}(\ud x)$, then there is an algebraic algorithm to compute the period matrix $\Omega(X_{G_{\ud T}})$.
%
The matrix 
$\left(\frac{\rd}{\rd t^\beta} T^{\rho}(\ud t)_{\ud f} \right)\Big|_{\substack{t^{\beta}=0, {\beta}\in I\setminus I' \\  t^{\alpha}=T^{\alpha}, {\alpha} \in I'}}$
gives a linear transformation formula from $\omega^{\alpha}_{\beta}(X_G) $ to $\omega^{\alpha}_{\beta}(X_{G_{\ud {T}}})$. Note that 
$$
\left(\frac{\rd}{\rd t^{\beta}} T^{\rho}(\ud t)_{\ud f} \right)\Big|_{t^a=0, a \in I} =\delta^\rho_{\beta},
\quad
\frac{\rd}{\rd t^{\beta}} \left( \cZ_{[\g_{\alpha}]}\left(\big[\ud f\big]\right)(\ud t)\right)\Big|_{t^a=0, a \in I}= \omega^{\alpha}_{\beta}(X_G),
$$
 where $\delta^\rho_{\beta}$ is the Kronecker delta. Thus the matrix
$\frac{\rd}{\rd t^{\beta}} \left( \cZ_{[\g_{\alpha}]}\left(\big[\ud f\big]\right)(\ud t)\right)$
can be thought of as a generalization of the period matrix of a formal deformation of $X_G$.
We will sometimes call  the matrix 
$$
\frac{\rd}{\rd t^{\beta}} \left( \cZ_{[\g_{\alpha}]}\left(\big[\ud f\big]\right)(\ud t)\right)_{\{{\alpha}, {\beta} \in I \}}
$$
the{ \em{period matrix of an extended formal deformation (or an extended period matrix) of $X_G$}}, associated to the $L_\infty$-quasi-isomorphism $\ud f$.

From property $(a)$ above, we see that the $1$-tensor $\{T^{\alpha}(\ud t)_{\ud f} \}$ determines
the generating series $\cZ_{[\g]}\left(\big[\ud f\big]\right)(\ud t)$ completely if it is combined
with the period integral $C_{[\g]}:\H \rightarrow \C$.
Property $(b)$  in {Theorem \ref{fifththeorem}} also allows us to  define an invertible matrix ($2$-tensor) $\CG_{\alpha}{}^{\beta}(\ud t)_{\ud f}  
:=\rd_{\alpha} T^{\beta}(\ud t)_{\ud f} =\delta_{\alpha}{}^{\beta} + \cO(\ud t)$,
where $\partial_{\alpha}$ means the partial derivative with respect to $t^{\alpha}$, 
with inverse $\CG^{-1}_{\alpha}{}^{\beta}(\ud t)_{\ud f} $ in $\C[[\ud t]]$.
We can further define a 3-tensor $\{A_{{\alpha}{\beta}}{}^\s(\ud t)_{\ud f} \} \in \C[[\ud t]]$
depending only on the $L_\infty$-homotopy type of $\ud f$ such that
for all ${\alpha}, {\beta},\s \in I$,
\eqn\onediff{
A_{{\alpha}{\beta}}{}^\s(\ud t)_{\ud f}  
= \sum_{\rho \in I}\left(\rd_{\alpha} \cG_{\beta}{}^\rho(\ud t)_{\ud f} \right) \cG^{-1}_{\rho}{}^\s(\ud t)_{\ud f} .
}

Our approach shall provide an effective algorithm (using a Gr\"obner basis)
for computing the 3-tensor $A_{{\alpha}{\beta}}{}^\g(\ud t)_{\ud f}$; see subsection \ref{subs4.5}.
Also there is a concrete algorithm to determine the $1$-tensor $\{T^{\alpha}(\ud t)_{\ud f}\}$ from 
the 3-tensor $\{A_{{\alpha}{\beta}}{}^\g(\ud t)_{\ud f}\}$ (its calculation can be implemented in a computer algebra system
such as {\sc Singular} by our new homotopy method).
This means that, for an arbitrary homogeneous polynomial $F(\ud x)$ of degree $kd-(n+1), k \geq 1$, there is an effective algorithm to compute the period integral
$
\int_\g Res \left({ F(\ud x) \over G(\ud x)^k}\Omega_n \right)
$
from the finite data $\{ C_{[\g]}(e_{\alpha})$ : ${\alpha} \in I \}$. Specializing the generating series $\cZ_{[\g]}\left(\big[\ud f\big]\right)(\ud t)$ for any $1$-parameter family, by setting $t^{\alpha} = 0$ for all ${\alpha} \in I$ except for one parameter $t^{\beta}$ with ${\beta} \in I'$, one can derive an ordinary differential equation of higher order, which turns out to be the usual Picard-Fuchs equation.
In fact, $A_{{\alpha}{\beta}}{}^\s(\ud t)_{\ud f}$ can be regarded as a generalization of the Gauss-Manin connection;
see subsection  \ref{subs4.62}.\\

As another application of our new approach to understanding $\H$ and $C_{[\g]}$, we were able to prove the existence of a cochain level realization of the Hodge filtration 
and a polarization on $\H$ (statements $(c)$ and $(d)$ of Theorem \ref{firsttheorem}). 
We put a \textit{weight filtration} on $(\cA_X, \cdot, K_X)$ which induces the Hodge filtration on $\H$ under the isomorphism $H_{K_X}^0(\cA_X) \simeq \H$ and  analyze this filtered complex to obtain a certain spectral sequence, which we call \textit{the classical to quantum spectral sequence}; see Propositions \ref{wspec} and \ref{homotopyhf}. 
Moreover, we lift a polarization of $\H$ to a bilinear paring on $\cA_X$; see Definition \ref{polgo} and Theorem \ref{polgood}.
This might provide a new optic for understanding the period domain of \textit{homotopy polarized Hodge structures} and 
the (infinitesimal) variation of polarized Hodge structures at the level of cochains. 

\subsection{Approach to the proofs of the theorems} \label{int1.3}
We turn to the general framework behind the theorems. In this subsection, we will briefly
indicate how we approach their proofs.
We associate to $X_G$ a representation of a finite dimensional abelian Lie algebra $\frg$ of dimension
$n+2$ on a polynomial algebra with $n+2$ variables;
$$
\rho_{X}: \frg \to \End_{\Bbbk}(A), \quad A:= \Bbbk[y, x_0, \cdots, x_n]=\Bbbk[y, \underline x].
$$

This Lie algebra representation 
comes from the Schr\"{o}dinger representation of the abelian Lie subalgebra of the Heisenberg Lie algebra
of dimension $2(n+2)+1$ twisted
by the Dwork  polynomial $y \cdot G(\ud x) \in \Bbbk[y, x_0,\cdots, x_n]$.
Let $\alpha_{-1}, \alpha_0, \cdots, \alpha_{n}$ be a $\Bbbk$-basis of $\frg$.
 We also introduce variables $y_{-1}=y, y_0=x_0, y_1=x_1, \cdots, y_n=x_n$ for notational convenience.
 If we consider the formal operators (twisting $\rho$ by $y \cdot G(\ud x)$),
\be
\rho_X(\alpha_i):=\exp(-y \cdot G(\ud x))\cdot \pa{y_{i}} \cdot \exp(y \cdot G(\ud x)) , \quad i=-1, 0, \cdots, n,
\ee
then we can see that
$$
\rho_X(\alpha_i)= \pa{y_i} + \Big[\pa{y_i}, y\cdot G(\ud x) \Big] + \frac{1}{2}\Big[\Big[\pa{y_i}, y \cdot G(\ud x)\Big],
y \cdot G(\ud x) \Big]+ \cdots
\equiv  \prt{\left(y \cdot G(\ud x)\right)}{y_i} +\pa{y_i}.
$$

With the notion of {\it period integrals} of Lie algebra representations, we will show 
that the period integrals of such representations are the Griffiths period integrals of the hypersurface $X_G$.
If we use the dual Chevalley-Eilenberg cochain complex $(\cA_{\rho_X}, \cdot, K_{\rho_X})$, which
computes the Lie algebra \textit{homology} associated to $\rho_X$, then we can realize
$C_{[\g]}$ as the homotopy type of a cochain map $\cC_\g: (\cA_{\rho_X}, \cdot, K_{\rho_X}) \to (\bC, \cdot, 0)$ of 
cochain complexes equipped with a super-commutative product.
In fact, $(\cA_X,K_X)$ in Theorem \ref{firsttheorem} is $(\cA_{\rho_X},K_{\rho_X})$.

This leads us to study 
 the category $\frC_\Bbbk$ of cochain complexes over a field $\Bbbk$ equipped with a super-commutative product.
An object of $\frC_\Bbbk$ is a unital $\bZ$-graded associative and super-commutative $\Bbbk$-algebra $\cA$ with differential $K$, denoted $(\cA, \cdot, K)$. 
A morphism in $\frC_\Bbbk$ is a cochain map 
(note that a morphism is not necessarily a ring homomorphism). 
\begin{itemize}
\item{The basic principle here is that all Theorems in this article can be derived systematically from a pair $(\cA_{X}, \cdot, K_{X}) \in \hbox{\it Ob}(\frC_\Bbbk)$ and $\cC_\g: (\cA_{X},\cdot, K_{X}) \to (\Bbbk, \cdot, 0) \in \hbox{\it Mor}(\frC_\Bbbk)$.}
\end{itemize}
Note that a BV algebra in Definition \ref{bvd} can be regarded as an object of $\frC_\Bbbk$. This category $\frC_\Bbbk$ is studied in the context of \textit{homotopy probability theory} by the first named author in \cite{Pa12}.
The failure of ring homomorphism (with respect to the product $\cdot$) of a morphism in $\frC_\Bbbk$
is related to the notion of independence (so called, cumulants) in probability theory and the differential $K$ is related to homotopy theory. But here we will not touch any issues related to probability theory. 
Instead we will provide a self-contained argument and proofs regarding $\frC_\Bbbk$.
The category $\frC_\Bbbk$ can be seen as a bridge between period integrals of $X_G$ and $L_\infty$-homotopy theory.
The relationship between the period integral of $X_G$ and $\frC_\Bbbk$ is made by the representation $\rho_X$, and the relationship
between $\frC_\Bbbk$ and $L_\infty$-homotopy theory will be given by \textit{the descendant functor} $\Des$.

The descendant functor is a homotopy functor from the category $\frC_\Bbbk$ to the category $\frL$ of $L_\infty$-algebras (we include Appendix \ref{subs6.2} explaining notations for the homotopy category of unital $L_\infty$-algebras suitable for our purpose), which is defined by using the binary product $\cdot$ of an object of $\frC_\Bbbk$. 
See Definition \ref{defdesc} and Theorem \ref{Desc} for details. This functor can be regarded as an organizing principle (or tool) to understand the \textit{correlations} among $\cC_\g (x_1), \cC_\g(x_1\cdot x_2), \cdots, \cC_\g(x_1 \cdots x_m)$, where $x_1, \cdots, x_m$ are homogeneous elements in $\cA_{X}$ and $m \geq 1$.
This functor unifies two different failures of compatibility of algebraic structures into one language; we show that measuring how much the product $\cdot$ fails to be a derivation of $K$ induces an $L_\infty$-algebra structure on $\cA_{X}$, denoted $(\cA_{X}, \ud \ell^{K_{X}})$, 
and measuring how much $\cC_\g$ fails to be a $\Bbbk$-algebra homomorphism induces an $L_\infty$-morphism from $(\cA_{X},\ud \ell^{K_{X}})$ to $(\Bbbk, \ud 0)$, denoted $\ud \phi^{\cC_\g}$.
Note that the descendant functor is independent of hypersurfaces and their period integrals and is a general notion which measures incompatibilities of mathematical structures of the category $\frC_\Bbbk$.

 Once we get a descendant $L_\infty$-algebra $(\cA_X,\ud \ell^{K_X})$, we can study an extended formal deformation functor attached to it. This deformation includes the classical geometric deformation and has new directions
 which violate Griffiths transversality.  
 Theorem \ref{thirdtheorem}, Theorem \ref{fourththeorem}, and Therem \ref{fifththeorem}
  can be derived by a careful analysis of $(\cA_X,\ud \ell^{K_X})$ and the 
 descendant $L_\infty$-morphism $\ud \phi^{\cC_\g}$.

\subsection{Physical motivation; (0+0)-dimensional field theory}
We briefly explain the physical motivation behind the article. 
We decided to include it because the physical viewpoint was crucial to the conception of the paper. 
Even if the description is not entirely precise from a mathematical point of view, our hope is that it will be more helpful than confusing in guiding the reader through the rather elaborate constructions to follow.
What we prove regarding $X_G$ and its period integrals in the article is essentially to work out the details of the simplest possible field theory, a (0+0)-\textit{dimensional field theory} with the Dwork polynomial $S_{cl}=y\cdot G(\ud x)$ as the classical action.
Here (0+0) means 0-dimensional time and 0-dimensional space.
\foot{Every physical quantity has physical dimension $[\hbox{mass}]^a [\hbox{length}]^b [\hbox{time}]^c$, $a,b,c\in \Z$.
In the present case, we are dealing with $(0+0)$-dimensional space-time so that there is only one unit, which is converted to the {\it weight} $\wt$
of $y$. Note that a natural filtration generated by $y$ induces the Hodge filtration on $\H$; see $(c)$ of Theorem \ref{firsttheorem}.}
The basic principle is that any space can be viewed as the space of fields of a classical field theory of dimension 0.
A key point of the paper is that this principle is very useful in putting classical objects into an illuminating context that makes them amenable to natural generalizations.

Let $CFT_{S_{cl}}$ be such a (0+0)-dimensional classical field theory. Then we can view the smooth hypersurface $X_G$ as (the reduced component) of 
the classical equations of motion space of $CFT_{S_{cl}}$ \footnote{This physical view point suggests that we can use
the potential $S_{cl}=\sum_{i=1}^k y_k G_k(\ud x)$ when we deal with a smooth projective 
complete intersection $X$ in $\BP^n$ given by homogeneous polynomials $G_1(\ud x), \cdots, G_k(\ud x)$. }, i.e.
$$
\left\{
\eqalign{
\Fr{\rd S_{cl}}{\rd y}&= G(\underline{x})=0
,\cr
\Fr{\rd S_{cl}}{\rd x_i}&= y\cdot \Fr{\rd G(\underline{x})}{\rd x_i}=0, \qquad\forall i=0,1,\cdots, n.
}\right.
$$
The space of classical fields
is $\A:=\A^1_{\C}\times \A^{n+1}_{\C}\setminus \{\ud 0\}$ 
whose ring of regular algebraic functions is isomorphic to $A=$ $\C[y,x_0,x_1,\cdots, x_n]$ $=\C[\ud y]$.
We call $y_i$'s classical fields for each $i=-1, \cdots, n$.
The gauge group $\C^*$ acts on $\A^1_{\C}\times \A^{n+1}_{\C}\setminus \{\ud 0\}$ by $\l \cdot (y, x_0, \cdots, x_n):= (\l^{-d} y, \l x_{0}, \cdots, \l x_n)$ for $\l \in \C^*$ (note that $\ch(y)=-d$ and $\ch(x_i)=1$), so that $S_{cl}$ is invariant under the gauge action. 
Then the ring $\cR_{cl}$ of classical observables modulo physical equivalence is isomorphic to
the charge zero part 
\foot{If $X_G$ is Calabi-Yau ($d=n+1$), 
then classical observables (elements of charge zero) lift to quantum observables; this is the \textit{anomaly-free} case.
If $X_G$ is not Calabi-Yau ($d \neq n+1$), 
then classical observables (elements of charge zero) do not lift to quantum observables; there is an \textit{anomaly} in this case.
In fact, quantum observables have the background charge $c_X=d-(n+1)$.}
 of the Jacobian ring $J(S_{cl}):=\C[\ud y]\left/\left(\Fr{\rd S_{cl}}{\rd y},\Fr{\rd S_{cl}}{\rd x_0},\cdots, \Fr{\rd S_{cl}}{\rd x_n}\right)\right.$.
This motivates us to construct the commutative differential graded algebra $(\cA_X, \cdot, Q_X)$ whose cohomology is isomorphic to $\cR_{cl}$ (when $X_G$ is Calabi-Yau)
or $J(S_{cl})$ (when $X_G$ is not Calabi-Yau).
Roughly speaking, one can regard the passage from the ring $\cR_{cl}$ or $J(S_{cl})$ to the CDGA $(\cA_X, \cdot, Q_X)$ as an enhancement of classical algebraic geometry to derived algebraic geometry. 

We like to emphasize that if one just applies natural field theoretic constructions to the variety viewed as the classical equations of motion space, 
all classical constructions in the paper follow. 
We summarize the correspondence between the structures in $CFT_{S_{cl}}$ and the structures
arising in our paper in Table \ref{tone}. 

%

\begin{table}[h!] \normalsize
 \begin{center}
  \begin{tabular} 
 {| p{ 5.5cm} |  p{5.5cm}  |  }
\hline
 (0+0)-dimensional classical field theory $CFT_{S_{cl}}$ &  enhanced homotopy theory of hypersurfaces \\
\hline
\hline
 the space of classical fields modulo the gauge group & the weighted projective space $\BP^{n+1}(-d, 1, \cdots, 1)$ \\  
\hline
the classical equations of motion space &   the union of one point and the hypersurface $X_G$ \\
\hline
the space of classical observables modulo physical equivalence &  the charge zero part of the Jacobian ring $\C[\ud y]\left/\left(\Fr{\rd S_{cl}}{\rd y},\Fr{\rd S_{cl}}{\rd x_0},\cdots, \Fr{\rd S_{cl}}{\rd x_n}\right)\right.$\\
\hline
homotopy enhancement of $CFT_{S_{cl}}$  & the CDGA $(\cA_X, \cdot, Q_X)$  \\
\hline
     \end{tabular}
   \end{center}
  \caption{Classical field theory} \label{tone}
\end{table}

Then we quantize $CFT_{S_{cl}}$ by essentially following the Batalin-Vilkovisky (BV) quantization scheme in \cite{BV}, to construct a (0+0)-dimensional
quantum field theory $QFT_{S_{cl}}$ whose partition function 
is the Griffiths period integral of $X_G$; this leads us to the construction of
$\bvx= (\cA_X, \cdot, Q_X, K_X)$ and the cochain map $\cC_\g$ which enhances $C_{[\g]}$. 
We call $\eta_i$ the anti-field of the classical field $y_i$ and $\Delta$ is the BV operator in \cite{BV}. 
We view the differential $K_X$ as a BV quantization of the differential operator $Q_X$.
In this case, the space of quantum observables modulo physical equivalence is isomorphic to the middle dimensional primitive
cohomology  $\H$ of $X_G$, since the $0$-th $K_X$-cohomology group $H^0_{K_X}(\cA_X)$ is isomorphic to $\H$.
For each middle dimensional homology class $[\g] \in H_{n-1}(X_G,\Z)_0$, the cochain map $\cC_{\g}$ becomes a Feynman path integral, in the sense
of Batalin-Vilkovisky in \cite{BV}, such that the expectation value $\cC_\g(O)$ of a quantum observable $O$ is the period integral $\int_\g \omega$, where
$\omega$ is a representative of cohomology class $[\omega]\in \H$; recall that 
$$
\cC_\g (O ) = -\frac{1}{2 \pi i} \int_{\tau(\g)}\bigg( \int_{0}^{\infty} \pi_0(O) \cdot e^{y G(\underline x)}  dy\bigg) \Omega_n, \quad O\in \cA_X,
$$
and the measure $ -\frac{1}{2 \pi i} \int_{\tau(\g)}\bigg( \int_{0}^{\infty} \pi_0( \bullet ) \cdot e^{y G(\underline x)}  dy\bigg) \Omega_n$ 
can be regarded as
a path integral measure in $QFT_{S_{cl}}$.
In addition, $QFT_{S_{cl}}$ has a smooth formal based moduli space $\cM_{X_G}$
whose tangent space is isomorphic to $\H$ and, if $X_G$ is Calabi-Yau (anomaly-free in physical terminology), the tangent space has a structure of the formal Frobenius manifold.
The quantum master equation in $QFT_{S_{cl}}$ 
$$
 \rd_{\alpha}\rd_\beta e^{\G}
 =\sum_{\g \in I} A_{\alpha\beta}{}^\g \rd_\g e^{\G} +  {K}\left({\La}_{\alpha\beta}\cdot e^{\G}\right)\hbox{ for }\forall
 \alpha,\beta\in I,
 $$
 can be seen as a vast generalization of the Picard-Fuchs type differential equations.
We also have worked out the generating functions of every quantum correlation (see \mgfcy\ for their definition) up to finite ambiguity by 
an explicitly executable algebraic algorithm.

If we just apply a natural algebraic homotopy theoretical quantization, which is proposed by the first named author in \cite{Pa10}
 and enhances the BV quantization in \cite{BV}, to the classical field theory $CFT_{S_{cl}}$, 
then all quantum constructions in our paper follow.
We also summarize the correspondence between the structures in $QFT_{S_{cl}}$ and the structures
appearing in the paper in Table \ref{ttwo}.
 \begin{table}[h!] \normalsize
 \begin{center}
   \begin{tabular} 
 {| p{ 5.5cm} |  p{5.5cm}  |  }
\hline
(0+0)-dimensional quantum field theory $QFT_{S_{cl}}$ & enhanced homotopy theory of hypersurfaces \\
\hline
\hline
 the space of quantum observables modulo physical equivalence &  the middle dimensional cohomology $\H= H^{n-1}_{\pr} (X_G, \C)$ \\
\hline
 BV quantization of $CFT_{S_{cl}}$  & the BV algebra $\bvx$ \\
\hline
 Feynman path integral and partition function & the Griffith period integrals $\cC_{\g}$ \\
\hline
 the quantum master equation  & a generalization of the Picard-Fuchs equations\\
\hline
 the generating functions of every quantum correlation & the generating power series $\cZ_{[\g]}\left(\big[\ud \varphi^\H\big]\right)=\cC_{\g}(e^{\G_{\ud \varphi^\H}}-1)$ \\  
\hline
     \end{tabular} 
   \end{center}
  \caption{Quantum field theory}\label{ttwo}
\end{table}
\subsection{Plan of the paper}
Now we explain the contents of each section of the paper. The paper consists of 3 main sections and the appendix. 
In the first main section, Section \ref{section2}, we explain the general theory of \textit{period integrals} associated to a Lie algebra representation. 
In subsection \ref{subs2.1}, we define the notion of \textit{period integrals} of a Lie algebra representation $\rho$ (see Definition \ref{Liedef}). 
Then, in subsection \ref{subs2.2}, we explain how to construct a cochain complex associated to $\rho$ which is \textit{dual} to the Chevalley-Eilenberg complex, and a way of associating a morphism into $(\Bbbk,0)$ in the category $\frC_\Bbbk$ 
to its period integral. 
Then we illustrate by an example why the dual Chevalley-Eilenberg complex is crucial and more suitable to understand the period integral of $\rho$ than the cohomology Chevalley-Eilenberg complex attached to $\rho$ in subsection \ref{subs2.3}.

 The second main section, Section \ref{section3}, is about the general theory of the category $\frC_\Bbbk$. 
 The key concepts are the \textit{descendant functor}, \textit{generating power series}, and \textit{flat connections}.  
 In subsection \ref{subs3.0}, we explain the basic philosophy of the descendant functor.
  In subsection \ref{subs3.1}, we provide a way to understand the category $\frC_\Bbbk$ in terms of $L_\infty$-homotopy theory; we construct the homotopy \textit{descendant functor} from the category $\frC_\Bbbk$ to the category $\frL$ of $L_\infty$-algebras. Then we show that a descendant $L_\infty$-algebra is {formal} in subsection \ref{subs3.2}.
In subsection \ref{subs3.3},  we attach a deformation problem to the descendant $L_\infty$-algebra of $(\cA, \cdot, K)$ and explain what we deform. 
 In subsection \ref{subs3.4}, we define a notion of the generating power series, which organizes various \textit{correlations and deformations} of period integrals into one power series in the deformation parameters, and show that they are $L_\infty$-homotopy invariants.
 Then we verify that the generating power series attached to a versal formal deformation satisfies a system of partial differential equations (Theorem \ref{diff}) with respect to derivatives of deformation parameters and show the coefficients $A_{\alpha\beta}{}^{\g}(\ud t)$ appearing in the differential equations are $L_\infty$-homotopy invariants, in subsection \ref{subs3.5}. 
 In subsection \ref{subs3.6}, we provide a way to compute the generating power series explicitly.
 Finally, in subsection \ref{subs3.7}, this system of partial differential equations is interpreted as the existence of a flat connection on the tangent bundle of a formal deformation space attached to $(\cA, \ud \ell^K)$.
In light of this, Section \ref{section2} can be regarded as a general way to provide examples of objects along with morphisms to the initial object (the period integrals of Lie algebra representations) in the category $\frC_\Bbbk$.

In the third main section, Section \ref{section4}, we apply all the general machinery of the previous sections to reveal hidden structures on the singular cohomologies and the Griffiths period integrals of smooth projective hypersurfaces. Section \ref{section4} can be viewed as a source of explicit examples of non-trivial period integrals of certain Lie algebra representations, and gives non-trivial examples of objects and morphisms into the initial object in $\frC_\Bbbk$.
In subsection \ref{subs4.0}, we apply the general theory to the toy model to illustrate our homotopical viewpoint of
understanding the Griffiths period integral.
In subsection \ref{subs4.1}, we explain how to attach a Lie algebra representation $\rho_X$ to a projective smooth hypersurface $X$. In subsection \ref{subs4.2}, we briefly recall Griffiths' theory of period integrals of smooth projective hypersurfaces and construct a non-trivial period integral of its associated Lie algebra representation. 
Then, in subsection \ref{subs4.3}, we explicitly construct a commutative differential graded algebra $(\cA_X, Q_X)$ whose cohomology
essentially describes the coordinate ring of $X_G$ and in subsection \ref{subs4.31} we construct the cochain complex $(\cA_X, K_X)$ with super-commutative product (a quantization of $(\cA_X, \cdot, Q_X$)) attached to the hypersurface $X$. As a consequence we prove  $(a), (b)$ and $(e)$ of Theorem \ref{firsttheorem}. 
We prove $(c)$ and $(d)$ of Theorem \ref{firsttheorem} in subsections \ref{subs4.32} and subsection \ref{subs4.33}, respectively. 
In subsection \ref{subs4.40}, we compute its $K$-cohomology $H^i_K(\cA)$ of $\cA$ for every $i \in \bZ$.
In subsection \ref{subs4.4}, we prove Theorem \ref{thirdtheorem}.
We verify Theorems \ref{fourththeorem} and \ref{fifththeorem} in subsections \ref{subs4.6} and \ref{subs4.61}.
We provide a precise relationship between the Gauss-Manin connection and our flat connection on the tangent bundle 
of a formal deformation space in subsection \ref{subs4.62}.
Finally, in subsection \ref{subs4.5}, we explain how to compute (extended formal) deformations of the Griffiths period integrals and the period matrices
via the ideal membership problem based on the Gr\"ober basis.

%


The main idea of this paper originated from the first named author's work on the algebraic formalism of quantum field theory, \cite{Pa10}. Thus, in the appendix, Section \ref{section6}, we decided to add an explanation of the quantum origin of the Lie algebra representation attached to a given hypersurface $X_G$ (subsection \ref{subs6.1}). 
Finally, we include subsection \ref{subs6.2} on $L_{\infty}$-algebras, $L_{\infty}$-morphisms, and $L_{\infty}$-homotopies in order to explain the notations and conventions used throughout the paper. \\

Before finishing the introduction, we mention two things. Firstly, our theory can be generalized to toric complete intersections from hypersurfaces and conjecturally to any algebraic varieties.
We will carry out the details for (toric) complete intersections in another papers and the relevant references which play a similar role as \cite{Gr69} would be \cite{AdSp06}, \cite{Dim95}, and \cite{Ter90}. 
Secondly, it may seem artificial to study $\frC_\Bbbk$ at a first glance: we study the category $\frC_\Bbbk$ whose objects are $(\cA,\cdot, K)$, where we \textit{do not require compatibility} between the super-commutative product $\cdot$ and the differential $K$, and whose morphisms are \textit{not structure preserving maps} in the sense that they preserve only additive and differential structures (i.e. 
they are cochain maps), not super-commutative ring structure (note a difference between the category $\frC_\Bbbk$ and the category of CDGAs, i.e. commutative differential graded algebras, where all the structures are \textit{compatible}).
But this category $\frC_\Bbbk$ is worth investigating and studying; the objects in $\frC_\Bbbk$ include BV algebras, and the Griffiths period integral of the hypersurface $X_G$ can be interpreted (very neatly) as a morphism from $(\cA_X, \cdot, K_X)$ to the initial object $(\Bbbk,\cdot, 0)$ in the category $\frC_\Bbbk$. Further the shadow $L_\infty$-homotopy information obtained by applying the descendant functor \textit{measures the failure of compatibilities among structures} and reveals hidden structures on the period integral $C_{[\g]}$.\\



\subsection*{Acknowledgments}
The work of Jae-Suk Park was supported by IBS-R003-G1 and Mid-career Researcher Program through NRF grant funded by the MEST (No. 2010-0000497).
The work of Jeehoon Park was partially supported by the Basic Science Research Program through the National Research Foundation of Korea(NRF) funded by the Ministry of Education, Science and Technology(2013023108) and was partially supported by the Basic Science Research Program through the National Research Foundation of Korea(NRF) funded by the Ministry of Education(2013053914).
The authors would like to thank Minhyong Kim for helpful comments and Gabriel C. Drummond-Cole for proofreading.
The first author thanks John Terilla for useful discussions on related subjects.
Finally, the authors thank for the anonymous referee for many helpful comments and suggestions to improve the article.

 \newsec{Lie algebra representations and period integrals}  \label{section2}

 \subsection{Period integrals of Lie algebra representations} \label{subs2.1}
 
 Let $\Bbbk$ be a field of characteristic $0$ and $\frg$ be a finite dimensional Lie algebra over $\Bbbk$.
 Let $\rho: \frg \to \End_{\Bbbk}(A)$ be a $\Bbbk$-linear representation of $\frg$.
 We assume that $A$ is a commutative associative $\Bbbk$-algebra (with unity) throughout the paper.
 \bed \label{Liedef}
We call a $\Bbbk$-linear map $C: A \to \Bbbk$  a period integral\footnote{We use this terminology in a different sense than arithmetic geometers (a comparison of rational structures of relevant cohomology groups); we simply choose this terminology since the period integrals of smooth hypersurfaces can be understood as an example.} attached to $\rho$ if $C(x)=0$ for every $x$ in the image of $\rho(g)$ for every $g \in \frg$. 
\eed

Note that such a map $C$ is necessarily zero if $A$ is an irreducible $\frg$-module. 
For a given Lie algebra representation $\rho$, it would be an interesting question to find non-trivial period integrals. 
Here we present a simple non-trivial example.
\begin{example} \label{exa1}
Let $\frg$ be a one-dimensional Lie algebra $\Bbbk=\bR$ generated by $\a$. Let $\rho$ be a Lie algebra representation on 
$A=\Bbbk[x]$ given by 
$\rho(\a)= \prt{}{x}-x \in \End_\Bbbk(A)$. Then we consider the Gaussian probability measure 
$\frac{1}{\sqrt{2\pi}} e^{-\frac{x^2}{2}} dx$
and define a $\Bbbk$-linear map 
\eqn\gpi{
\eqalign{
C: \Bbbk[x] &\to \Bbbk 
,\cr
f(x) &\mapsto C(f(x)):=\frac{1}{\sqrt{2\pi}}\int_{-\infty}^\infty f(x) e^{-\frac{x^2}{2}} dx.
}
}
This is an example of a period integral of $\rho$, since 
\be
 \frac{1}{\sqrt{2\pi}}\int_{-\infty}^\infty \left(\prt{f(x)}{x}- xf(x)\right) e^{-\frac{x^2}{2}} dx=0.
\ee
\end{example}

The above Gaussian period integral is a special example of a more general kind. 
Let 
$$
A=\Bbbk[\ud q]=\Bbbk[q^1, q^2, \cdots, q^m]
$$ 
be a polynomial ring with $m$ variables for $m\geq 1$. Let
$S=S(\ud q)\in A$. Let $J_S \subset A$ be the Jacobian ideal of $S(\ud q)$, i.e. the ideal generated 
by $\prt{S(\ud q)}{q^1}, \prt{S(\ud q)}{q^2}, \cdots, \prt{S(\ud q)}{q^m}$. Let $\frg=\frg_S$ be the finite dimensional 
abelian Lie algebra over $\Bbbk$ of dimension $m$, generated by $\alpha_1, \alpha_2, \cdots, \alpha_m$.
We define the following Lie algebra representation $\rho_S: \frg \to \End_{\Bbbk}(A)$:
\eqn\QJR{
\rho_S(\alpha_i) = \exp\left(-S(\ud q)\right)\cdot \prt{}{q^i} \cdot \exp\left(S(\ud q)\right)=\prt{}{q^i} + \prt{S(\ud q)}{q^i}, \quad i= 1, 2, \cdots, m.
}
We remark that this representation $\rho_S$ is obtained by twisting the Schr\"{o}dinger representation of (a certain abelian Lie subalgebra) of the Heisenberg Lie algebra
by the polynomial $S(\ud q)$; see subsection \ref{subs6.1} for details. This motivates us to call $\rho_S$ the {\it quantum Jacobian Lie algebra representation}
associated to $S(\ud q) \in A$.
It turns out that there are many interesting non-trivial examples of period integrals of $\rho_S$.

\begin{example}\label{exa11}
We give an example for which $m=1$, which generalizes the previous Gaussian example. 
Let $S(x)\in \bR[x]=A$ be a polynomial such that
$$\lim_{x\to \infty} f(x) e^{S(x)} = \lim_{x \to -\infty} f(x) e^{S(x)}=0$$ for every $f(x) \in \bR[x]$.
Let $\frg$ be a one-dimensional Lie algebra generated by $\a$.
Then the $\bR$-linear map
\eqn\spi{
\eqalign{
C: \bR[x] &\to \bR
,\cr
f(x) &\mapsto C(f(x)):=\int_{-\infty}^\infty f(x) e^{S(x)} dx,
}
}
is an example of a period integral of $\rho_S$, since 
\be
C\left(\rho_{S}(\a)(f(x))\right)=\int_{-\infty}^\infty \big(\prt{f(x)}{x}+\prt{S(x)}{x} f(x)\big) e^{S(x)} dx=0, \quad \forall f(x) \in \bR[x].
\ee

\end{example}

Such a period integral $C$ attached to $\rho=\rho_S$ gives rise to a map $\overline C: A/\cN_{\rho} \to \Bbbk$,
where $\cN_\rho:= \sum_{\alpha\in\frg} \im \rho(\alpha)$. Note that, in general, $C$ fails to be an algebra homomorphism and $\cN_\rho$ fails to be an ideal of $A$. This failure will play a pivotal role in studying the period integral $C$ via $L_\infty$-homotopy theory.

Our main example, which we will focus on in section \ref{section4}, is  
when the Lie algebra representation $\rho_S$ is constructed out of $S=y \cdot G(\ud x)$, the so-called Dwork polynomial of $G(\ud x)$, where $G(\ud x)$ is the defining equation of a smooth projective hypersurface. 
Then the rational period integral of a smooth projective hypersurface $X_G$, which was extensively studied by Griffiths and Dwork,
can be interpreted as the period integral of $\rho_S$ (see subsection \ref{subs4.2}).

We remark that studying a non-trivial period integral of a Lie algebra representation of a {\it non-abelian} Lie algebra
would be a very interesting question, though we limit all the examples to the abelian case in this article.


\subsection{Cochain map attached to a period integral}\label{subs2.2} 

We now explain our strategy to study
period integrals, assuming such a nonzero period integral $C$ is given. The main idea is to enhance the $\Bbbk$-linear map
$C: A \to \Bbbk$ to the cochain complex level (we do this in this subsection) and develop an infinity homotopy theory (see section \ref{section3}) by analyzing the failure of
$C$ to be an algebra homomorphism systematically. In this paper, the relevant homotopy theory will be the $L_\infty$-homotopy theory (this fact is related to the assumption that $A$ is a commutative $\Bbbk$-algebra). 

Let $C: A \to \Bbbk$ be a nontrivial period integral attached to $\rho:\frg \to \End_{\Bbbk}(A)$.
We will construct a cochain complex $(\cA, \cdot, K)=(\cA_\rho, \cdot, K_\rho)$ with super-commutative product $\cdot$ whose degree 0 part is $A$, and a cochain map $\cC: (\cA, K) \to (\Bbbk,0)$, where we view $(\Bbbk,0)$ as a cochain complex which has only degree 0 part and zero differential. 

 Let $\alpha_1, \cdots, \alpha_n$ be a $\Bbbk$-basis of $\frg$ where $n$ is the dimension of $\frg$. 
We consider $\frg$ as a $\bZ$-graded $\Bbbk$-vector space with only degree 0 part.
Then $\frg[1]$ is a $\bZ$-graded $\Bbbk$-vector space with only degree -1 part.
Let $\eta_1, \cdots, \eta_n$ be the $\Bbbk$-basis of $\frg[1]$ corresponding to $\alpha_1, \cdots, \alpha_n,$  so that they have degree -1. We consider the following $\bZ$-graded super-commutiave algebra
\eqn\scalg{
S (\frg[1]) = T(\frg [1])/J
}
where $T(\frg[1])$ is the tensor algebra of $\frg[1]$ and $J$ is the ideal of $T(\frg[1])$ generated
by the elements of the form $x\otimes  y- (-1)^{|x|\cdot|y|} y\otimes x$ with $x, y \in T(\frg[1])$.
Note that $|x|$ here means the degree of $x$. We can also view $A$ as a $\bZ$-graded $\Bbbk$-algebra concentrated in degree zero part. Then we define the $\bZ$-graded $\Bbbk$-algebra $\cA_\rho$ as the supersymmetric tensor product of $A$ and $S(\frg[1])$. 

\bep
The $\Bbbk$-algebra $\cA=\cA_\rho$ is a $\bZ$-graded super-commutative  algebra and we have a decomposition of $\cA$ into
\be
\cA^{-n} \oplus \cA^{-(n-1)} \oplus \cdots \oplus \cA^{-1} \oplus \cA^{0}
\ee
where $\cA^m$ is the $\Bbbk$-subspace of $\cA$ consisting of degree $m$ elements, with $\cA^{0} = A$.
\eep
\begin{proof}
The fact that $\cA$ is super-commutative ($x y = (-1)^{|x|\cdot |y|} yx$ for every homogeneous $x,y \in \cA$) follows from the construction. It is clear that $\eta_i ^2 = 0$ for $i =1, \cdots, n$, since $2 \eta_i^2 = 0$  (recall that $\eta_1, \cdots, \eta_n$ is a $\Bbbk$-basis of $\frg[1]$ whose degree is -1) and the characteristic of $\Bbbk$ is not 2. Therefore the smallest degree which the elements of $\cA$ can have is $-n$ (for example, $\eta_1\cdots \eta_2 \cdots \eta_n$ has degree $-n$).
 \end{proof}

Now we construct a differential $K=K_\rho$ coming from the Lie algebra representation $\rho$;
\be
K_\rho: \cA^m \to \cA^{m+1},
\ee
where $m \in \bZ$. Define
\eqn\Kdef{
K_\rho= \sum_{i=1}^n \rho_{\alpha_i}\otimes\pa{\eta_i} - I\otimes \sum_{i,j,k=1}^n\frac{1}{2} f_{ij}{}^k \eta_k \pa{\eta_i}
\pa{\eta_j}: \cA^m \to \cA^{m+1},
}
where $\{f_{ij}{}^k \} \in \Bbbk$ are the structure constants of the Lie algebra $\frg$ defined by the relation $[\alpha_i, \alpha_j] =\sum_{k=1}^n f_{ij}{}^k \alpha_k$ and $\rho_{\alpha_i}=\rho(\alpha_i)$. Note that $\rho_{\alpha_i}$ only acts on the degree zero part $A=\cA^0$ via the representation $\rho$ and the partial derivative operator $\pa{\eta_i}$ increases degree 1, since $\eta_i$ has degree -1.  We shall often omit the tensor product symbol in the expression of $K_\rho$.
Next, we show that $K_\rho$ is actually a differential of $\cA_\rho$.
\bep
We have that $K_\rho^2 =0$ and $K_\rho (1_{\cA})=0$.
\eep
\begin{proof}
This follows from the fact that $\rho:\frg \to \End_{\Bbbk}(A)$ is a Lie algebra representation.
 \end{proof}

We extend our period map $C:A \to \Bbbk$ to $\cC: \cA \to \Bbbk$ by setting $\cC(x) = C(x)$ if
$x\in \CA^0=A$ and $\cC(x)= 0$ otherwise. 
\bep \label{enhancetochain}
We can extend any period map $C$ attached to $\rho$ to a cochain map $\cC$ from $(\cA_\rho, K_\rho)$ to $(\Bbbk, 0)$, i.e. $\cC \circ K_\rho = 0$.
\eep
\begin{proof}
Note that $A=\cA_\rho^0$. We only  have to check $\cC (K_\rho(x))=0$ when $K_\rho(x) \in A=\cA_\rho^0$. Let us write the general element
$x$ in $\cA^{-1}_\rho$ as 
\be
x = \sum_{i=1}^n F_i \cdot \eta_i, \text{ where $F_i \in A.$}
\ee
Then $K_\rho (x) = \sum_{i=1}^n \rho(\alpha_i)(F_i) \in A$. By the definition of the period integral
attached to $\rho$, we immediately see that $C (K_\rho (x))=0$ for any $x \in \cA_\rho^{-1}$. 
 \end{proof}

\begin{example}\label{exa2}
We illustrate the above construction for Example \ref{exa11}. In this case, the cochain complex $\cA_\rho$ with super-commutative product, associated to $\rho=\rho_{S(x)}$, is given by
\be
\cA_\rho= \bR[x][\eta], \quad \eta^2=0,\ \eta x= x \eta,
\ee
where $\eta$ is an element of degree -1 (the so-called ghost component). Then $\cA_\rho^{m}=0$ unless $m =0, -1$. The differential $K_\rho$ is given by
\be
K_\rho = \rho(\a) \prt{}{\eta}= \left(\prt{}{x}+\prt{S(x)}{x}\right) \prt{}{\eta}.
\ee
The period integral $C$ in \spi\ can be enhanced to a cochain map $\cC: (\cA_\rho, \cdot, K_\rho) \to (\bR, \cdot, 0)$ by Proposition \ref{enhancetochain}. 
Then $\cC$ induces the map $\overline{\cC}: H_K(\cA_\rho) \to \bR$, where $H_K(\cA_\rho)=
\oplus_{i\leq 0} H_K^i(\cA_\rho)$ with $i$-th degree cohomology $H_K^i(\cA_\rho)$ of $(\cA_\rho, K_\rho)$.
Then it turns out that $H_K(\cA_\rho)= H^0(\cA_\rho)$ and is a finite dimensional $\bR$-vector space of dimension
equal to the 
degree of $\prt{S(x)}{x}$.
In section \ref{section3}, we will provide a general machinery to understand important properties (correlations and deformations) of such an enhanced period integral $\cC$ by using $L_\infty$-homotopy theory. 
We refer to subsection \ref{subs4.0} for a detailed analysis of Example \ref{exa2} via $L_\infty$-homotopy theory.
\end{example}

When we study the period integral of a Lie algebra representation $\rho$ of $\frg$, we are particularly interested in the case 
where $\rho(\alpha)$
is a differential operator of order $n$ for $\alpha \in \frg$. Recall the definition of Grothendieck:
\bed \label{order}
Let $\cA$ be a unital $\bZ$-graded $\Bbbk$-algebra. Let $\pi \in \End_{\Bbbk}(\cA)$. We call $\pi$ a differential operator of order $n$, if
$n$ is the smallest positive integer such that $\ell_n^\pi \neq 0$ and $\ell_{n+1}^\pi =0$,
where
\be
\ell_n^\pi(x_1, x_2, \cdots, x_n)=[[\cdots[[\pi, L_{x_1}],L_{x_2}], \cdots],L_{x_n}](1_\cA),
\ee
for $x_1, \cdots, x_n \in \cA$. Here $L_x:\cA\to\cA$
is left multiplication by $x$, and the commutator $[L,L']:=L\cdot L' - (-1)^{|L|\cdot|L'|}L' \cdot L \in \End_{\Bbbk}(\cA)$, and $1_{\cA}$ is the identity element of $\cA$.
\eed
If the Lie algebra $\frg$ is non-abelian, then $K_\rho$ for its arbitrary Lie algebra representation $\rho$ has order at least $2$, because of the term $\pa{\eta_i}\pa{\eta_j}$ in \Kdef. In general, $K_\rho$
for any Lie algebra representation $\rho$ has order at least the order of $\rho(\alpha_i)+1$ for any $\alpha_i\in \frg$, because of the term $\rho(\alpha_i) \pa{\eta_i}$ in \Kdef.

\subsection{The origin of the cochain complex associated to $\rho$ }\label{subs2.3}

In subsection \ref{subs2.2}, we constructed a cochain complex $(\cA_\rho, K_\rho)$ associated to $\rho$ and explained how to enhance the period integral $C$ of $\rho$ to a cochain map $\cC$. 
This can be seen as a degree-twisted cochain complex of the (homology version of) the Chevalley-Eilenberg complex in \cite{CE48}; see Proposition \ref{lhi}.
For a given Lie algebra representation $\rho$, there are two kinds of standard complexes, 
the cochain complex for the Lie algebra cohomology 
$H^k(\frg, A)$ and the chain complex for the Lie algebra homology $H_k(\frg,A)$. It will be crucial to use
the Lie algebra homology complex instead of the cohomology complex for our analysis of period integrals, 
for which we
will explain the reason.

We briefly
recall the Chevalley-Eilenberg complex for cohomology. We define a $\bZ$-graded vector space
\be
C(\frg;\rho) = \bigoplus_{p \geq 0} C^p(\frg;\rho), \text{ where } C^p(\frg;\rho):= A  \otimes_{\Bbbk} \Lambda^p \frg^*.
\ee
 If $\beta \in \frg^*=\Hom_\Bbbk(\frg, \Bbbk)$, define $\epsilon(\beta): \Lambda^p \frg^* \to \Lambda^{p+1} \frg^*$, for any integer $p \geq 0$, by wedging with $\beta$:
\be
\epsilon(\beta) \omega = \beta \wedge \omega.
\ee
Similarly, if $X \in \frg$, then define $\iota(X): \Lambda^p \frg^* \to \Lambda^{p-1} \frg^*$, for any
integer $p \geq 1$, by contracting with $X$:
\be
\iota(X) \beta = \beta (X), \quad \text{ for } \beta \in \frg^*
\ee
and extending it as an odd derivation
\be
\iota(X)(\alpha \wedge \beta) = \iota(X)\alpha \wedge \beta + (-1)^{|\alpha|}\alpha \wedge \iota(X) \beta
\ee
to the exterior algebra $\Lambda^{\bullet} \frg^*$ of $\frg^*$. Here $\alpha \in \Lambda^p \frg^*$ if and only if $|\alpha|=p$. Notice that $\epsilon(\alpha) \iota(X) + \iota(X) \epsilon(\alpha)= \alpha(X) \id$.
If we let $\{\alpha_i\}$ and $\{\beta^i\}$ be canonically dual bases for $\frg$ and $\frg^*$ respectively, then 
it is well-known that the Chevalley-Eilenberg differential on $C(\frg;\rho)$ can be written as
\eqn\Ddef{
d_\rho=d= \sum_i  \rho(\alpha_i) \otimes \epsilon(\beta^i) - \id \otimes \frac{1}{2} \sum_{i,j,k} f_{ij}^k \cdot \iota(\alpha_k)\epsilon(\beta^i)\epsilon(\beta^j): C^p(\frg;\rho) \to C^{p+1}(\frg;\rho).
}
This construction gives the cochain complex $(C(\frg;\rho), d_\rho),$ 
called the Chevalley-Eilenberg complex attached to $\rho$, which computes the Lie algebra cohomology
of $\rho$. If one compares $d_\rho$ with
$K_\rho$, then the wedging operator $\epsilon(\beta^i)$ corresponds to $\prt{}{\eta_i}$ and the contracting operator $\iota(\alpha_k)$ corresponds to multiplication by $\eta_k$. Note that the Chevalley-Eilenberg complex is obtained by adding 
degree 1 elements $\{ \beta^i \} $ to $A$ and 
the cochain complex $(A_\rho, K_\rho)$ is obtained by adding degree -1 elements $\{ \eta_i \}$ to $A$. Thus
$C(\frg;\rho)$ has no negative degree components and $\cA_\rho$ has no positive degree components.
This duality between $K_\rho$ and $d_\rho$ leads us to prove that $(\cA_\rho, K_\rho)$ is, in fact, 
a degree-twisted version of the Lie algebra homology Chevalley-Eilenberg complex.
In proving such a result, we also briefly recall the (dual) Chevalley-Eilenberg complex which computes the Lie algebra homology; we consider a $\bZ$-graded vector space 
\be
E(\frg;\rho) = \bigoplus_{p\geq 0} E_p(\frg; \rho), \  \text{where} \ E_p(\frg; \rho):=A \otimes_{\Bbbk} \Lambda^p \frg,
\ee
and equip it with the differential $\delta_\rho$ defined by
\eqn\dco{
\eqalign{
\delta_\rho\left( a \otimes (x_1\wedge \cdots \wedge x_n) \right) &=
\sum_{i=1}^n (-1)^{i+1} \rho(x_i) (a) \otimes (x_1 \wedge \cdots \wedge \hat{x_i} \wedge \cdots, \wedge x_n) 
\cr
&+\sum_{1\leq i < j \leq n} (-1)^{i+j} a \otimes ([x_i,x_j] \wedge x_1 \wedge \cdots \wedge \hat{x_i}\wedge \cdots
\wedge \hat{x_j}
\wedge \cdots \wedge x_n).
}
}
Then we twist \footnote{The degree twisting is needed for consistency with degree convention of
the $L_\infty$-morphisms which we will consider in Section \ref{section3}; we want our $L_\infty$-morphisms have degree 1
instead of -1.} the degree of the chain complex $E(\frg;\rho)$ in order to make a cochain complex $\tilde E(\frg;\rho) 
=\bigoplus_{p\leq 0} \tilde E^p(\frg; \rho)$;
$$
\tilde E^p (\frg; \rho) := E_{-p}(\frg; \rho), \quad p \leq 0.
$$
Then $\tilde E(\frg; \rho)$ has only negative degrees up to the $\Bbbk$-dimension of $\frg$ and becomes
a cochain complex.
\bep \label{lhi}
The cochain complex $(\cA_\rho, K_\rho)$ is isomorphic to the cochain complex $(\tilde E(\frg;\rho), \delta_\rho)$.
\eep

\begin{proof}
If we denote a $\Bbbk$-basis of $\frg$ by $\alpha_1, \cdots, \alpha_n$, the $\Bbbk$-linear map sending
$\alpha_i$ to $\eta_i$, for each $i =1, \cdots, n,$ clearly defines an $A$-module isomorphism 
(a $\Bbbk$-vector space isomorphism, in particular).
The commutativity with differentials follows from a direct comparison between \Kdef\ and \dco.
\end{proof}

Since we assume that $A$ has a commutative associative product in addition to the module structure,
this induces a natural super-commutative product on $\tilde E(\frg;\rho)$, i.e. the tensor product algebra
of $A$ and the alternating algebra $\Lambda^{\bullet} \frg$. Then it is clear that 
the module isomorphism in Proposition \ref{lhi} also respects the algebra structure.
The binary product of $A$ has important information about correlation of period integrals.

The Chevalley-Eilenberg complex $(E(\frg;\rho), d_\rho)$ has been studied more often by algebraic topologists and algebraic 
geometers
rather than $(E(\frg;\rho), \delta_\rho)$ (or equivalently $(\cA_\rho, K_\rho)$).\footnote{For example, the Chevalley-Eilenberg complex
$C(\frg;\rho_{y\cdot G(\ud x)})$ of the {\it quantum Jacobian Lie algebra representation} $\rho_{y\cdot G(\ud x)}$ 
in \QJR\, where $q^1=y, q^2=x_0, q^3=x_2, \cdots, q^m=x_n$ and $G(\ud x)$ is a defining polynomial of a smooth projective
hypersurface $X_G$ of dimension $n-1$, turns out to be the {\it algebraic Dwork complex} studied in several articles,
including  \cite{AdSp06}, and \cite{Ka68}.}
But, in our theory, $(E(\frg;\rho), \delta_\rho)$ is more useful and the failure of $K_\rho$ being a derivation will play a key role in deriving an $L_\infty$-algebra from $(\cA_\rho, \cdot, K_\rho)$ by the \textit{descendant functor} (in Section \ref{section3}) and studying the corresponding formal deformation theory.

We add a simple justification why $(\cA_\rho, \cdot, K_\rho)$ is more suitable than $(C(\frg;\rho), \wedge, d_\rho)$ for 
understanding the period integral $C$ of $\rho$. Observe that we can similarly enhance $C:A\to
\Bbbk$ to a cochain map $\cC': (C(\frg;\rho), d_\rho) \to (\Bbbk,0)$ by setting $\cC'(x)=C(x)$ if $x \in 
C^0(\frg;\rho)=A$ and $\cC'(x)=0$ otherwise; note that $\cC' \circ d_\rho=0$ merely by the definition of $\cC'$. But this cochain map $\cC'$ loses the key information of the period integral $C$; we illustrate
this by using the toy
 example \ref{exa1}. The induced map of $\cC'$ on the $0$-th cohomology
$H^0(\frg,A)=H^0_{d_\rho}(C(\frg;\rho))$ contains all the information of $C$. In Example \ref{exa1},
we see that $H^0(\frg,A):=\ker(d_\rho) \cap C^0(\frg;\rho)=0$, since the differential equation
\be
\prt{f(x)}{x}- x f(x)=0
\ee
does not have a non-zero solution in $A= \Bbbk[x]$. This means that the map $\overline{\cC'}: H^\bullet(\frg,A)
\to \bR$ induced from the cochain map $\cC'$ is zero; so it is not a good
cochain level realization of C in \gpi. On the other hand, the 0-th cohomology $H^0_{K_\rho}(\cA_\rho)$
is isomorphic to the $\Bbbk$-vector space $\Bbbk[x]/\cN_\rho$, where 
\be
\cN_\rho=\{ \prt{f(x)}{x} -x f(x) \ : \ f(x) \in \Bbbk[x]\}.
\ee
The induced map $\overline{\cC}$ of $\cC$ on the cohomology $\Bbbk[x]/\cN_\rho$ has substantial information about $C$. This justifies our use of the (twisted) 
homological version $(\cA_\rho, K_\rho)$ of the Chevalley-Eilenberg complex rather than the cohomological version.

%
%



\newsec{The descendant functor and homotopy invariants} \label{section3}

 This section is about the general theory of the {\it descendant functor} from
 the category $\frC_\Bbbk$ to the category $\frL_\Bbbk$ of $L_\infty$-algebras over $\Bbbk$.
This theory will provide the general strategy for analyzing the period integrals of a Lie algebra representation and its associated cochain complex and cochain map via $L_{\infty}$-homotopy theory.

\subsection{Why descendant functors?} \label{subs3.0}

Here we explain how the notion of a {\it descendant functor} arises in the study of period integrals of Lie algebra representations.
This provides the general framework behind the main theorems of this paper. Let $\Bbbk$ be a field of characteristic $0$ and $\mg$ be a finite dimensional Lie algebra over $\Bbbk$. 
Let $\rho: \frg \to \End_{\Bbbk}(A)$ be a $\Bbbk$-linear representation of $\frg$, where
$A$ is a commutative associative $\Bbbk$-algebra, such that $\rho(g)$ acts on $A$ as a linear differential operator
for all $g \in \mg$. 
We called a  $\Bbbk$-linear map $C: A \to \Bbbk$  a {\it period integral} attached to $\rho: \frg \to \End_{\Bbbk}(A)$, if $C(x)=0$ for every 
$x \in \CN_{\rho}:=\sum_{g \in \mg}\im \rho(g)$. Hence such a period integral $C$ induces a map 
$\cP_{C}: A/\CN_\rho \rightarrow \Bbbk$.
Two closely related properties of our period integrals  are that, in general, 

(i) $C: A \to \Bbbk$ fails to be an algebra homomorphism, 

(ii) $\CN_\rho$ fails to be  an ideal of $A$.  

The {\it descendant functor} $\Des$ is designed to provide a general framework to understand such period integrals 
and their higher structures  by exploiting those failures and successive failures systematically.
The domain category of $\Des$ is the category $\frC_\Bbbk$
and the target category is the category $\mL_\Bbbk$ of $L_\infty$-algebras over $\Bbbk$.

The  category $\frC_\Bbbk$ is defined
 such that
objects are triples $(\CA,\hbox{}\cdot\hbox{}, K)$, where the pair $(\cA,\hbox{}\cdot\hbox{})$ is
a $\Z$-graded super-commutative associative $\Bbbk$-algebra  while the pair $(\CA, K)$ is a 
cochain complex over $\Bbbk$,
and morphisms are cochain maps.
 We note that two of the salient properties of the category $\frC_\Bbbk$ are that
 
(i) morphisms are not required to be algebra homomorphisms, 

(ii) the differential and multiplication in an object have no compatibility condition.

Then the functor $\Des:\frC_\Bbbk \to \mL_{\Bbbk}$ takes

(i) a morphism $f$ in $\mC_\Bbbk$ to an $L_\infty$-morphism $\underline{\phi}^f=\phi_1^f, \phi_2^f,\phi_3^f,\cdots$, where $\phi_1^f=f$ and $\phi_2^f,\phi_3^f,\cdots$ measure
the failure and higher failures of $f$ being an algebra homomorphism,

(ii)
an object $(\CA,\hbox{}\cdot\hbox{}, K)$ in $\frC_\Bbbk$ to an $L_\infty$-algebra
$(\CA,\underline{\ell}^{K}=\ell_1^K,\ell_2^K, \ell_3^K,\cdots)$, where $\ell_1^K=K$ and $\ell^K_2, \ell_3^K,\cdots$
measure the failure and higher failures of $K$ being a derivation of the multiplication in $\cA$.

Morphisms in both categories $\mC_\Bbbk$ and $\mL_\Bbbk$ come with a natural notion of homotopy and 
the descendant functor induces a well defined functor from the homotopy category  $h\frC_\Bbbk$ to the 
homotopy category $h\mL_\Bbbk$. We sometimes use $\Longrightarrow$ to denote the descendant functor $\Des:\frC_\Bbbk \Longrightarrow \mL_{\Bbbk}$. The exposition given here regarding $\Des$ is only a sketch and we refer to \cite{Pa12} for a full treatment.

The general construction, developed in subsection \ref{subs2.2},
associates to a representation $\rho: \frg \to \End_{\Bbbk}(A)$ an object $(\cA_\rho,\hbox{}\cdot\hbox{}, K_\rho$) of
the category $\frC_\Bbbk$ such that the subalgebra $\CA^0_\rho$ is isomorphic to $A$
and the $0$-th
cohomology $H^0_K(\CA_\r)$ of the cochain complex $(\CA_\r, K_r)$ is isomorphic to the quotient $A/\CN_\r$.
Then, a period integral $C: A \to \Bbbk$  attached to $\rho$ gives a morphism $\cC$ up to homotopy (see Proposition
\ref{enhancetochain})
from the object $(\CA_\r,\hbox{}\cdot\hbox{}, K_\r)$ onto the initial object $(\Bbbk, \cdot,0)$ of the category
 $\frC_\Bbbk$, i.e. $\cC$ is a $\Bbbk$-linear map of degree $0$ from $\CA$ onto $\Bbbk$ such that $\cC\circ K =0$. We then realize the period map 
 $\cP^\r_{C}: A/\CN_\r \rightarrow \Bbbk$
by the following commutative diagram
$$
\xymatrix{
H_K(\cA_\r)\ar@{->}_{f}[dr] \ar[rr]^{\cP_\cC(0)} &&\Bbbk
\\
&\cA_\r\ar@{->}[ru]_{{\cC}} &&
}
$$
where $f$ is a cochain quasi-isomorphism from $H_{K}(\CA_\r)$, considered as a cochain complex with zero differential, into
the cochain complex $(\CA,K)_\r$ which induces the identity map on cohomology. Note that 
$\cP_\cC(0):= f\circ \cC:H_K(\CA_\r)\rightarrow \Bbbk$ depends
only on the cochain homotopy types of $f$ and $\cC$. Note also that $\cP_\cC(0)$ is a zero map on $H_K^i(\cA_\r)$
unless $i=0$ so that it can be identified with $\cP^\r_C$ via isomorphism between $H^0_K(\cA_\r)$ and $A/\CN_\r$.

Our general framework together with various propositions will allow us to enhance the above commutative diagram 
as follows;
$$
\xymatrix{
S\left(H_K(\cA_\r)\right)\ar@{..>}_{\underline{\w}}[dr] \ar@{..>}[rr]^{\underline{\kappa}} &&\Bbbk
\\
&S\left(\cA_\r\right)\ar@{..>}[ru]_{\underline{{\phi}}^{\cC}} &&
}
$$
where $S(V)$ is the super-commutative algebra of $V$ (see \scalg) and the dotted arrows are the following $L_\infty$-morphisms
\begin{enumerate}
\item
the $L_\infty$-morphism $\underline{{\phi}}^{\cC}=\phi^{\cC}_1,\phi^{\cC}_2,\cdots$ is the
descendant $\Des(\cC)$ of $\cC$ such that $\phi^{\cC}_1=\cC$,
\item
$\underline{\w}= \w_1,\w_2,\cdots$ is an $L_\infty$-quasi-isomorphism from $H_{K}(\CA_\r)$ considered as a zero $L_\infty$-algebra 
into the $L_\infty$-algebra $\Des(\cA_\r)$ such that $\w_1=f$.
\item
the $L_\infty$-morphism $\underline{\kappa}=\kappa_1,\kappa_2,\cdots$ is the composition $\underline{\phi}^\cC\bullet \underline{\w}$
in $\mL_\Bbbk$
such that $\kappa_1= \phi^{\cC}_1\circ \w_1=\cP_\cC(0):H_K(\CA_\r)\rightarrow \Bbbk $.
\end{enumerate}
Note that the $L_\infty$-morphism $\underline{{\phi}}^{\cC}=\Des(\cC)$ is determined by $\cC$, while
there are many different choices of   $L_\infty$-quasi-isomorphism $\underline{\w}$.
One of our main theorem shows that $L_\infty$-morphism $\underline{\kappa}:=\underline{\phi}^\cC\bullet \underline{\w}$ depends
only on the cochain homotopy type of $\cC$ and the $L_\infty$-homotopy type of $\underline{\w}$.

\subsection{Explicit description of the homotopy descendant functor} \label{subs3.1}


%
Here we shall give details about the \textit{homotopy descendant functor} from the homotopy category of $\frC_\Bbbk$ to the homotopy category of $\frL_\Bbbk$ of $L_\infty$-algebras over $\Bbbk$
by using the binary product as a crucial ingredient.

  When the representation space of $\rho$ has an associative and commutative binary product, $(\cA_\rho, K_\rho)$ attached to a Lie algebra representation $\rho$ also has a $\Bbbk$-algebra structure.
This implies that $(\cA_\rho, \cdot, K_\rho)$ and a period integral $\cC: (\cA_\rho,\cdot, K_\rho) \to (\Bbbk,\cdot, 0)$ is an object and a morphism of $\frC_\Bbbk$ respectively. 

An $L_\infty$-algebra structure on $\cA$ is a sequence of $\Bbbk$-linear maps $\underline \ell=\ell_1, \ell_2, \cdots$ such that 
$\ell_n: S^n(\cA) \to \cA$ satisfy certain relations (see Definition \ref{L2}) and an $L_\infty$-morphism from an $L_\infty$-algebra $(\cA, \underline \ell)$ to 
another $L_\infty$-algebra $(\cA', \underline \ell')$ is a sequence of $\Bbbk$-linear maps $\underline \phi=\phi_1, \phi_2, \cdots$ such that $\phi_n:S^n(\cA) \to \cA'$ satisfy certain relations (see Definition \ref{LM2}).
We use a variant of the standard $L_{\infty}$-algebra such that every 
$\Bbbk$-linear map $\ell_n$ for $n=1, 2, \cdots,$ is of degree 1. See subsection \ref{subs6.2} for our presentation of $L_{\infty}$-algebras, $L_\infty$-morphisms, and $L_\infty$-homotopies (suitable for the purpose of describing correlations), which is based on 
partitions of $\{1, 2, \cdots, n\}$,

As we already indicated, our algebraic analysis is based on two facts: the differential $K$ is not a $\Bbbk$-derivation of the 
(super-commutative) product of $\cA$ and the cochain map (period integral) $f: (\cA, \cdot, K) \to (\Bbbk, \cdot, 0)$ is not a $\Bbbk$-algebra map in general.
Therefore it is natural to propose the following definition for $\ell_1^K:\cA\to \cA$ and $\ell_2^K: S^2(\cA) \to \cA$:
\eqn\elltwo{
\eqalign{
\ell_1^K(x) =& K x
,\cr
\ell_2^K(x,y) =& K(x \cdot y) - Kx \cdot y - (-1)^{|x|} x \cdot Ky.
}
}
so that $\ell^K_2$ measures the failure of $K$ to be a derivation of the product. 
In the case of morphisms, we propose the following definition for $\phi^f_1:\cA \to \cA', \phi^f_2:S^2(\cA) \to \cA'$ in the same vein:
\eqn\ptwo{
\eqalign{
\phi^f_1(x) =& f (x)
,\cr
\phi^f_2(x,y) =& f(x\cdot y) - f(x) \cdot f(y), 
}
}
so that $\phi^f_2$ measures the failure of $f$ being an algebra map.
But then there would be many choices for how to measure higher failures\footnote{ 
A homotopy associative algebra (so called,  $A_\infty$-algebra) can be also used as a target category; we can construct a $A_\infty$-descendant functor from $\frC_\Bbbk$ to the category of $A_\infty$-algebras over $\Bbbk$, which even works if we drop the super-commutativity of the multiplication in an object of the category $\frC_\Bbbk$. The descendant functor
is a more general notion; we recently found that the pseudo character (a generalization of the trace of a group representation) 
is also a sort of a descendant which measures different higher failures.
But we limit our study
only to $L_\infty$-descendant functor in this article.
 }
, i.e. how to define $\ell^K_3, \ell^K_4, \cdots$ and $\phi^f_3, \phi^f_4, \cdots$ in a systematic way. We provide one particular way so that resulting $\ud \ell^K$ becomes
an $L_\infty$-algebra and $\ud \phi^f$ becomes an $L_\infty$-morphism in a functorial way; see Theorem \ref{Desc}.
%
For the definition of the descendant functor, let us set up some notation related to partitions. A partition $\pi= B_1 \cup B_2\cup \cdots$ of the set $[n]=\{1,2, \cdots, n\}$ is a decomposition of $[n]$ into a pairwise
disjoint non-empty subsets $B_i$, called blocks. Blocks are ordered by the minimum element of each block and each block is ordered by the ordering induced from the ordering of natural numbers. The notation $|\pi|$ means the number of blocks in a partition $\pi$ and $|B|$ means the size of the block $B$. If $k$ and $k'$ belong to the same block in $\pi$, then we use
the notation $k \sim_\pi k'$. Otherwise, we use $k \nsim_\pi k'$. Let $P(n)$ be the set of all partitions of $[n]$.
We refer to the appendix for more details regarding the partition $P(n)$ appearing in the following definition.

\bed \label{defdesc}
For a given object $(\cA, \cdot, K)$ in $\frC_\Bbbk$, we define $\Des\left(\cA, \cdot, K\right) = (\cA, \underline \ell^K)$, where $\underline \ell^K=\ell_1^K, \ell_2^K, \cdots$ is the family of linear maps $\ell_n^K: S^n(\cA) \to \cA$, inductively defined by the formula\footnote{This is
a sum over all $\pi \in P(n)$ which have a block $B_i$ of the given size, and moreover if there is more than one such block, then we sum over each of choice of such a block. Note that $\epsilon(\pi)$ depends on $\ud x= (x_1, \cdots, x_n)$, even though $\ud x$ is omitted for the sake of the notational simplicity. Such a dependence is explained in the appendix.}
\eqn\dpl{
K(x_1 \cdots x_n)=\sum_{\substack{\pi \in P(n)\\ |B_i|=n-|\pi|+1}}\!\!\!\! \epsilon(\pi,i) \cdot x_{B_1}
\cdots x_{B_{i-1}}\cdot  \ell^K(x_{B_i}) \cdot x_{B_{i+1}} \cdots x_{B_{|\pi|}},
}
for any homogeneous elements $x_1, x_2, \cdots, x_n \in \cA$. Here we use the the following notation:
\be
x_B&=& x_{j_1} \otimes \cdots \otimes x_{j_{r}} \text{ if }  B=\{j_1, \cdots, j_{r}\},\\
\ell(x_{B})&=&\ell_r(x_{j_1}, \cdots, x_{j_r}) \text{ if } B=\{j_1, \cdots, j_r\},\\
\epsilon(\pi, i) &=& \epsilon(\pi) (-1)^{|x_{B_1}|+\cdots+ |x_{B_{i-1}}|}.
\ee
 For a given morphism $f: (\cA, \cdot, K) \to (\cA', \cdot, K')$ in $\frC_\Bbbk$, we define a morphism $\Des(f)$ as the family $\underline \phi^f = \phi_1^f, \phi_2^f, \cdots$ constructed inductively as
\eqn\dlm{
f(x_1\cdots x_n) = \sum_{\pi \in P(n)} \epsilon(\pi) \phi^f (x_{B_1}) \cdots \phi^f (x_{B_{|\pi|}}),\quad n \geq 1,
}
where $\phi^f(x_B)=\phi^f_r(x_{j_1},\cdots, x_{j_r})$ if $B=\{j_1, \cdots, j_r\}$, $1 \leq j_1, \cdots, j_r \leq n$, for any homogeneous elements $x_1, \cdots, x_n \in \cA$. Here $\phi^f_n: S^n(\cA)\to \cA'$ is a $\Bbbk$-linear map defined on the super-commutative symmetric product of $\cA$. 
\eed

\begin{remark}
We will call the above $\Des\left(\cA, \cdot, K\right)$ and $\Des(f)$ a descendant $L_\infty$-algebra and a descendant $L_\infty$-morphism respectively. These descendant $L_\infty$-structures will appear in the study of the period integrals of smooth projective hypersurfaces. Note that not every $L_\infty$-algebra over 
$\Bbbk$ is a descendant $L_\infty$-algebra over $\Bbbk$; for example, an induced minimal $L_\infty$-algebra structure on the cohomology $H_K(\cA)$ of a descendant $L_\infty$-algebra $(\cA,\ud \ell^K)$ is always trivial; see Proposition \ref{zero}.
\end{remark}

According to the definition, we have \elltwo.
 For $n \geq 2$, the following holds:
\be
&&\ell_n^K(x_1, \cdots, x_{n-1}, x_n)=  \ell_{n-1}^K(x_1,\cdots, x_{n-2}, x_{n-1}\cdot x_n)\\
&&
-\ell_{n-1}^K(x_1, \cdots, x_{n-1}) \cdot x_{n} 
-(-1)^{|x_{n-1}|(1+|x_1|+\cdots + |x_{n-2}|)}  x_{n-1}\cdot \ell_{n-1}^K(x_1, \cdots, x_{n-2}, x_n).
\ee

If $\cA$ has a unit $1_\cA$ and $K(1_\cA)=0$, then one can also easily check that
\eqn\dl{
\ell_n^K(x_1, x_2, \cdots, x_n)=[[\cdots[[K, L_{x_1}],L_{x_2}], \cdots],L_{x_n}](1_\cA)
}
for any homogeneous elements $x_1, x_2, \cdots, x_n \in \cA$. Here $L_x:\cA\to\cA$
is left multiplication by $x$ and $[L,L']:=L\cdot L' - (-1)^{|L|\cdot|L'|}L' \cdot L \in \End_{\Bbbk}(\cA)$, where $|L|$ means the degree of $L$.

Unravelling the definition also shows \ptwo.
For $n \geq 2$, we have
\be
\phi_n^f(x_1, \cdots, x_n) = \phi_{n-1}^f (x_1, \cdots, x_{n-2}, x_{n-1}\cdot x_n) -
\sum_{\substack{\pi \in P(n), |\pi|=2 \\ n-1 \nsim_\pi n}} \phi^f (x_{B_1})\cdot \phi^f(x_{B_2}).
\ee

Let $\art_{\Bbbk}$ denote the category of unital $\bZ$-graded commutative Artinian local $\Bbbk$-algebras with residue field $\Bbbk$.
Let $\g \in (\mm_\ma\otimes \cA)^0$, where  $\ma\in \hbox{\it Ob}(\art_{\Bbbk})$, and let $\mm_\ma$ denote the unique maximal ideal of $\ma$; the tensor product $\mm_\ma\otimes \cA$ also has a natural induced $\bZ$-grading and $(\mm_\ma\otimes \cA)^n$ denotes the $\Bbbk$-subspace of homogeneous elements of degree $n$.
						
\bel \label{ls}
For every $\g \in (\mm_\ma \otimes \cA)^0$ and for any homogeneous element $\lambda \in \ma \otimes \cA$
whenever $\ma\in \hbox{\it Ob}(\art_{\Bbbk})$,
we have identities in $\ma \otimes \cA$;
\eqn\desdef{
K(e^\g -1) = e^\g \cdot L^K(\g),
}
where $L^K(\g)= \sum_{n\geq1}\frac{1}{n!} \ell^K_n (\g, \cdots, \g)$, and
\eqn\desderi{
K(\lambda\cdot e^\g)=
L^K_\g(\l) \cdot e^{\g}
+(-1)^{|\lambda|} \lambda\cdot K(e^{\g}-1),
}
where $L^K_\g(\l):=   {K}\lambda +{\ell}^{K}_2\big(\g, \lambda \big)
+\sum_{n=3}^\infty \frac{1}{(n-1)!}{\ell}^{K}_n\big(\g,\cdots,\g, \l \big).$

\eel
\begin{proof}
 Note that we use the following notation (see Definition \ref{shl});
\be
&&{\ell}^K_n\big(a_1\otimes v_1, \cdots,  a_n\otimes v_n\big) \\
&&=(-1)^{|a_1|+|a_2|(1+|v_1|) +\cdots + |a_n|(1+|v_1|+\cdots +|v_{n-1}|)}
 a_1\cdots a_n\otimes \ell^K_n\left(v_1,\cdots, v_n\right)
\ee
for $a_i \otimes v_i \in (\mm_\ma\otimes \cA)^0$ with $i = 1, 2, \cdots, n$.
Hecne the above infinite sum is actually a finite sum, since $\mm_\ma$ is a nilpotent $\Bbbk$-algebra.
Then the first formula follows from a simple combinatorial computation by plugging in $\g=\sum_{i=1}^na_i
\otimes v_i$; the reader may regard
\desdef\ as an alternative definition of the $L_\infty$-descendant algebra $(\cA, \ud \ell^K)$.

For the second equality, let $\ma=\Bbbk[\e]/(\e^2)$ be the ring of dual numbers, which is an object of $\art_\Bbbk$.
Denote the power series $e^X -1$ by $P(X)$ and let $P'(X)=e^X$ be the formal derivative of $P(X)$ with respect to $X$.
Then it is easy to check that $P(\g + \e \cdot \l) = P(\g) + (\e \cdot \l) \cdot P'(\g)$ for $\e \cdot \l \in (\ma \otimes \cA)^0$.
This says that 
$$
\eqalign{
P'(\g + \e \cdot \l) \cdot L^K(\g+ \e \cdot \l)&=
K\left( P(\g + \e \cdot \l) \right)  \quad \text{by}\  \desdef
\cr
&= K(P(\g)) +K\left(  (\e \cdot \l)\cdot P'(\g) \right)
\cr
&=P'(\g)\cdot L^K(\g) + K\left( (\e \cdot \l)\cdot P'(\g) \right) \quad \text{by}\  \desdef. 
}
$$
On the other hand, we have that 
$$
P'(\g + \e \cdot \l) \cdot L^K(\g+ \e \cdot \l) = P'(\g)\cdot L^K(\g)+(\epsilon \cdot \l) \cdot P'(\g)\cdot L^K(\g)
 + P'(\g)\cdot L^K_{\g}(\e \cdot \l)
$$
where $L^K_{\g}(\e \cdot \l) =   {K}(\e \cdot \l) 
+\sum_{n=2}^\infty \frac{1}{(n-1)!}{\ell}^{K}_n\big(\g,\cdots,\g, (\e \cdot \l)\big)$.
By comparison, we get that
\be
K\left(   (\e \cdot \l)\cdot P'(\g)  \right)
&=&(\e \cdot \l) \cdot P'(\g)\cdot L^K(\g)
 + P'(\g)\cdot L^K_{\g}(\e \cdot \l)\\
&=&L^K_{\g}(\e \cdot \l)\cdot  P'(\g) +(\e \cdot \l) \cdot K(P(\g)).
\ee
Then this implies the second identity according to our sign convention.
\end{proof}

\bel
The descendants $\underline \ell^K$ define an $L_\infty$-algebra structure over $\Bbbk$ on $\cA$.
\eel
\begin{proof}

For every $\g \in (\mm_\ma \otimes \cA)^0$, we consider $L^K(\g)= \sum_{n\geq1}\frac{1}{n!}\ell^K_n (\g, \cdots, \g)$.
Then \desdef\ says that
%
$$
K (e^{\g}-1)=e^{\g}\cdot L^K(\g)=L^K(\g) \cdot e^{\g}.
$$
Applying $K$ to the above, we obtain that
\be
0=
{K}\left( L^K(\g) \cdot e^{\g}\right)
&=&
L_\g^K(L^K(\g)) \cdot e^{\g}
- L^K(\g)\cdot K(e^{\g}-1) \\
&=& 
L_\g^K(L^K(\g)) \cdot e^{\g},
\ee
where we have used $K^2=0$ for the $1$st equality. The $2$nd equality follows from Lemma \ref{ls} and the $3$rd equality results from the fact that
$L^K(\g)\cdot {K}(e^{\g}-1)=L^K(\g)^2\cdot e^{\g}$
vanishes,
since $L^K(\g)^2=0$ by the super-commutativity of the product (note that $L^K(\g)$ has degree 1).
Hence the following expression
\be
\chi(\g):=L_\g^K(L^K(\g)) =  {K}(L^K(\g)) +{\ell}^{K}_2\big(\g,L^K(\g) \big)
+\sum_{n=3}^\infty \frac{1}{(n-1)!}{\ell}^{K}_n\big(\g,\cdots,\g, L^K(\g)\big)
\ee
vanishes  for every  $\g \in (\mm_\ma\otimes \cA)^0$ whenever $\ma \in \hbox{\it Ob} (\art_{\Bbbk} )$.
We then consider a scaling $\g\rightarrow \l\cdot \g$, $\l \in \Bbbk^*$, and the corresponding decomposition
$\chi(\g)= \chi_1(\g)+ \chi_2(\g)+ \chi_3(\g)+\cdots$ such that $\chi_n(\l\cdot\g)=\l^n \chi_n(\g)$,
i.e., we have for $n\geq 1$
\be
\chi_n(\g)
=\sum_{k=1}^n \frac{1}{(n-k)! k!} {\ell}^{K}_{n-k+1}\left({\ell}^{K}_k(\g,\cdots,\g),\g,\cdots,\g\right).
\ee
It follows that $\chi_n(\g)=0$, for all $n\geq 1$ and for all $\g\in (\mm_\ma \otimes \cA)^0$
and every $\ma \in\hbox{\it Ob}(\art_{\Bbbk})$. 
Hence $\big(\cA, \underline{\ell}^{K},1_\cA\big)$ is
an $L_\infty$-algebra over $\Bbbk$ by Definition \ref{shl}. 
\end{proof}

\bel \label{morkey}
For every $\g \in (\mm_\ma \otimes \cA)^0$ and for any homogeneous element $\lambda \in \ma \otimes \cA$,
We have identities in $\ma \otimes \cA$;
\eqn\kexp{
f(e^{\g}-1) = e^{\Phi^f(\g)} -1,  
}
where $\Phi^f(\g)=\sum_{n\geq 1} \frac{1}{n!}\phi^f_n (\g, \cdots, \g)$, and
\eqn\desmor{
f(\lambda \cdot e^{\g}) =  \Phi^f_\g(\l)  \cdot e^{\Phi^f(\g)},
}
where $\Phi^f_\g(\l):=  {{\phi}}^{f}_1\left({\lambda}\right)
+\sum_{n=2}^\infty \frac{1}{(n-1)!}{{\phi}}^{f}_n\big(\g,\cdots,\g, {\lambda}\big).$
\eel
\begin{proof}

Recall the sign convention from Definition \ref{shlm};
\be
&&{\phi}_n\big(a_1\otimes v_1, \cdots, a_n\otimes v_n\big)\\
&&=(-1)^{|a_2||v_1| +\cdots + |a_n|(|v_1|+\cdots +|v_{n-1}|)}
 a_1\cdots a_n\otimes \phi_n\left(v_1,\cdots,v_n\right).
\ee
for $a_i \otimes v_i \in (\mm_\ma \otimes \cA)^0$ with $i =1, 2, \cdots, n$, whenever $\ma \in \hbox{\it Ob}(\art_{\Bbbk})$.
The first equality is a simple combinatorial reformulation of Definition \ref{defdesc} with the above sign convention. 
We leave this
as an exercise (plug in $\g=\sum_{i=1}^na_i
\otimes v_i$); the reader can regard \kexp\ as an alternative definition of a descendant $L_\infty$-morphism $\ud \phi^f$.

Again, let $\ma=\Bbbk[\e]/(\e^2) \in  \hbox{\it Ob}(\art_{\Bbbk})$ be the ring of dual numbers.
Denote the power series $e^X -1$ by $P(X)$ and let $P'(X)=e^X$ be the formal derivative of $P(X)$ with respect to $X$.
Note that $P(\g + \e \cdot \l) = P(\g) + (\e \cdot \l) \cdot P'(\g)$ for $\e \cdot \l \in (\ma \otimes \cA)^0$.
Inside $\ma \otimes \cA$ we have that
$$
\eqalign{
f( P(\g)) + f\left( (\e \cdot \l) \cdot P'(\g) \right)   =f \left(  P(\g + \e \cdot \l)  \right) 
 &=   P\left(\Phi^f(\g + \e \cdot \l) \right)
 \cr
 &= P \left(\Phi^f(\g) + \Phi^f_\g(\e \cdot \l)\right)
 \cr
 &= P(\Phi^f(\g)) + P'(\Phi^f(\g)) \cdot  \Phi_\g^f(\e \cdot \l). 
 \cr
}
$$
Then this finishes the proof of \desmor.
\end{proof}

\bel
The descendants $\underline \phi^f$ define an $L_\infty$-morphism from $(\cA, \underline \ell^K)$ into $(\cA', \underline \ell^{K'})$.
\eel

\begin{proof}

Applying $K'$ to the relation 
${f}\left( e^{\g}-1\right) = e^{\Phi^{f}(\g)}-1$ in \kexp,
we obtain that
$$
{f}\left({K} (e^{\g}-1)\right) ={K}' (e^{{\Phi^{f}(\g)}}-1),
$$
where we have used $K'f = fK$. Then \desdef\ implies that
\eqn\gtya{
{f}\left({L^K(\g)}\cdot e^{\g}\right) =
{L^{K'}}\left(\Phi^{f}(\g)\right)\cdot e^{{\Phi^{f}(\g)}}.
}
Then \gtya\ combined with \desmor\ says that
\eqn\doum{
\Phi^f_\g (L^K(\g)) \cdot   e^{{\Phi^{f}(\g)}} 
={L^{K'}}\left(\Phi^{f}(\g)\right)\cdot e^{{\Phi^{f}(\g)}}.
}
Therefore the following expression
$$
{\zeta}(\g):=
\left({{\phi}}^{f}_1\left({L^K(\g)}\right)
+\sum_{n=2}^\infty \frac{1}{(n-1)!}{{\phi}}^{f}_n\big(\g,\cdots,\g, {L^K(\g)}\big)
\right)
-{L^{K'}}\left(\sum_{n=1}^\infty \frac{1}{n!} {{\phi}}^{f}_n(\g,\cdots, \g)\right)
$$
{\it vanishes} 
 for  every $\g\in \left(\mm_\ma\otimes \cA\right)^0$
whenever $\ma \in \hbox{\it Ob}(\mA_{k})$. 
Then we consider a scaling $\g\rightarrow \l\cdot \g$, $\l \in \Bbbk^*$, and the corresponding decomposition
${\zeta}(\g)= {\zeta}_1(\g)+ {\zeta}_2(\g)+ {\zeta}_3(\g)+\cdots$ such that 
${\zeta}_n(\l\cdot\g)=\l^n {\zeta}_n(\g)$,
where
\be
{\zeta}_n(\g)&=&
\sum_{j_1+ j_2 =n}
\frac{1}{j_1! j_2!} {{\phi}^{f}}_{j_1+1}\left({\ell}^{K}_{j_2}\big(\g,\cdots,\g\big), \g,\cdots, \g\right)\\
&&-
\sum_{r=1}^{\infty}\sum_{j_1+\cdots + j_r =n}\frac{1}{r!}\frac{1}{j_1!\cdots j_r!}
{\ell}^{K'}_r \left({{\phi}}^{f}_{j_1}(\g,\cdots,\g), \cdots, {{\phi}}^{f}_{j_r}(\g,\cdots,\g)\right).
\ee
It follows that ${\zeta}_n(\g)=0$, for all $ n\geq 1$ and all $\g\in (\mm_\ma \otimes \cA)^0$
whenever $\ma \in\hbox{\it Ob}(\art_{\Bbbk})$. Thus, by using Definition \ref{shlm}, it follows that the sequence $\underline \phi^f$ defines an $L_\infty$-morphism between $L_\infty$-algebras.
 \end{proof}

\bel
Let $f: (\cA, K) \to (\cA',K')$ and $f': (\cA',K') \to (\cA'', K'')$ be two morphisms in $\frC_\Bbbk$. Then we have
\be
\underline \phi^{f' \circ f} = \underline \phi^{f'} \bullet \underline \phi^{f},
\ee
where $\bullet$ is the composition of $L_\infty$-morphisms (see Definition \ref{compl}).
\eel

\begin{proof}
Consider ${\g} \in (\mm_\ma\otimes \cA)^0$ whenever
$\ma \in\hbox{\it Ob}(\art_{\Bbbk})$. Then \kexp\ implies that
$$
e^{\Phi^{f'}\big( {\Phi}^{f}({\g})\big)}-1=f' (e^{\Phi^f(\g)}-1)=({f}'\circ{f})\left(e^{{\g}}-1\right) =  e^{{\Phi}^{{f}'\circ{f}}({\g})}-1.
$$
Therefore, if we set ${X}({\g}) =  {\Phi}^{{f}'}\big( {\Phi}^{{f}}({\g})\big)$
and ${{Y}} ({\g}) =   {\Phi}^{{f}'\circ{f}}({\g})$, then we have
${X}({\g}) ={Y}({\g})$ for every ${\g} \in (\mm_\ma\otimes \cA)^0$.
Consider the decomposition ${X}({\g})={X}_1({\g})+{X}_2({\g})+\cdots$
and ${{Y}}({\g})={{Y}}_1({\g})+{{Y}}_2({\g})+\cdots$,
where ${X}_n(\l\cdot{\g})=\l^n\cdot {X}_n({\g})$ and ${{Y}}_n(\l\cdot{\g})=\l^n\cdot {{Y}}_n({\g})$, $\l \in \Bbbk^*$. 
Then ${X}_n({\g})={Y}_n({\g})$ for all $n\geq 1$.
The equality \kexp\ implies that
\be
{{Y}}_n({\g})&=& \frac{1}{n!}{\phi}^{{f}'\circ {f}}_n({\g},\cdots,{\g}).
\ee
The direct computation gives us that
\be
{X}_n({\g}) &=& \sum_{r=1}^{\infty}\sum_{j_1+\cdots+ j_r=n}\frac{1}{r!}\frac{1}{j_1!\cdots j_r!}{\phi}^{{f}'}_r\left(
{\phi}_{j_1}({\g},\cdots,{\g}), \cdots ,{\phi}_{j_r}({\g},\cdots,{\g})\right)\\
&\equiv& \frac{1}{n!}\left(\underline{{\phi}}^{{f}'}\bullet \underline{{\phi}}^{{f}}\right)_n({\g},\cdots,{\g}) 
\ee
We hence conclude that 
${\phi}^{{f}'\circ {f}}_n({\g},\cdots,{\g})
=\left(\underline{{\phi}}^{{f}'}\bullet \underline{{\phi}}^{{f}}\right)_n({\g},\cdots,{\g})$
for all $n\geq 1$. It follows that $\underline{{\phi}}^{{f}'\circ {f}}=\underline{{\phi}}^{{f}'}\bullet \underline{{\phi}}^{{f}}$.
\end{proof}

Now we turn our attention to cochain homotopies in $\frC_\Bbbk$ and $L_\infty$-homotopies in $\frL_\Bbbk$.
\bed \label{homoC}
Two morphisms $f:(\cA,\cdot, K) \to (\cA',\cdot, K')$ and $\tilde f:(\cA,\cdot, K) \to (\cA',\cdot, K')$ in $\frC_\Bbbk$ are called homotopic, denoted by $f \sim \tilde f$, if there is a polynomial family $F(\t):(\cA,\cdot, K) \to (\cA',\cdot, K')$ in $\t$ of morphisms with $F(0)=f$ and $F(1)=\tilde f$ satisfying
\eqn\cohy{
\pa{\t} F(\t) = K' \sigma(\t) + \sigma(\t) K
}
for some polynomial family $\sigma(\t)$ in $\t$ belonging to $\Hom (\cA,\cA')^{-1}$.
In this case, we call $\sigma(\t)$ a homotopy between $f$ and $f'$.
\eed

\bed
Given two morphisms $f:(\cA,\cdot, K) \to (\cA',\cdot, K')$ and $\tilde f:(\cA,\cdot, K) \to (\cA',\cdot, K')$ in $\frC_\Bbbk$,
and a homotopy $\sigma=\sigma(\t)$ and a polynomial family $F=F(\t)$ as in \cohy, we define
the descendant homotopy $\l_n=\l_n^{F,\sigma} \in \Hom(S^n\cA, \cA')^{-1}$ for each $n \geq 1$, by the following formula
\eqn\hde{
  \sigma(e^\g-1) =e^{\Phi^F(\g)} \cdot \La^{F,\sigma} (\g),
}
where $\La^{F,\sigma} (\g):=\sum_{n=1}^\infty \frac{1}{n!} \l_n(\g, \cdots, \g)$ and $\g \in (\mm_\ma \otimes \cA)^0$ whenever $\ma \in \hbox{\it Ob}(\art_{\Bbbk})$.
\eed

Compare the formula \hde\ with the formula \desdef\ for the descendant $L_\infty$-algebra
and the formula \kexp\ for the descendant $L_\infty$-morphism.

\bel \label{deshomotopy}
Let $f, \tilde f: (\cA, K, \cdot) \to (\cA', K', \cdot)$ be morphisms which are homotopic in $\frC_\Bbbk$ in the sense of Definition \ref{homoC}. For every $\g \in (\mm_\ma \otimes \cA)^0$ and any homogeneous element $\mu
\in \ma \otimes \cA$, we have the following identity in $\ma \otimes \cA$;
\eqn\hho{
\sigma\left( e^\g \cdot \mu \right) = e^{\Phi^{F}(\g)} \cdot \La_\g^{F,\sigma} (\mu) + 
\Phi^F_\g(\mu) \cdot \La^{F,\sigma}(\g),
}
where $\Lambda_\g^{F,\sigma} (\mu) = \l_1(\mu)+ \sum_{n=2}^\infty \frac{1}{(n-1)!}{\l}_n\left({\g},\cdots,{\g}, \mu\right)$.

\eel

\begin{proof} 
Again, let $\ma=\Bbbk[\e]/(\e^2) \in  \hbox{\it Ob}(\art_{\Bbbk})$ be the ring of dual numbers.
Denote the power series $e^X -1$ by $P(X)$ and let $P'(X)=e^X$ be the formal derivative of $P(X)$ with respect to $X$.
For $\e \cdot \mu \in (\ma \otimes \cA)^0$, we have the following identities in $\ma \otimes \cA$
$$
\eqalign{
\sigma\left( P(\g) \right)+ \sigma \left( P'(\g) \cdot \e \mu \right) 
&= \sigma \left( P(\g + \e \cdot \mu)  \right)
\cr
&= P'\left(\Phi^F(\g+ \e \cdot \mu)\right) \cdot \La^{F,\sigma}(\g + \e \cdot \mu)
\cr
& = \left( P'\left(\Phi^F(\g)\right) +P'\left(\Phi^F_\g( \e \cdot \mu)\right)-1 \right) 
\cdot \left(  \La^{F,\sigma}(\g) + \La_\g^{F,\sigma} (\e \cdot \mu) \right)
\cr
& = \sigma\left( P(\g) \right) + P'\left(\Phi^F_\g( \e \cdot \mu)\right) \La^{F,\sigma}(\g)
+ P'\left(\Phi^F(\g)\right)\La_\g^{F,\sigma} (\e \cdot \mu) -  \La^{F,\sigma}(\g).
}
$$
This translates into
$$
\sigma \left( P'(\g) \cdot \e \mu \right) =  P'\left(\Phi^F_\g( \e \cdot \mu)\right) \La^{F,\sigma}(\g)
+ P'\left(\Phi^F(\g)\right)\La_\g^{F,\sigma} (\e \cdot \mu) -\La^{F,\sigma}(\g).
$$
Then \hho\ follows from this computation combined with our sign convention.
\end{proof}

\bel
The descendants of morphisms which are homotopic in $\frC_\Bbbk$ are $L_\infty$-homotopic in the sense of definition \ref{shlh}. 
\eel
\begin{proof}

By Definition \ref{homoC}, there is a polynomial family $F(\t)$ in $\t$ of morphisms with $F(0)=f$ and $F(1)=\tilde f$ satisfying
\eqn\hhmm{
\pa{\t} F(\t) = K' \sigma(\t) + \sigma(\t) K
}
for some polynomial family $\sigma(\t)$ in $\t$ belonging to $\Hom (\cA,\cA')^{-1}. $ 
 Recall that
\be
\Phi^F(\g)=\sum_{n\geq 1} \frac{1}{n!}\phi^F_n (\g, \cdots, \g)\quad \text{and}\quad  F(e^\g-1)=e^{\Phi^F(\g)}-1,
\ee
where $\ud \phi^F = \ud \phi^{F(\t)}= \phi^F_1, \phi^F_2, \cdots$ is the descendant of $F=F(\tau)$.
Then the identity \hhmm\ implies that
\eqn\hhho{
\prt{}{\t} e^{{\Phi}^{{F}}({\g})}= K' \sigma(e^\g-1) + \sigma K(e^\g-1)=K' \left( P'(\Phi^F(\g)) \cdot \La^{F,\sigma} (\g) \right) + \sigma(P'(\g) \cdot L^K(\g)  ),
}
where $F=F(\t)$ and $\sigma=\sigma(\t)$. 
The formulas \desderi\ and \hho\ say that the right hand side of \hhho\ is the same as
(with cancellation of middle terms due to \doum)
$$
\eqalign{
L^{K'}_{\Phi^F(\g)} \left(  \La^{F,\sigma} (\g) \right) \cdot P'(\Phi^F(\g))  + P'( \Phi^F(\g))\cdot \La^{F,\sigma}_\g (L^K(\g)).
}
$$
The left hand side of \hhho\ is the same as $
P'\left(  \Phi^F (\g) \right) \cdot \prt{}{\t} \Phi^{F}(\g).$ Therefore we eventually get 
\eqn\homotive{
 \prt{}{\t} \Phi^{F}(\g)=L^{K'}_{\Phi^F(\g)} \left(  \La^{F,\sigma} (\g) \right)  +  \La^{F,\sigma}_\g (L^K(\g)).
}
%
%
%
 Decomposing the equality \homotive\ by $\Bbbk^*$-action $ \g \rightarrow a \g,  a \in \Bbbk^*$, we have
the following form of the flow equation appearing in Definition \ref{shlh} of $L_\infty$-homotopy of $L_\infty$-morphisms: for every $ {\g} \in (\mm_{\ma}\otimes {\cA})^0$ whenever $\ma \in \hbox{\it Ob}(\art_{\Bbbk})$
\be
&&\prt{}{\t}{\phi^F_n}({\t})(\g, \cdots, \g)\\
&=&
\sum_{k=1}^n\sum_{j_1+\cdots + j_r =n-k}\frac{1}{k!}\frac{1}{r!}\frac{1}{j_1!\cdots j_r!}
\ell^{{K}'}_{r+1} \left({\phi^F_{j_1}}({\g},\cdots,{\g}), \cdots, {\phi^F_{j_r}}({\g},\cdots,{\g}),{{\l}}_k({\g},\cdots,{\g})\right)\\
&&+ \sum_{j_1+ j_2 =n}
\frac{1}{j_1! j_2!} {{\l}}_{j_1+1}\left({\g},\cdots, {\g},\ell^{{K}}_{j_2}({\g},\cdots,{\g})\right),
\ee
where $\phi^F_n(0)= \phi_n^f$ and $\phi^F_n(1)=\phi_n^{\tilde f}$, for all $n \geq 1$.
\end{proof}

\ber
The explicit description of $L_\infty$-homotopies
in Definition \ref{shlh}, which is hard to come up with in the literature,
is motivated by our formalism of the descendant functor: we have defined $L_\infty$-homotopy (motivated by 
the derivation of the equality \homotive\ from the natural notion of cochain homotopy $\sigma$) so that $\Des$
becomes a homotopy functor.
\eer

Let $h\frC_\Bbbk$ be the homotopy category of $\frC_\Bbbk$, i.e. objects in $h\frC_\Bbbk$ are the same as those in $\frC_\Bbbk$ and  morphisms in $h\frC_\Bbbk$ are homotopy classes of morphisms in $\frC_\Bbbk$ in the sense of Definition \ref{homoC}.
We also define the category $\frL_\Bbbk$ (respectively, $h\frL_\Bbbk$) to be the (respectively, homotopy) category of $L_{\infty}$-algebras over $\Bbbk$ in the same way by using Definition \ref{shlh} for $L_\infty$-homotopy. 
If we combine all of the above lemmas, we have the following result.
\bet  \label{Desc}
The assignment $\Des$ is a homotopy functor from the homotopy category $\frH\frC_\Bbbk$ to the homotopy category $\frH\frL$, which we call the descendant functor.
\eet

The functor $\Des$ is faithful but not fully faithful; two non-isomorphic objects in $\frC_\Bbbk$ can give isomorphic $L_\infty$-algebras under $\Des$, the trivial $L_\infty$-algebra $(V, \ud 0)$ can not be a descendant $L_\infty$-algebra (unless $V$ has a commutative and associative binary product), and an arbitrary $L_\infty$-morphism between descendant $L_\infty$-algebras need not be a descendant $L_\infty$-morphism.


\subsection{Cohomology of a descendant $L_\infty$-algebra} \label{subs3.2}

Here we prove that a descendant $L_\infty$-algebra is {formal} in the sense of Definition \ref{smoothformal}.
By Theorem \ref{transfer}, on the cohomology of an $L_{\infty}$-algebra $(V,\underline \ell)$ there is a minimal $L_{\infty}$-algebra 
structure $(H, \underline \ell^H)$ together with an $L_{\infty}$-morphism $\underline \varphi$ from $H$ to $V$. If an $L_{\infty}$-algebra $(V, \underline \ell)$
is a descendant $L_{\infty}$-algebra, i.e. $(V, \underline \ell)=(\cA, \underline \ell^K)=\Des(\cA,\cdot, K)$ for some object $(\cA, \cdot, K)$ in $\frC_\Bbbk$, then this minimal $L_{\infty}$-algebra structure on the cohomology $H_K$ of $(\cA,K)$ is trivial.

\bep \label{zero}
Let $(\cA,\cdot, K)$ be an object of $\frC_\Bbbk$ and let $(\cA, \underline \ell^K)$ be its descendant $L_{\infty}$-algebra. Then the
minimal $L_{\infty}$-algebra structure on 
the cohomology $H_K$ of $(\cA, \underline \ell^K)$ is trivial, i.e. $\ell_2^{H_K}=\ell_3^{H_K}=\cdots=0$, and there is an $L_\infty$-quasi-isomorphism 
$\underline{\varphi}^H$ from $(H_K, \ud 0)$ into $(\cA, \ud \ell^K)$.
\eep

\begin{proof}

Let $H=H_K$ for simplicity. 
Let $f: H\rightarrow \cA$ be a cochain quasi-isomorphism between the cochain complexes $(H,0)$ and $(\cA,K)$, which induces the identity map on $H$. Note that $f$ is defined up to cochain homotopy.
We remark that choosing $f$ amounts to choosing a splitting of the short exact sequence
$$
0 \to \Im K^{-i} \to \Ker K^i \to  H^i \to 0
$$
for each $i$.
 Let $h: \cA \rightarrow H$ be a homotopy inverse
to $f$ such that $h\circ f=I_H$ and $f\circ h= I_\cA + K\b +\b K$ for some $\Bbbk$-linear map $\b$
from $\cA$ into $\cA$ of degree $-1$. Let $1_H = h(1_{\cA})$.
We need to establish that (i) the minimal $L_\infty$-structure $\underline{\n}=\n_2,\n_3,\cdots$ on 
$H$ is trivial (we use the notation $\n_k=\ell_k^{H_K}$), i.e., $\n_2=\n_3=\cdots =0$, (ii) there is a family 
$\underline{\varphi}^H=\varphi_1^H,\varphi_2^H,\cdots$, where $\varphi_1^H=f$ and $\varphi_k^H$ is a $\Bbbk$-linear map from $S^k(H)$ into $\cA$ of degree $0$ for all $k\geq 2$ such that, for all homogeneous
elements $a_1,\cdots, a_n$ of $H$ and for all $n\geq 1$
\be
\sum_{\pi \in P(n)}
\ep(\pi)
{\ell}^{K}_{|\pi|}\Big(
{\varphi^H}\big(a_{B_1}\big), 
\cdots,{\varphi^H}\big( a_{B_{|\pi|}}\big)
\Big)
&=0.
\ee
We  shall use mathematical induction.
 Set $\varphi_1^H= f$.
Then we have
\be
K \varphi_1^H=0.   
\ee
Define a $\Bbbk$-linear map $L_2: S^2(H)\rightarrow \cA$ of degree $1$
such that, for all homogeneous elements
$a_1,a_2 \in H$,
$$
L_2\big(a_1,a_2\big):=\ell_2^K\big(\varphi_1^H(a_1),\varphi_1^H(a_2)\big).
$$
Then $\Im L_2 \subset \Ker K\cap \cA$ by the definition of $\ell_2^K$ and thus we can 
define $\n_2: S^2(H)\rightarrow H$ by
$\n_2\big(a_1,a_2\big):=h\circ  L_2\big(a_1,a_2\big)$.
On the other hand, it follows that $h\circ  L_2\big(a_1,a_2\big)=0$ for every homogeneous elements
$a_1,a_2 \in H$, because we have
\be
L_2\big(a_1,a_2\big)
&=&K \big(\varphi^H_1(a_1)\cdot\varphi^H_1(a_2)\big)
-K\varphi^H_1(a_1)\cdot\varphi^H_1(a_2)-(-1)^{|a_1|}\varphi^H_1(a_1)\cdot K\varphi^H_1(a_2)\\
&=&K \big(\varphi^H_1(a_1)\cdot\varphi^H_1(a_2)\big).
\ee
 Hence $\n_2=0$. From $h\circ  L_2\big(a_1,a_2\big)=0$, we have
$0=f\circ h\circ  L_2\big(a_1,a_2\big)=L_2\big(a_1,a_2\big) + K \varphi^H_2(a_1, a_2)$
where $\varphi^H_2\big(a_1,a_2\big):=\b\circ L_2\big(a_1, a_2\big)$
is a $\Bbbk$-linear map from $S^2(H)$ to $\cA$ of degree $0$.
Hence we have, for all homogeneous $a_1,a_2 \in H$,
\be
K \varphi^H_2\big(a_1,a_2\big)+\ell_2^K\big(\varphi^H_1(a_1),\varphi^H_1(a_2)\big)=0. 
\ee

Fix $n \geq 2$ and assume that $\n_2=\cdots=\n_{n}=0$
and there is a family  $\phi^{[n]}=\varphi^H_1,\varphi^H_2, \cdots, \varphi^H_n$,
where $\varphi^H_k$ is a linear map $S^k(H)\rightarrow \cA$ of degree $0$ for $1\leq k\leq n$,
$\varphi^H_1=f$ and 
such that, for all $1\leq k\leq n$,
\eqn\obtion{
\sum_{\pi \in P(k)}
\ep(\pi)
{\ell}^{K}_{|\pi|}\Big(
\phi^{[n]}\big(a_{B_1}\big), 
\cdots,\phi^{[n]}\big( a_{B_{|\pi|}}\big)
\Big)
=0.
}

Define the linear map $L_{n+1}: S^{n+1}(H)\rightarrow \cA$ of degree $1$ by
\be
L_{n+1}\big(a_1,\cdots, a_{n+1}\big)
&:=&\sum_{\substack{\pi \in P(n+1)\\ |\pi|\neq 1}}
\ep(\pi)
{\ell}^{K}_{|\pi|}\Big(
\phi^{[n]}\big(a_{B_1}\big), 
\cdots,\phi^{[n]}\big( a_{B_{|\pi|}}\big)\\
&=&
K\sum_{\substack{\pi \in P(n+1)\\|\pi|\neq 1}}
\ep(\pi)
\phi^{[n]}\big(a_{B_1}\big)\cdots\phi^{[n]}\big( a_{B_{|\pi|}}\big)
\Big).
\ee
Then $\Im L_{n+1} \subset \Ker K \cap \cA$, and 
\eqn\tecc{
h\circ L_{n+1}\big( a_1,\cdots, a_{n+1}\big)=0.
}
Define the linear map $\n_{n+1}:S^{n+1}(H)\rightarrow H$ of degree $1$ by
$$
\n_{n+1}\big( a_1,\cdots,a_{n+1}\big):= h\circ L_{n+1}\big( a_1,\cdots, a_{n+1}\big).
$$
 It follows that, by the assumption,
$\n_2=\cdots=\n_{n}=\n_{n+1}=0$.
By applying the map $f: H \rightarrow \cA$ to \tecc, we obtain
$$
K\circ \b\circ L_{n+1}\big( a_1,\cdots, a_{n+1}\big)+ L_{n+1}\big( a_1,\cdots, a_{n+1}\big)=0
$$
Set $\varphi^H_{n+1}= \b\circ L_{n+1}: S^{n+1}(H)\rightarrow \cA$,
which is a $\Bbbk$-linear map of degree $0$.
 Hence
$$
 L_{n+1}\big( a_1,\cdots, a_{n+1}\big)
+K\varphi^H_{n+1}\big( a_1,\cdots, a_{n+1}\big)=0.
$$
If we set 
$\phi^{[n+1]}=\varphi^H_1, \cdots, \varphi^H_n, \varphi^H_{n+1}$, then the above identity can be rewritten as follows:
$$
\sum_{\substack{\pi \in P(n+1)}}
\ep(\pi)
{\ell}^{K}_{|\pi|}\Big(
\phi^{[n+1]}\big(a_{B_1}\big),
\cdots,\phi^{[n+1]}\big( a_{B_{|\pi|}}\big)
\Big)=0.
$$
This finishes the proof.
 \end{proof}

\subsection{Deformation functor attached to a descendant $L_\infty$-algebra} \label{subs3.3}

In general, one can always consider a deformation functor associated to an $L_\infty$-algebra.
Here we consider a deformation problem attached to the 
descendant $L_\infty$-algebra $(\cA, \ud \ell^K)$ of an object $(\cA, \cdot, K)$ in $\frC_\Bbbk$. 
Recall that $\art_{\Bbbk}$ is the category of $\bZ$-graded commutative artinian local $\Bbbk$-algebras with residue field $\Bbbk$. 
Let $\ma \in \hbox{\it Ob}(\art_{\Bbbk})$, and  $\mm_\ma$ denote the maximal ideal of $\ma$.
\bed\label{hhh}
Let $\G, \tilde \G \in (\mm_\ma \otimes  \cA)^0$ such that 
\be
K(e^\G-1) =0=K(e^{\tilde \G}-1).
\ee
  Then we say that $\G$ is homotopy equivalent (or gauge equivalent) to $\tilde \G$, denoted by $\G \sim \tilde \G$, if there is a one-variable polynomial solution $\G(\t) \in (\mm_\ma \otimes \cA)^0[\t]$ with $\G(0) = \G$ and $\G(1) =\tilde \G$ to the following flow equation
 \be
 \prt{}{\tau}   e^{\G(\t)} = K \left(\l(\tau) \cdot e^{\G(\tau)}\right)
 \ee
 for some one-variable polynomial $\l(\tau) \in (\mm_\ma \otimes \cA)^{-1}[\t]$.
\eed
Note that the above infinite sum is actually a finite sum, since $\mm_\ma$ is a nilpotent $\Bbbk$-algebra.
We used the product structure on $\cA$ to define the homotopy equivalence. Note that
\be
K(e^\G-1)=0 \quad \text{is equivalent to}\quad \sum_{n=1}^{\infty} \frac{1}{n!}\ell_n^K(\G, \cdots,\G)=0
\ee
by \desdef.
Define a covariant functor $\mathfrak{Def}_{(\cA,\ud \ell^K)}$ from $\art_{\Bbbk}$ to $\hbox{\bf Set}$, where $\hbox{\bf Set}$ is the category of sets, by
\eqn\dfuctor{
\ma \mapsto \mathfrak{Def}_{(\cA,\ud \ell^K)}(\ma)=\Big\{\G_\ma \in (\mm_\ma \otimes  \cA)^0  \ : \  \sum_{n=1}^{\infty} \frac{1}{n!}\ell_n^K(\G_\ma, \cdots,\G_\ma)=0\Big \}/\sim
}
where $\sim$ is the equivalence relation given in Definition \ref{hhh}. One can send 
a morphism in $\art_{\Bbbk}$ to a morphism in $\hbox{\bf Set}$ in an obvious way. Let $\widehat\de$ be the extension of $\de$ to the category of $\bZ$-graded complete commutative noetherian local $\Bbbk$-algebras with residue field $\Bbbk$; see \cite{Har10}, p 83. We want to study this deformation functor. Any $L_\infty$-structure on $H_K=H_K(\cA)$, induced from the descendant $(\cA, \ud \ell^K)$ via Theorem \ref{transfer}, is trivial by Proposition \ref{zero}, i.e. $(\cA, \ud \ell^K)$ is {formal} (see Definition \ref{smoothformal}). 
Theorem \ref{transfer} also implies that there is an $L_\infty$-quasi-isomorphism $\underline \varphi^H= \varphi_1^H, \varphi_2^H, \cdots$ from $(H_K, \underline 0) \to (\cA, \underline \ell^K)$, i.e. $\varphi_m^H$ is a $\Bbbk$-linear map from $S^m(H_K)$ into $\cA$ of
degree 0 for all $m\geq 1$ such that
\be
\sum_{\pi \in P(n)} \epsilon(\pi) \ell^K_{|\pi|} \big(\varphi^H(a_{B_1}), \cdots, \varphi^H(a_{B_{|\pi|}})\big)=0
\ee
for every set of homogeneous elements $a_1, \cdots, a_n$ of $H_K$ and for all $n\geq 1$. 
In fact, the $L_\infty$-homotopy types of $L_\infty$-morphisms from $(H_K,\ud 0)$ into $(\cA, \ud\ell^K)$ are classified by $\widehat\de(\widehat{SH})$. To be more precise, we have the following results:

\bep \label{defmoduli}
 Assume that $H_K=H_K(\cA)$ is finite dimensional over $\Bbbk$. 

(a) The deformation functor $\mathfrak{Def}_{(\cA,\ud \ell^K)}$ is pro-representable by $\widehat{SH}:=\varprojlim_n \bigoplus_{k=0}^nS^k(H_K^*)$ with $H_K^* = \Hom(H_K, \Bbbk)$, i.e. there is an isomorphic natural transformation from $\Hom(\widehat{SH}, \cdot)$ to $\widehat\de(\cdot)$. 

(b) There is a bijection between $\widehat\de(\widehat{SH})$ and the set
\be
\Big\{\ud \varphi \ : \ \ud\varphi=\varphi_1, \varphi_2, \cdots \text{ is an $L_\infty$-morphism from $(H_K,\ud 0)$ to $( \cA, \ud \ell^K)$ } \Big\}/ \sim
\ee
where $\sim$ means the $L_\infty$-homotopy equivalence relation given in Definition \ref{shlh}.

\eep

\begin{proof}
(a) This is a special case of Theorem 5.5 of \cite{Ma02}, because $(\cA, \ud \ell^K)$ is {formal} by Proposition \ref{zero}.

(b) 
Let $\{e_{\alpha} \ : \ \alpha \in I \}$ be a homogeneous $\Bbbk$-basis of the $\bZ$-graded $\Bbbk$-vector space $H_K=H_K(\cA)$, where $I$ is an index set.
Let $t^\alpha$ be the $\Bbbk$-dual of $e_\alpha$, where $e_\alpha$ varies in $\{e_{\alpha}\ : \ \alpha \in I \}$. Then $\{t^{\alpha}\ : \ \alpha \in I \}$ is the dual $\Bbbk$-basis of $\{e_{\alpha}\ : \ \alpha \in I \}$ and the degree $|t^{\alpha}|=-|e_\alpha|$ where $e_\alpha \in H_K^{|e_\alpha|}$. 
We consider the power series ring $\Bbbk[[t^\alpha]]=\Bbbk[[t^\alpha  : {\alpha \in I}]]$. Let $J$ be the unique maximal ideal of $\Bbbk[[t^\alpha]]$,
so that 
\eqn\artalg{
\ma_N:=\Bbbk[[t^\alpha]]/J^{N+1} \in \hbox{\it Ob}(\art_{\Bbbk})
}
for arbitrary $N \geq 0$.
Then $\widehat{SH}$ is isomorphic to $\varprojlim_N \ma_N$.
For a given $L_\infty$-homotopy type of $\ud \varphi$, we define
\be
[\G]=[\G(\ud \varphi)]=\bigg[\sum_{n=1}^\infty \frac{1}{n!}\sum_{\alpha_1, \cdots,\alpha_n} t^{\alpha_n}\cdots t^{\alpha_1} \otimes \varphi_n(e_{\alpha_1}, \cdots, e_{\alpha_n})\bigg] \in \widehat\de(\widehat{SH}),
\ee
where $[\cdot]$ means the homotopy equivalence class.
Let us check that this is a well-defined map, i.e. it sends $L_\infty$-homotopy types to homotopy equivalence classes. Let $\ud{\tilde \varphi}$ be $L_\infty$-homotopic to $\ud \varphi$ and $\tilde \G := \G(\ud {\tilde \varphi})$. 
By Definition \ref{shlh}, there exists a polynomial solution
$\ud \Phi(\tau)=\Phi_1(\tau), \Phi_2(\tau), \cdots$ of the flow equation with respect to a polynomial family $\ud \lambda(\tau) = \lambda_1(\tau), \lambda_2(\tau), \cdots$ such
that $\ud \Phi(0)=\ud \varphi$ and $\ud \Phi(1)=\ud {\tilde\varphi}$. If we set
\be
\Xi(\tau) = \sum_{n\geq 1} \frac{1}{n!} \sum_{\alpha_1, \cdots,\alpha_n} t^{\alpha_n}\cdots t^{\alpha_1} \otimes \Phi_n(\tau)(e_{\alpha_1}, \cdots, e_{\alpha_n}), \\
\Lambda(\tau) = \sum_{n\geq 1} \frac{1}{n!} \sum_{\alpha_1, \cdots,\alpha_n} t^{\alpha_n}\cdots t^{\alpha_1} \otimes \lambda_n(\tau)(e_{\alpha_1}, \cdots, e_{\alpha_n}),
\ee
then we have $\Xi(0)=\G$ and $\Xi(1)=\tilde \G$. Moreover, the flow equation implies (by a direct computation) that
\be
\prt{}{\t}e^{\Xi(\tau)} = K (\Lambda(\tau) \cdot e^{\Xi(\tau)}).
\ee
Hence $\Xi(0)=\G$ and $\Xi(1)=\tilde \G$ are homotopic to each other in the sense of Definition \ref{hhh}.
 Conversely, for a given equivalence class $[\G]$ of a solution of the Maurer-Cartan equation
\be
[\G] =\bigg[\sum_{n=1}^\infty \frac{1}{n!}\sum_{\alpha_1, \cdots,\alpha_n} t^{\alpha_n}\cdots t^{\alpha_1} \otimes v_{\alpha_1\cdots \alpha_n}\bigg] \in \widehat\de(\widehat{SH}),
\ee
define $\varphi[\G]_n(e_{\alpha_1}, \cdots, e_{\alpha_n}):= v_{\alpha_1\cdots \alpha_n}$ for any $n\geq 1$ and extend it $\Bbbk$-linearly. Then $\ud {\varphi[\G]}$ is an $L_\infty$-morphism and is also a well-defined map by a similar argument.
\end{proof}

This proposition, combined with Proposition \ref{zero}, says that there exists an $L_\infty$-quasi-isomorphism $\ud \varphi^H$ such that the pair $(\widehat{SH}, \G)$ where
\be
\G =\sum_{n=1}^\infty \frac{1}{n!}\sum_{\alpha_1, \cdots,\alpha_n} t^{\alpha_n}\cdots t^{\alpha_1} \otimes \varphi_n^H(e_{\alpha_1}, \cdots, e_{\alpha_n}) \in \widehat\de(\widehat{SH}),
\ee
 is a \textit{universal family} in the sense of Definition 14.3, \cite{Har10}. If $\G \in (\mm_{\widehat{SH}}\otimes \cA)^0$ corresponds to $\ud \varphi$, 
 we will use the notation $\G=\G_{\ud \varphi}$.

\begin{remark}
(a) This proposition can be generalized to any $L_\infty$-algebra $(V, \ud \ell)$, if we replace $\art_{\Bbbk}$ by the category of unital Artinian local $\Bbbk$-CDGAs (commutative differential graded algebras) with residue field $\Bbbk$; see Theorem 5.5, \cite{Ma02}.

(b) Let $\frC_\Bbbk\mA$ be the full subcategory of $\frC_\Bbbk$ whose objects consist of unital $\bZ$-graded commutative artinian local $\Bbbk$-algebras with residue field $\Bbbk$ and a differential which kills the unity (we do not require a compatibility between the differential and the multiplication). Then we can construct a (generalized) deformation functor $\mathfrak{PDef}_{(\cA, \cdot, K)}$ from $\frC_\Bbbk\mA$ to $\hbox{\bf Set}$, by 
associating, for any object $(\cA, \cdot, K)$ of $\frC_\Bbbk$,  
\be
\ma &\mapsto&  \mathfrak{PDef}_{(\cA, \cdot, K)}(\ma)=\{\G_\ma \in (\mm_\ma \otimes  \cA)^0  \ : \  K e^{\G_\ma}=0 \}/\sim
\ee
where $\sim$ is the equivalence relation given in Definition \ref{hhh}. The key point here is that a morphism $f:(\ma,\cdot, K_\ma) \to (\ma',\cdot, K_{\ma'})$,  in $\frC_\Bbbk\mA$ (recall that these are not ring homomorphisms) sends the solution $\G_\ma$ of $K e^{\G_\ma}=0$ to the solution $\Phi^f(\G_\ma)$ of $K e^{\Phi^f(\G_\ma)}=0$. We will study this functor more carefully in a sequel paper.
\end{remark}

Now we examine what we are deforming by the functor $\de$. A solution of the functor $\de$ gives a formal deformation of $(\cA, \cdot, K)\in \hbox{\it Ob}(\frC_\Bbbk)$ inside the category $\frC_\Bbbk$; see Lemma \ref{lemsix} below.
We use new notation $K_\G$ for $L^K_\G$ in \desderi
\eqn\kgamma{
K_\G(\lambda):=L^K_\G(\lambda):= \left(  {K}\lambda +{\ell}^{K}_2\big(\G, \lambda \big)
+\sum_{n=3}^\infty \frac{1}{(n-1)!}{\ell}^{K}_n\big(\G,\cdots,\G, \lambda\big)
\right),
}
 for any homogeneous element $\lambda \in \ma \otimes \cA$.

\bel \label{lemsix}
Let $\ma \in \hbox{\it Ob}(\art_{\Bbbk})$ and $\G \in (\mm_\ma \otimes \cA)^0$ be a solution of the Maurer-Cartan equation in \dfuctor. Then $K_\G$ is a $\Bbbk$-linear map on $\ma \otimes \cA$ of degree 1 and satisfies
\be
K_\G^2=0.
\ee
In other words, $(\ma \otimes \cA, \cdot, K_\G)$ 
 is also an object of the category $\frC_\Bbbk$.
\eel

\begin{proof}
If we plug in $\lambda = K_\G(v)$ for any homogeneous element $v\in \ma \otimes \cA$ in \desderi, then 
\be
K( K_\G(v) \cdot e^\G)= K_\G (K_\G(v)) \cdot e^\G = K_\G^2(v) \cdot e^\G
\ee
since $K  e^\G=0$. But the left hand side of the above equality is
\be
K ( K(v \cdot e^\G) - (-1)^{|v|} v \cdot K e^\G) = K^2 (v\cdot e^\G)=0
\ee
by using \desderi\ again. Therefore $ K_\G^2(v) \cdot e^\G=0$, which implies that $K_\G^2=0$.
It is obvious that $K_\G$ has degree 1 by its construction.
\end{proof}

By the above lemma, we can consider the cohomology $H_{K_\G}(\ma \otimes \cA)$ of the cochain complex 
$(\ma \otimes \cA, \cdot, K_\G)$. Moreover, we can formally deform 
a morphism $\cC:(\cA, \cdot, K)\to (\Bbbk, \cdot, 0)$ by using a solution for $\de$.

\bel
Let $\cC: (\cA, K) \to (\Bbbk, 0)$ be a cochain map. Let $\G \in (\mm_{\ma} \otimes \cA)^0$ be the solution of the Maurer-Cartan equation in \dfuctor. If we define
\be
\cC_\G (x) = \cC\left( x \cdot e^\G\right), \quad x\in \ma \otimes \cA,
\ee
then $\cC_\G: (\ma \otimes \cA, K_\G) \to (\ma \otimes \Bbbk,0)$ is also a cochain map, where $K_\G$ is 
given in \kgamma.
\eel

\begin{proof}
We have to check that $\cC_\G \circ K_\G=0$. This follows from \desderi;
\be
(\cC_\G \circ K_\G )(x) = \cC \left(K_\G(x) \cdot e^\G\right) = \cC \left(K \left(x \cdot e^\G\right)\right)=0,
\ee
for any homogeneous element $x\in \ma \otimes \cA$.
\end{proof}

\subsection{Invariants of homotopy types of $L_{\infty}$-morphisms} \label{subs3.4}

Throughout this subsection we assume that the cohomology space $H_K=H_K(\cA)$ is a finite dimensional $\Bbbk$-vector space: let $\{ e_{\alpha} : \alpha \in I\}$ be a homogeneous $\Bbbk$-basis of $H_K$ where $I \subseteq \bN$ is a finite index set. Let $\{t^\alpha \ : \ \alpha\in I\}$ be the dual $\Bbbk$-basis of $\{e_{\alpha} \ : \ \alpha \in I \}$. Recall that the degree of $t^{\alpha}$ is $-|e_\alpha|$ where $e_\alpha \in H^{|e_\alpha|}_K$. 

\begin{proposition} \label{comLi}
Let $\cC:(\cA,  \cdot, K) \to (\Bbbk, \cdot, 0)$ be a morphism in the category $\frC_\Bbbk$.  
Let $C:H_K \to \Bbbk$ be the induced map on cohomology. Then $C:H_K\to \Bbbk$ can be 
enhanced to an $L_\infty$-morphism $\ud\phi^\cC \bullet \ud\varphi^H$ for some $L_\infty$-quasi-isomorphism 
$\ud \varphi^H$, i.e. in the following diagram
\eqn\Liecompv{
\xymatrixcolsep{5pc}\xymatrix{(H_K, \ud 0) 
\ar@{..>}@/^2pc/[rr]_{\ud\phi^\cC \bullet \ud\varphi^H}
\ar@{..>}@/_2pc/[r]^{\ud \varphi^H}   \ar[r]^{\varphi_1^H} &
(\cA, \ud \ell^{K})
  \ar@{..>}@/_2pc/[r]^{\ud \phi^{\cC}}  \ar[r]^{\phi_1^{\cC}=\cC} & (\Bbbk,\ud 0),}
}
we have $C= \cC \circ \varphi_1^H.$
\end{proposition}

\begin{proof}
Theorem \ref{Desc} gives us the descendant $L_{\infty}$-morphism $\underline \phi^{\cC}$ of the cochain map $\cC$.  
Proposition \ref{zero} supplies us with
an $L_{\infty}$-quasi-morphism $\ud \varphi^H$ (which is \textit{not a descendant $L_\infty$-morphism}).
Let $\ud\phi^\cC \bullet \ud\varphi^H$ be the $L_{\infty}$-morphism from $(H_K,\underline 0)$ to $(\Bbbk,\underline 0)$ which is defined to be the composition of $L_{\infty}$ morphisms $\underline \varphi^H$ and $\underline \phi^{\cC}$.  Then the desired result $C=\cC \circ \varphi_1^H$ follows, since the first piece $\varphi_1^H$ of $\ud \varphi^H$ is a cochain quasi-isomorphism and it induces the identity on $H_K$.

\end{proof}

\bed\label{doit}
Let $\cC:(\cA,  \cdot, K) \to (\Bbbk, \cdot, 0)$ be a morphism in the category $\frC_\Bbbk$.   

(a) Let $\ma \in \hbox{\it Ob}(\art_{\Bbbk})$. A solution $\G  \in (\mm_\ma \otimes  \cA)^0$ of the Maurer-Cartan equation
$\sum_{n=1}^{\infty} \frac{1}{n!}\ell_n(\G, \cdots,\G)=0$, i.e. $K(e^\G-1)=0$, is called a versal solution, if
the corresponding $L_\infty$-morphism $\ud{\varphi[\G]}$ is an $L_\infty$-quasi-isomorphism. 

(b) Let  $\G \in (\mm_{\widehat{SH}} \otimes \cA)^0$ be a solution of $K(e^\G-1)=0$
corresponding to an $L_\infty$-morphism $\ud \varphi$ (by Proposition \ref{defmoduli}):
\eqn\GGG{
\G=\G(\ud t)_{\ud \varphi} =  \sum_{n=1}^\infty \frac{1}{n!}\sum_{\alpha_1, \cdots, \alpha_n} t^{\alpha_n} \cdots t^{\alpha_1}\otimes \varphi_n(e_{\alpha_1}, \cdots, e_{\alpha_n}) \in (\mm_{\widehat{SH}} \otimes \cA)^0.
}
We define the generating power series attached to $\cC$ and $\G$, as the following
power series in $\ud t=\{t^\a : \a \in I\}$ of  $\mm_{\widehat{SH}}$;
\be
\cC(e^\G-1) =\cC(e^{\G(\ud t)_{\ud \varphi}}-1) \in \widehat{SH}^0.
\ee
\eed

\bel \label{invhom}
Let $\cC:(\cA, \cdot, K) \to (\Bbbk, \cdot, 0)$ be a morphism in the category $\frC_\Bbbk$. If we define $\Omega_{\alpha_1\cdots\alpha_n} \in \Bbbk$ 
by the following equality
\eqn\mgf{
\cC(e^{\G(\ud t)_{\ud \varphi}}-1) = e^{\Phi^\cC(\G(\ud t)_{\ud \varphi})} -1= \sum_{n=1}^\infty \frac{1}{n!}\sum_{\alpha_1, \cdots, \alpha_n}t^{\alpha_n} \cdots t^{\alpha_1}\otimes \Omega_{\alpha_1\cdots\alpha_n} \in  \widehat{SH}^0,
}
then we have
\eqn\hper{
\eqalign{
\Phi^\cC(\G(\ud t)_{\ud \varphi} )=&\sum_{n=1}^{\infty}\frac{1}{n!} \sum_{\alpha_1, \cdots, \alpha_n}  t^{\alpha_n} \cdots t^{\alpha_1}\otimes ( \ud \phi^\cC \bullet \ud \varphi)_n(e_{\alpha_1}, \cdots, e_{\alpha_n}) \in \mm_{\widehat{SH}}^0,
\cr 
\Omega_{\alpha_1\cdots\alpha_n} = &\sum_{\pi \in P(n)} \epsilon(\pi) (\phi^\cC \bullet \varphi)(e_{B_1}) \cdots (\phi^\cC \bullet \varphi)(e_{B_{|\pi|}}) \in \Bbbk, 
}
}
where $e_B=e_{j_1}\otimes \cdots \otimes e_{j_r}$ for $B=\{j_1, \cdots, j_r\}$. 
\eel

\begin{proof}
We leave this combinatorial lemma as an exercise; it follows from Definition \ref{compl}
of the composition of $L_\infty$-morphisms. 
\end{proof}

\bed \label{gpseries}
For a given $L_\infty$-morphism $\ud \kappa=\kappa_1, \kappa_2, \cdots,$ from an $L_\infty$-algebra $(V, \ud 0)$ to $(\Bbbk, \ud 0)$, we define the following power series in $\ud t=\{t^\a : \a \in I\}$ of  $\mm_{\widehat{SV}}$;
\be
\cZ_{[\ud \kappa]}(\ud t):= \exp\big(\sum_{n=1}^{\infty}\frac{1}{n!} \sum_{\alpha_1, \cdots, \alpha_n}  t^{\alpha_n} \cdots t^{\alpha_1}\otimes  \kappa_n(e_{\alpha_1}, \cdots, e_{\alpha_n}) \big) -1   \quad \in \widehat{SV}^0,
\ee
where $\{e_{\alpha} : \a \in I\}$ is a basis of a $\Bbbk$-vector space $V$ and $\{t^\a : \a \in I\}$ is its dual $\Bbbk$-basis.
Here we use the notation $\widehat{SV}:=\varprojlim_n \bigoplus_{k=0}^nS^k(V^*)$ with $V^* = \Hom(V, \Bbbk)$.
\eed

If we let $V=H_K$ and $\widehat{SV}= \widehat{SH}$, then \kexp\ implies that
\eqn\chazng{
\cZ_{[\ud\phi^{\cC}\bullet \ud \varphi]}(\ud t)=\cC(e^{\G(\ud t)_{\ud \varphi}}-1).
}

  The main theorem here is that the generating power series $\cC(e^{\G(\ud t)_{\ud \varphi}}-1)$ is an invariant of the 
  homotopy types of $\cC$ and $\G(\ud t)_{\ud \varphi}$. Accordingly, we show that $\cZ_{[\ud\phi^{\cC}\bullet \ud \varphi]}(\ud t)$ is an invariant of 
  the $L_\infty$-homotopy types of the two $L_\infty$-morphisms. Let $\G, \tilde \G \in (\mm_\ma \otimes  V)^0$ such that 
\be
\sum_{n=1}^{\infty} \frac{1}{n!}\ell_n(\G, \cdots,\G)=0=\sum_{n=1}^{\infty} \frac{1}{n!}\ell_n(\tilde \G, \cdots,\tilde\G).
\ee

\bet \label{homotopyperiod}
Let $\G$ be homotopy equivalent to $\tilde \G$ (see Definition \ref{hhh}).
 Let $\tilde \cC$ be a cochain map which is cochain homotopic to $\cC$. Then we have
\eqn\hmz{
\cC(e^\G -1 ) = \tilde \cC (e^{\tilde \G} -1).
}
Similarly, if $\ud \varphi$ (corresponding to $\G$) is $L_\infty$-homotopic to some $\ud{\tilde \varphi}$ (corresponding to $\tilde \G$) and $\ud \phi^\cC$ is $L_\infty$-homotopic to some $L_\infty$-morphism $\ud{\tilde \phi}$, then we have
\eqn\lmz{
\cZ_{[\ud\phi^{\cC}\bullet \ud \varphi]}(\ud t)=\cZ_{[\ud{\tilde \phi}\bullet \ud {\tilde\varphi}]}(\ud t). 
}
\eet

\begin{proof}

We first prove \lmz.
If $\ud \varphi$ (corresponding to $\G$) is $L_\infty$-homotopic to some $\ud{\tilde \varphi}$ (corresponding to $\tilde \G$) and $\ud \phi^\cC$ is $L_\infty$-homotopic to some $L_\infty$-morphism $\ud{\tilde \phi}$,
then Lemma \ref{hoco} implies that 
$\ud{ \phi}^{\cC}\bullet \ud \varphi$ is $L_\infty$-homotopic to $\ud{\tilde \phi}\bullet \ud {\tilde \varphi}$.
Because both $\ud{ \phi}^{\cC}\bullet \ud \varphi$ and $\ud{\tilde \phi}\bullet \ud {\tilde \varphi}$ are defined from $(H_K,\ud 0)$ into $(\Bbbk,\ud 0)$, i.e. both $H_K$ and $\Bbbk$ have zero $L_\infty$-algebra structures, they should be the same by Definition \ref{shlh}:
\be
\ud{ \phi}^{\cC}\bullet \ud \varphi=\ud{\tilde \phi}\bullet \ud {\tilde \varphi}.
\ee
Therefore we have the equality \lmz:
\be
\cZ_{[\ud{ \phi}^{\cC}\bullet \ud \varphi]}(\ud t)=\cZ_{[\ud{\tilde \phi}\bullet \ud {\tilde \varphi} ]}(\ud t).
\ee

We now prove \hmz\ by using the equalities \chazng\ and \lmz.
Since Proposition \ref{defmoduli}, (b) gives a correspondence between the homotopy types of $\G=\G_{\ud \varphi}$ and the $L_\infty$-homotopy types of $\ud \varphi$ and Lemma \ref{deshomotopy} says that $\ud{\phi}^{\tilde \cC}$ is 
$L_\infty$-homotopic to $\ud \phi^\cC$, we see the invariance $\cC(e^{\G}-1) = \tilde\cC (e^{\tilde \G}-1)$ as follows by using \chazng\ and \lmz:

\be
 \cC (e^{\G_{\ud \varphi}}-1)
=\cZ_{[\ud{ \phi}^{\cC}\bullet \ud \varphi]}(\ud t)
=\cZ_{[\ud{ \phi}^{\tilde \cC}\bullet \ud {\tilde\varphi}]}(\ud t)
=\tilde \cC (e^{\G_{\ud {\tilde \varphi}}}-1)
=\tilde \cC (e^{\tilde \G}-1).
\ee

We also provide an alternative proof of \hmz\ in order to illustrate a key idea behind $L_\infty$-homotopy invariance  in a better way. If $\cC:(\cA,K)\to (\Bbbk,0)$ is cochain homotopic to $\tilde \cC$, then $\tilde \cC=\cC + \cS \circ K$ where $\cS$ is a cochain homotopy. Since $\cC \circ K =0$, we have
\be
\tilde \cC(e^\G-1) = \cC (e^{\G}-1).
\ee
If $\G$ is homotopy equivalent to $\tilde \G$, Definition \ref{hhh} says that there is a polynomial solution 
$\G(\t) \in (\mm_{\widehat{SH}} \otimes \cA)^0[\t]$ with $\G(0) = \G$ and $\G(1) =\tilde \G$ to 
the following flow equation
 \be
 \prt{}{\tau}   e^{\G(\t)} = K \left(\l(\tau) \cdot e^{\G(\tau)}\right)
 \ee
 for some 1-variable polynomial $\l(\tau) \in (\mm_{\widehat{SH}} \otimes \cA)^0[\t]$. This implies that
 \be
 e^{\tilde \G} - e^{\G} = K \left( \int_{0}^{1} \l(\t) \cdot e^{\G(\t)} d\t \right),
 \ee
which in turn implies that 
$\cC(e^{\tilde \G}) = \cC (e^{\G}).$
 \end{proof}


Let $\cZ$ be the quotient $\Bbbk$-vector space of all (the degree 0) cochain maps from $(\cA, K)$
to $(\Bbbk,0)$ modulo the subspace of all the maps of the form $\cS \circ K$ where $\cS:\cA
\to \Bbbk$ varies
over $\Bbbk$-linear maps of degree -1. In other words, $\cZ$ is the space of cochain homotopy classes of maps from $(\cA,K)$ to $(\Bbbk,0)$.
\bep \label{cyclehomotopy}
Let $(\cA,K)=(\cA_\rho, K_\rho)$ be associated to a Lie algebra representation $\rho$. The $\Bbbk$-vector space $\cZ$ is isomorphic to the $\Bbbk$-dual of $H_K^{0}(\cA)$.
\eep
\begin{proof}
Since every representative $\cC$ of an element of $\cZ$ has degree 0 and $(\Bbbk,0)$ is concentrated in degree 0, all the homogeneous elements except for degree 0 elements in $\cA$ vanish under the map $\cC$. In fact, there is no $\Bbbk$-linear map $\cS:\cA \to \Bbbk$ of degree -1, since $\cA$ does not have positive degree, i.e. $\cA^{1}=\cA^{2}=\cdots=0$. Then it is clear that $\cZ$ is isomorphic to the $\Bbbk$-dual of $\cA^{0}/ K(\cA^{-1})=: H_K^0(\cA)$.
 \end{proof}

\subsection{Differential equations attached to variations of period integrals} \label{subs3.5}

The main goal of this subsection is to prove that the generating power series $\cC(e^{\G(\ud t)_{\ud \varphi^H}}-1)=
\cZ_{[\ud \phi^\cC \bullet \ud \varphi^H]}(\ud t)$ attached to $\cC$ and a Maurer-Cartan solution $\G(\ud t)_{\ud \varphi^H}$
corresponding to 
an $L_\infty$-quasi-isomorphism $\ud \varphi^H$ satisfies a system of second order partial differential equations. 


These differential equations, which are obtained by analyzing the binary product structure on $\cA$ and
the differential $K$, are governed by the underlying infinity homotopy structure on $(\cA, \cdot, K)$, namely the descendant $L_\infty$-algebra $(\cA, \ud \ell^K)$; we will show that the differential equations themselves are \textit{invariants of the $L_\infty$-homotopy type} of the solution of the Maurer-Cartan equation. See Theorem \ref{diff} for details. Moreover, these differential equations will lead to a flat connection on the tangent bundle
of the formal deformation space of $\de$; see Theorem \ref{fsc}.

 By Proposition \ref{defmoduli}, a (homotopy type of) solution $\G=\G_{\ud \varphi} \in (\mm_{\widehat{SH}} \otimes \cA)^0$ of the Maurer-Cartan equation
\eqn\MCeq{
 \sum_{n=1}^{\infty} \frac{1}{n!} \ell^K_n(\G, \cdots, \G)=0,
\label{MCeq}
}
gives us an ($L_\infty$-homotopy type of) $L_\infty$-morphism $\ud \varphi$ from $(H_K, \ud 0)$ to $(\cA, \ud \ell^K)$. 
We will make differential equations with respect to the parameters $\{t^\alpha\}$ in the complete local $\Bbbk$-algebra $\widehat{SH}$. 
Here recall that $\widehat{SH}:=\varprojlim_n \bigoplus_{k=0}^nS^k(H_K^*)$ which is isomorphic to $\Bbbk[[\ud t]]$.

\bet \label{diff}
Let $\cC:(\cA, \cdot, K) \to (\Bbbk, \cdot, 0)$ be a morphism in the category $\frC_\Bbbk$. 
Let $\G=\G(\ud t)_{\ud \varphi^H} \in (\mm_{\widehat{SH}} \otimes \cA)^0$ be a solution of the Maurer-Cartan equation \MCeq\ corresponding to an $L_\infty$-quasi-isomorphism $\ud \varphi^H$ from $(H_K,\ud 0)$ to $(\cA, \ud \ell^K)$. Assume that $H_K(\cA)$ is a finite dimensional $\Bbbk$-vector space and $H_{K_\G}(\widehat{SH} \otimes \cA)$ is a free $\widehat{SH}$-module satisfying
\eqn\fdim{
\dim_\Bbbk H_K(\cA) = \rk_{\widehat{SH}} H_{K_\G}(\widehat{SH} \otimes \cA).
}
Then there exist elements $A_{\a\b}{}^\g(\ud t)=A_{\a\b}{}^\g(\ud t)_\G \in \widehat{SH}$ depending on $\G$,  where $\ud t=\{t^\alpha \ : \ \alpha\in I\}$, such that
\eqn\hodi{
\left(\partial_{\alpha} \partial_\b  - \sum_\g A_{\a\b}{}^\g(\ud t) \partial_\g \right) \cC(e^\G-1)=0,\quad \text{for }   \a,\b \in I,
}
where $\partial_{\alpha}$ means the partial derivative with respect to $t^\a$ where $\a \in I$.
Moreover, if $\tilde\G$ and $\G$ are homotopy equivalent, then $A_{\a\b}{}^\g(\ud t)_{\tilde\G}=A_{\a\b}{}^\g(\ud t)_\G$.
\eet

\begin{proof}
Note that $\G=\G(\ud t)$ depends on $\ud t$. Then the condition $K (e^{\G(\ud t)}-1)=0$ implies that
\be
K(\partial_{\alpha} e^{\G(\ud t)}) = K \big(\partial_{\alpha} \G(\ud t) \cdot e^{\G(\ud t)}\big)=0,
\ee
which says, by the equality \desderi, that
\be
K_{\G(\ud t)}\big( \partial_{\alpha} \G(\ud t) \big) =0.
\ee
Therefore, combined with the condition \fdim, we may assume that $\{ [\partial_{\alpha} \G(\ud t)]  \ : \ \a\in I  \}$ is an $\widehat{SH}$-basis of $H_{K_\G}(\widehat{SH} \otimes \cA)$, since $\varphi^H_1$ is a cochain quasi-isomorphism.  Here $[\cdot]$ means the cohomology class. The condition $K e^{\G(\ud t)}=0$ also implies that
\be
K (\partial_{\alpha} \partial_\b e^{\G(\ud t)})=K\left(\big(\partial_{\a\b} \G(\ud t) + \partial_{\alpha}\G(\ud t) \partial_\b\G(\ud t)\big)\cdot e^{\G(\ud t)}\right) =  0,
\ee
where $\partial_{\a\b} = \frac{\partial^2}{\partial{t_{\alpha}}\partial{t_\b}}$, and the equality \desderi\ says that
\be
K_{\G(\ud t)}\big( \partial_{\a\b} \G(\ud t) + \partial_{\alpha} \G(\ud t) \partial_\b \G(\ud t) \big)\cdot e^{\G(\ud t)}=0.
\ee
Thus we should be able to write down $[\partial_{\a\b} \G(\ud t) + \partial_{\alpha} \G(\ud t) \partial_\b \G(\ud t)]$ as an $\widehat{SH}$-linear combination of $[\partial_{\alpha} \G(\ud t)]$'s, i.e. there exists a unique 3-tensor $A_{\a\b}{}^\g(\ud t) \in \widehat{SH}$ such that
\eqn\klm{
\partial_{\a\b} \G(\ud t) + \partial_{\alpha} \G(\ud t) \partial_\b \G(\ud t)
= \sum_\g A_{\a\b}{}^\g(\ud t) \partial_\g\G(\ud t) + K_{\G(\ud t)}(\Lambda_{\a\b}(\ud t))
}
for some $\Lambda_{\a\b}(\ud t) \in \widehat{SH} \otimes \cA$. Then this is equivalent to
\eqn\eqproof{
\partial_{\alpha} \partial_\b e^{\G(\ud t)} - \sum_\g A_{\a\b}{}^\g(\ud t) \partial_\g e^{\G(\ud t)} = K(\Lambda_{\a\b}(\ud t)\cdot e^{\G(\ud t)}).
}
We finish the proof by applying $\cC$ to the above equality and using the fact that $\cC$ is a cochain map, i.e. $\cC\circ K=0$. 

If $\tilde \G(\ud t)$ and $\G(\ud t)$ are homotopy equivalent, then, according to Definition \ref{hhh}, we have
\eqn\Labc{
e^{\tilde \G(\ud t)} - e^{\G(\ud t)} = K \left(\int_{0}^1 \lambda(\tau) e^{\G(\ud t) (\tau)} d\tau\right),
}
 where $\G(\ud t) (1)= \tilde \G(\ud t)$ and $\G(\ud t) (0) =\G(\ud t)$. By using \eqproof\ we can derive the following;
$$
\eqalign{
\partial_{\alpha} \partial_\b (e^{\tilde \G(\ud t)} - e^{\G(\ud t)}) + \left(\sum_\g A_{\a\b}{}^\g(\ud t)_\G \partial_\g e^{\G(\ud t)}-\sum_\g A_{\a\b}{}^\g(\ud t)_{\tilde \G} \partial_\g e^{\tilde\G(\ud t)} \right) 
\cr
= K \left(\Lambda_{\a\b}(\ud t)\cdot e^{\G(\ud t)} - \tilde \Lambda_{\a\b}(\ud t) \cdot e^{\tilde \G(\ud t)}\right).
}
$$
Therefore \Labc\ implies that
$$
 \sum_\g A_{\a\b}{}^\g(\ud t)_\G \partial_\g e^{\G(\ud t)}-\sum_\g A_{\a\b}{}^\g(\ud t)_{\tilde \G} \partial_\g e^{\tilde\G(\ud t)} \in \Im K.
$$
 After we add $\sum_\g A_{\a\b}{}^\g(\ud t)_{\tilde \G} \partial_\g e^{\G(\ud t)}$ and subtract to the above, we can apply \Labc\ to prove that
$$
 D:=
 \bigg( \sum_\g  \big(A_{\a\b}{}^\g(\ud t)_\G -A_{\a\b}{}^\g(\ud t)_{\tilde \G}\big) \partial_\g\G(\ud t)\bigg)\cdot e^{\G(\ud t)} \in \Im K.
$$
Note that $D$ has the form $K(\xi \cdot e^{\G(\ud t)})$ for some $\xi$.
Then \desderi\ implies that
$$
\sum_\g  \big(A_{\a\b}{}^\g(\ud t)_\G -A_{\a\b}{}^\g(\ud t)_{\tilde \G}\big) \partial_\g\G(\ud t) \in \Im K_{\G(\ud t)}
$$
Since $\{ [\partial_{\alpha} \G(\ud t)]  \ : \ \a\in I  \}$ is an $\widehat{SH}$-basis of $H_{K_\G}(\widehat{SH} \otimes \cA)$, 
the desired result $A_{\a\b}{}^\g(\ud t)_\G =A_{\a\b}{}^\g(\ud t)_{\tilde \G}$ follows if we take the $K_{\G(\ud t)}$-cohomology of the above.
\end{proof}

The method used in the proof can be made into an effective algorithm (see subsection \ref{subs4.0} for a toy model case and
subsection \ref{subs4.5} for the smooth projective hypersurface case) to compute $\cC(e^{\G(\ud t)_{\ud \varphi}})$. It leads to the Picard-Fuchs type differential equation for a family of hypersurfaces if we interpret the period integrals of hypersurfaces as \textit{period integrals} of quantum Jacobian Lie algebra representations attached to hypersurfaces; 
see subsection \ref{subs4.6}.



\subsection{Explicit computation of the generating power series} \label{subs3.6}

The goal of this subsection is to reduce the problem of computing the generating power series $\cC(e^{\G}-1)$ attached to $\cC$ and $\G=\G_{\ud \varphi^H}$ to the problem of computing the 3-tensor
$A_{\a\b}{}^\g(\ud t)_\G$ for $\a, \b, \g \in I$ that appeared in \hodi. We also provide the proof of Theorem \ref{fifththeorem}
Let $\cC:(\cA, \cdot, K) \to (\Bbbk, \cdot, 0)$ be a morphism in the category $\frC_\Bbbk$. Let $\G=\G_{\ud \varphi^H} \in (\mm_\ma \otimes \cA)^0$ be a solution of the Maurer-Cartan equation \MCeq\ corresponding to an $L_\infty$-quasi-isomorphism $\ud \varphi^H$ from $(H_K,\ud 0)$ to $(\cA, \ud \ell^K)$. Assume that $H_K(\cA)$ is a finite dimensional $\Bbbk$-vector space and $H_{K_\G}(\widehat{SH} \otimes \cA)$ is a free $\widehat{SH}$-module satisfying \fdim.
Let us write $\G$ as
\eqn\GG{
\G=  \sum_{n=1}^\infty \frac{1}{n!}\sum_{\alpha_1, \cdots, \alpha_n} t^{\alpha_n} \cdots t^{\alpha_1}\otimes \varphi^H_n(e_{\alpha_1}, \cdots, e_{\alpha_n}) \in (\mm_{\widehat{SH}} \otimes \cA)^0.
}

\bel\label{onetensor}
There exist $T^\g(\ud t) \in \widehat{SH} \simeq \Bbbk[[\ud t]]$ and $\Lambda \in \widehat{SH}\otimes \cA$ such that
\eqn \TTT{
e^{\G} = 1 + \sum_\g T^\g(\ud t) \cdot \varphi^H_1(e_\g) + K(\Lambda) \in (\widehat{SH} \otimes \cA)^0.
}
\eel

\begin{proof}
Recall that $\{e_{\alpha}: \a \in I\}$ is a $\bZ$-graded homogeneous basis of $H_K$. Since $\varphi^H_1$ is a cochain quasi-isomorphism, $\{ \varphi^H_1(e_\g) \}$ 
form a set of representatives of a basis
of $H_K$. Therefore the result follows, because $K(e^\G)=0$ and $K$ is $\widehat{SH}$-linear.
\end{proof}

\ber
Since $\cC \circ K=0$, we can express the generating power series as
\eqn\mmgg{
\cC(e^{\G(\ud t)_{\ud \varphi^H}}-1)= \sum_\g T^\g(\ud t) \cdot \cC(\varphi^H_1(e_\g)).
}
So the above lemma, combined with Theorems \ref{thirdtheorem}
and \ref{tmthm}, gives us Theorem \ref{fifththeorem}. 
Moreover, its explicit expression \onediff\ in terms of $A_{\a \b}{}^{\g}(\ud t)_\G$ in \hodi\ can be derived by a direct computation.
\eer

Note that $\cC(\varphi^H_1(e_\g))$ does \textbf{not} depend on $\ud t$ and $T^\g(\ud t)$ does \textbf{not} depend on $\cC$.
Thus we only need to know the values $\cC(\varphi^H_1(e_\g))$ on the  basis $\{e_\g: \a \in I\}$ of 
the cohomology group $H_K$ and the formal power series   $T^\g(\ud t)=T^\g(\ud t)_{\G_{\ud \varphi^H} }$,
which depends only on 
the homotopy type of $\ud \varphi^H$, in order to compute $\cC(e^\G-1)$.
Let us explain how to compute $T^\g(\ud t)$. We reduce its computation to the computation of $A_{\a\b}{}^\g(\ud t)_\G$.
Let us write $T^\g(\ud t)$ as a power series in $\ud t$ as follows:
\be
T^\g(\ud t) = t^\g + \sum_{n=2}^\infty \frac{1}{n!} \sum_{\a_1,\cdots, \a_n}t^{\a_n} \cdots t^{\a_1} \cdot M_{\a_1 \a_2 \cdots \a_n}{}^\g, \quad 
\text{for } \g \in I,
\ee
for a uniquely determined $M_{\a_1 \a_2 \cdots \a_n}^\g \in \Bbbk$. We define a $\Bbbk$-multilinear operation $M_n: S^n(H_K) \to H_K$ as follows
\be
M_n(e_{\a_1}, \cdots, e_{\a_n}) = \sum_{\g} M_{\a_1 \a_2 \cdots \a_n}{}^\g  e_{\g}, \quad n \geq 2.
\ee
We also write that the 3-tensor $A_{\a\b}{}^\g(\ud t)_\G$ in \hodi\ as follows:
\be
A_{\a\b}{}^\g(\ud t)_\G=m_{\a\b}{}^\g 
+ \sum_{n=1}^{\infty} \frac{1}{n!} \sum_{\a_1,\cdots, \a_n}t^{\a_n} \cdots t^{\a_1} m_{\a_1 \cdots \a_n\a\b}{}^\g, \quad 
\text{for } \a, \b, \g \in I,
\ee
for a uniquely determined $m_{\a_1 \cdots \a_n\a\b}{}^\g \in \Bbbk$. We define a $\Bbbk$-multilinear operation $m_n: T^n(H_K) \to H_K$ as follows
\be
m_n(e_{\a_1}, \cdots, e_{\a_n}) = \sum_{\g} m_{\a_1 \a_2 \cdots \a_n}{}^\g  e_{\g}, \quad n \geq 2.
\ee

\bep \label{Mmm}
Let $M_1$ be the identity map on $H_K$. Then we have the following relationship between $m_n$ and $M_n$:
$$
\eqalign{
M_n(v_1,\cdots, v_n) =
&
\sum_{\substack{\pi \in P(n)\\|B_{|\pi|}|=n-|\pi|+1\\ n-1\sim_\pi n}} 
\ep(\pi) M_{|\pi|}\left( v_{B_1}, \cdots, v_{B_{|\pi|-1}},
m\big(v_{B_{|\pi|}}\big)\right).
}
$$
for any homogeneous element $v_1, \cdots, v_n \in H_K$ and $n\geq 2$. The following is also true:
\begin{itemize}
\item
$M_n$ is a $\Bbbk$-linear map from $S^n(H_K)$ into $H_K$ of degree zero for all $n\geq 1$

\item
$M_{n+1}\big(v_1,\cdots, v_n, 1_{H}\big)=M_n\big(v_1,\cdots, v_n\big)$ for all $n\geq 1$,
where $1_{H}$ is a distinguished element corresponding to $1_\cA$.
\end{itemize}
\eep

\begin{proof}
The equality \hodi\ and \mmgg\ imply that
$$
\partial_{\alpha} \partial_\b \left( \sum_\g T^\g(\ud t) \cdot \cC(\varphi^H_1(e_\g)) \right)=
 \sum_\r A_{\a\b}{}^\r(\ud t)_\G \partial_\r \left( \sum_\g T^\g(\ud t) \cdot \cC(\varphi^H_1(e_\g)) \right),
$$
for $\a, \b \in I$.
The combinatorial formula between $m_n$ and $M_n$ follows essentially from comparing the $\ud t$-coefficients of both sides.
We leave readers to verify this formula and the remaining properties of $M_n$ as exercises. 
\end{proof}

This explicit combinatorial formula says that the data of $m_n$, $n \geq 2$ completely determines
$M_n, n \geq 2$, and vice versa; knowing $A_{\a\b}{}^\g(\ud t)_\G$ for  $\a, \b, \g \in I$, is equivalent to knowing $T^\g(\ud t)$ for $\g \in I$. 
By Proposition \ref{Mmm}, it is enough to give an algorithm to compute the 3-tensor $A_{\a\b}{}^\g(\ud t)_\G$ and to know the values $\cC(\varphi_1^H(e_{\alpha})), \a \in I$, in order to compute the generating power series $\cC(e^\G-1)$. 
We will provide an effective algorithm for computing $A_{\a\b}{}^\g(\ud t)_\G$, when a versal Maurer-Cartan solution $\G$ is associated to $(\cA_X, \cdot, K_X)$, which is attached to a projective smooth hypersurface $X_G$; see
subsection \ref{subs4.5}.


\subsection{A flat connection on the tangent bundle of a formal deformation space}\label{subs3.7}

Here we will show that the system of differential equations \hodi\ can be reformulated to give a flat connection on the tangent bundle
of a formal deformation space.
 The proposition \ref{defmoduli}, (a) says that the deformation functor $\mathfrak{Def}_{(\cA,\ud \ell^K)}$ is pro-representable by $\widehat{SH}$. 
 Thus we can consider a formal deformation 
space $\cM$ attached to the universal deformation ring $\widehat{SH}$ of $\de$ so that $\Omega^0(\cM)=\widehat{SH}$, i.e. $\cM$ is the formal spectrum of $\widehat{SH}$. Let $\TM$ be the tangent bundle of $\cM$
and $\CTM$ be the cotangent bundle of $\cM$. Let $\G(\cM,\TM)$ be the $\Bbbk$-space of global sections of $\TM$
and $\Omega^p(\cM)$ be the $\Bbbk$-space of differential $p$-forms on $\cM$.
We list some important properties of the 3-tensor $A_{\a\b}{}^\g(\ud t)_\G$ in \hodi, which we use to prove the existence
of a flat connection on $\TM$.

\bel \label{conlem}
 The 3-tensor $A_{\a\b}{}^\g:=A_{\a\b}{}^\g(\ud t)_\G$ in Theorem \ref{diff} satisfies the following properties.
$$
\eqalign{
A_{\a\b}{}^\g- (-1)^{| e_{\alpha}| |e_{\beta}| } A_{\b\a}{}^\g=0
,\cr 
\partial_{\alpha} A_{\b\g}{}^\s-(-1)^{| e_{\alpha}| |e_{\beta}|} \partial_\b A_{\a\g}{}^\s
+ \sum_\r \left( A_{\b\g}{}^\r  A_{\a \r}{}^\s 
- (-1)^{| e_{\alpha}| |e_{\beta}|} A_{\a\g}^\r A_{\b \r}{}^\s\right)=0,
}
$$
for all $\a,\b,\g,\s \in I$, where $\partial_{\alpha}$ means the partial derivative 
with respect to $t^\a$.
\eel
\begin{proof}
The first one follows from the super-commutativity of the binary product of $\cA$. The second one can be proved by taking the derivatives of \hodi\ with respect to $\ud t$ and using the associativity of the binary product of $\cA$. We leave this as an exercise.
 \end{proof}

Note that $\TM$ is a trivial bundle 
and let us
write $\TM=\cM \times V$, where $V$ is isomorphic to $H_K^*$. The most general connection is of the form
$d+A$, where $A$ is an element of $\Omega^1(\cM) \hat\otimes \End_\Bbbk(V)$.
Here $\hat \otimes$ denotes the completed tensor product. We define a 1-form valued matrix ${A}_\G$ by
\eqn\asc{
({A}_\G)_\b{}^\g :=- \sum_{\alpha} d\!t^\a \cdot  A_{\a\b}{}^\g(\ud t)_\G, \quad \b, \g \in I,
}
where $A_{\a\b}{}^\g(\ud t)_\G$ is given in \hodi. Then ${A}_\G\in\Omega^1(\cM) \hat\otimes \End_\Bbbk(V)$.

\bet \label{fsc}
Let $\G=\G_{\ud\varphi^H} \in \de(\widehat{SH})$ which corresponds to an $L_\infty$-quasi-isomorphism $\ud \varphi^H$. Then the $\Bbbk$-linear operator ${D}_{\G}:=d + {A}_\G$ defined in \asc\  
\eqn\flatsc{
{D}_\G:=d+ {A}_\G: \G(\cM, \TM) \to  \Omega^1(\cM)\otimes_{\Omega^0(\cM)} \G(\cM, \TM), 
}
is a flat connection on $\TM$. 
\eet
%

\begin{proof}
We compute the curvature of the connection ${D}_\G$ as follows; 
\be
{D}_\G^2 = d {A}_\G + {A}_\G^2.
\ee
Then a simple computation confirms that $d {A}_\G + {A}_\G^2 = 0$ is equivalent to the second equality of Lemma \ref{conlem} by using the first equality of Lemma \ref{conlem}. Thus $d {A}_\G + {A}_\G^2 = 0$ and ${D}_\G$ is flat.
\end{proof}

\bep
Let $\ud \varphi^H$ be an $L_\infty$-quasi-isomorphism from $(H_K,\ud 0)$ to $(\cA, \ud \ell^K)$. The generating power series $\cC(e^\G-1)$ attached to $\cC$ and $\G=\G_{\ud \varphi^H}$ satisfies the equation
\be
d   \big(\partial_\g \cC(e^\G-1) \big) 
=- {A}_{\G} \big(\partial_\g \cC(e^\G-1) \big), \quad \text{ i.e. } \ {D}_\G \big(\partial_\g \cC(e^\G-1) \big)=0,
\ee
where we view $\partial_\g \cC(e^\G-1)$ as a column vector indexed by $\g \in I$.
\eep

\begin{proof}
If we multiply \hodi\ by $d\! t^\a$ and sum over $\a \in I$, then we get 
\eqn\hodim{
\sum_\g\left( \sum_{\alpha} d\!t^\a \partial_{\alpha} \delta_\b{}^{\g}   -  \sum_{\alpha} d\!t^\a\cdot  A_{\a\b}{}^\g(\ud t) \right)
 \partial_\g\cC(e^\G-1)=0,\quad \text{for }   \b \in I,
}
where $\delta_\b{}^\g$ is the Kronecker delta symbol. Then the desired
equality follows from the definition of ${A}_\G$ in \asc, by noting that $d = \sum_{\alpha} d\!t^\a \partial_{\alpha}$.
\end{proof}

\newsec{Period integrals of smooth projective hypersurfaces} \label{section4}

In \cite{Gr69}, P. Griffiths extensively studied the period integrals of smooth projective hypersurfaces. 
We use his theory to give a non-trivial example of a period integral of a certain Lie algebra representation and reveal its hidden structures, namely its correlations and variations, by applying the general theory we developed so far.  
We start with a toy example in order to illustrate our applications of the general theory more transparently.

\subsection{Toy model}\label{subs4.0}

We go back to the example \ref{exa11}. We have a (quantum Jacobian) Lie algebra representation $\rho_S$ attached to 
a polynomial $S(x) \in \R[x]$ of degree $d+1$.
The associated cochain complex $(\cA_{\rho_S}, \cdot, K_{\rho_S})=(\cA_S, \cdot, K_S) \in \hbox{\it Ob}(\frC_\Bbbk)$ is given by $\cA_S= \R[x][\eta]=\cA_S^{-1} \oplus \cA_S^0$ and $K_S= \left( \prt{}{x} + S'(x)\right)\prt{}{\eta}$. The map $C$ defined by $C(f)= \int_{-\infty}^{\infty} f(x) e^{S(x)} dx$ can be enhanced to $\cC_S:(\cA_S, \cdot, K_S)
\to (\Bbbk, \cdot, 0)$ by Proposition \ref{enhancetochain}.

\bep \label{Etoy}
The cohomology group $H^{-1}_{K_S}(\cA_S)$ vanishes and $H^0_{K_S}(\cA_S)$ is a finite dimensional $\Bbbk$-vector space.
Its dimension is the degree of the polynomial $S'(x)$. 
\eep

\begin{proof}
Even though the proof is straightforward, we decide to provide details in order to give some indication about the more complicated
multi-variable version, Lemma \ref{memcon}. Note that $K_S$ consists of the quantum part $\Delta=\prt{}{x} \prt{}{\eta}$
and the classical part $Q=S'(x) \prt{}{\eta}$. The image of $\cA^{-1}$ under the classical part $Q$ defines the
Jacobian ideal
of $\R[x]$. We first solve an ideal membership problem (just the Euclidean algorithm in the one variable case);
for a given any $f(x) \in \R[x]$, there is an effective algorithm to find
 unique polynomials $q^{(0)}(x), r^{(0)}(x) \in \R[x]$ such that
\eqn\clideal{
f(x) =  S'(x)  q^{(0)}(x)+r^{(0)}(x), \quad \deg(r^{(0)}(x)) \leq d-1.
}
Then we use the quantum part $\Delta$ to rewrite \clideal;
$$
\eqalign{
f(x) &=  -\prt{ q^{(0)}(x)}{ x} +  (S'(x) + \prt{}{x}) q^{(0)}(x) +r^{(0)}(x) 
\cr
&= -\prt{ q^{(0)}(x)}{ x} +r^{(0)}(x) +K_S( q^{(0)}(x) \cdot \eta).
}
$$
Then we again use the Euclidean algorithm to find unique $q^{(1)}(x), r^{(1)}(x) \in \R[x]$ such that
\be
 -\prt{ q^{(0)}(x)}{ x}=  S'(x)  q^{(1)}(x)+r^{(1)}(x), \quad \deg(r^{(1)}(x)) \leq d-1.
\ee
This implies that
$$
\eqalign{
f(x) &
= S'(x)  q^{(1)}(x)+r^{(1)}(x)+r^{(0)}(x) +K_S\left( q^{(0)}(x) \cdot \eta \right)
\cr
&= -\prt{ q^{(1)}(x)}{ x}+r^{(1)}(x)+r^{(0)}(x) + K_S \left( q^{(1)}(x) \cdot \eta +q^{(0)}(x) \cdot \eta   \right).
}
$$
We can continue this process, which stops in finite steps, which shows that the dimension of $H^0_{K_S}(\cA_S)$ is not
bigger than $d$. But $1, x, \cdots, x^{d-1}$ can not be in the image of $K_S$ because of degree. Thus the
result follows.
The vanishing of $H^{-1}$ is trivially derived, since any solution $u(x)$ to the differential equation $\prt{u(x)}{x} + S'(x)
\cdot u(x)=0$ can not be a polynomial.
\end{proof}

This type of interaction between quantum and classical components of $K_S$ will be the key technique to
compute effectively (a fancy way of doing integration by parts) the generating power series $\cC(e^\G-1)$ of period integrals of the quantum Jacobian Lie algebra
representation. Therefore $(\cA_S,K_S)$ is quasi-isomorphic to the finite dimensional space 
$(H^0_{K_S}(\cA_S), 0)$ with zero differential; this is a perfect example to apply all the results 
developed in section \ref{section3}. We record a general result in order to compute descendant $L_\infty$-algebras.

\bep \label{QJRD}
Let $S \in \Bbbk[q^1, \cdots, q^m]$ be a multi-variable polynomial. The descendant $L_\infty$-algebra 
of the cochain complex $(\cA, \cdot, K)_{\rho_S}=(\cA_{\rho_S},\cdot, K_{\rho_S})$ associated to
the quantum Jacobian Lie algebra representation $\rho_S$ in \QJR\ is a differential graded Lie algebra over $\Bbbk$,
i.e. $\ell_3^{K_{\rho_S}} = \ell_4^{K_{\rho_S}}=\cdots =0$. Moreover, $(\cA_{\rho_S}, \cdot, \ell_2^{K_{\rho_S}})$ is a Gerstenhaber algebra over $\Bbbk$.
\eep

\begin{proof}
The equality \dl\ implies the result, since $K_{\rho_S}$ is a differential operator on $\cA_{\rho_S}$ of order 2.
For the second claim, we have to show that $\ell_2^{K_{\rho_S}}$ is a derivation of the product:
\eqn\psn{
\ell_2^{K_{\rho_S}}(a \cdot b, c)= (-1)^{|a|} a \cdot \ell_2^{K_{\rho_S}}(b, c) +(-1)^{|b|\cdot |c|} \ell_2^{K_{\rho_S}}(a,c)\cdot b,
}
for any homogeneous elements $a,b,c \in \cA_{\rho_S}$. This follows from a direct computation, which uses
again the fact that $K_{\rho_S}$ is a differential operator of order 2.
\end{proof}


Thus, by Proposition \ref{QJRD}, the descendant $L_\infty$-algebra $(\bR[x][\eta], \ell_2)$, where 
$$
\ell_2(u,v ):=\ell_2^{K_S}(u,v)= K_S(u\cdot v)
- K_S(u)\cdot v - (-1)^{|u|} u \cdot K_S(v),
$$
is a differential graded Lie algebra (no higher homotopy structure) which is quasi-isomorphic to $(H,0)$
where $H=H_{K_S}(\cA_S)$. 
But note that the descendant $L_\infty$-morphism
$\ud \phi^{\cC}$ from $(\bR[x][\eta], \ell_2)$ to $(\bR,0)$
does have a non-trivial higher homotopy structure: $\phi^\cC_3, \phi^\cC_4, \cdots$ do not vanish.
This higher structure governs the correlation of the period integral $\int_{-\infty}^{\infty} u(x)e^{S(x)} dx$.
Let $\{ e_0, \cdots, e_{d-1} \}$ be an $\bR$-basis of $H$.
We define an $L_\infty$-morphism $\ud f=f_1, f_2, \cdots,$ from
$(H,0)$ to $(\bR[x][\eta], \ell_2)$ by
$$
f_1(e_i) = x^i,\ \text{for} \ i=0, 1, \cdots, d-1, \quad  f_2= f_3 = \cdots=0,
$$
which is clearly an $L_\infty$-quasi-isomorphism.
The version of Theorem \ref{thirdtheorem} for the toy model (which follows from Proposition \ref{comLi}
and Theorem \ref{homotopyperiod}) can be summarized as the following commutative diagram:
\eqn\toycomp{
\xymatrixcolsep{5pc}\xymatrix{(H, \ud 0) 
\ar@{->}@/^2pc/[rr]^{C= \phi^{\cC}_1 \circ f_1}
\ar@{..>}@/_4pc/[rr]_{\ud \kappa=\ud \phi^{\cC} \bullet \ud f}
\ar@{..>}@/_2pc/[r]^{\ud f}   \ar[r]^{f_1} &
(\bR[x][\eta], \ell_2)
  \ar@{..>}@/_2pc/[r]^{\ud \phi^{\cC}}  \ar[r]^{\phi_1^{\cC}} & (\bR,\ud 0)}
  ,
}
Then $\S(\ud t)_{\ud f} := \sum_{n=1}^\infty \frac{1}{n!}\sum_{\alpha_1, \cdots, \alpha_n} t^{\alpha_n} \cdots t^{\alpha_1}\otimes f_n(e_{\alpha_1}, \cdots, e_{\alpha_n})=\sum_{i=0}^{d-1} t_i \cdot x^i \in \bR[x] [\ud t]$, where 
$\{ t^\alpha \} = \{t_0, \cdots, t_{d-1}\}$
is an $\bR$-dual basis to $\{ e_0, \cdots, e_{d-1} \}$. Hence the generating power series for $\cC$ and $\ud f$ is 
$$
\cZ_{[\ud \phi^\cC \bullet \ud f]}(\ud t)=
\cC(e^{\S(\ud t)_{\ud f}} -1) =\exp \left(\sum_{n=1}^{\infty} \frac{1}{n!} 
\phi^{\cC}_n \left(\sum_{i=0}^{d-1} t_i \cdot x^i,  \cdots,  \sum_{i=0}^{d-1} t_i \cdot x^i   \right )\right) -1,
$$
 which is an $L_\infty$-homotopy invariant (Theorem \ref{homotopyperiod}). 
Lemma \ref{onetensor} says that there 
exists a power series $T^{[i]} (\ud t)=T^{[i]} (\ud t)_{\ud f} \in \bR[[\ud t]]$ for each $i=0, 1, \cdots, d-1$ such that 
\be
\cZ_{[\ud \phi^\cC \bullet \ud f]}(\ud t) = \sum_{i=0}^{d-1} T^{[i]} (\ud t) \cdot 
\left( \int_{-\infty}^{\infty} x^i e^{S(x)} dx \right).
\ee
Note that the 1-tensor $T^{[i]}(\ud t)$ has all the information of the integrals $\int_{-\infty}^{\infty} x^m e^{S(x)} dx$ for $m \geq d$, and 
is completely determined by the cochain complex $\big(\bR[x][\eta], $ $(\prt{}{x} + S'(x)) \prt{}{\eta}\big)$
with super-commutative multiplication. The Euclidean algorithm enhanced with quantum component in
Proposition \ref{QJRD}, which we generalize to the Griffiths period integral (Lemma \ref{memcon}), can be used to compute $T^{[i]}(\ud t)$ effectively and consequently compute
all the moments $\int_{-\infty}^{\infty} x^m e^{S(x)} dx,  \forall m \geq d$ from the finite data 
$$
\int_{-\infty}^{\infty}e^{S(x)} dx,
\int_{-\infty}^{\infty} x \cdot e^{S(x)} dx, \cdots, \int_{-\infty}^{\infty} x^{d-1} e^{S(x)} dx.
$$
Our general
theory says that this determination mechanism is governed by $L_\infty$-homotopy theory, more precisely, the interplay
between the 1-tensor $T^\g(\ud t)$ and the 3-tensor $A_{\a,\b}{}^\g(\ud t)_\S$ (Proposition \ref{Mmm}).

%

\subsection{The Lie algebra representation $\rho_X$ associated to a smooth projective hypersurface $X_G$} \label{subs4.1}

Let $n$ be a positive integer. We use $\underline x = [x_0, x_1, \cdots, x_n]$ as a projective coordinate of the projective
$n$-space $\BP^n$.
Let $G(\underline{x}) \in \bC[\underline x]$ be the defining homogeneous polynomial equation of degree $d \geq 1$ for a smooth projective hypersurface, denoted $X=X_G$,
of $\BP^n$. 
Let $\frg$ be an abelian Lie algebra of dimension $n+2$. Let $\alpha_{-1}, \alpha_0, \cdots, \alpha_{n}$ be a $\bC$-basis of $\frg$. Let 
\be
A:= \bC[y, x_0, \cdots, x_n]=\bC[y, \underline x]
\ee
be a commutative polynomial $\bC$-algebra generated by $y, x_0, \cdots, x_n$.
We also introduce variables $y_{-1}=y, y_0=x_0, y_1=x_1, \cdots, y_n=x_n$ for notational convenience.
For a given $G(\underline x)$ of degree $d$, we associate a Lie algebra representation $\rho_{X}=\rho_{X_G}$ on $A$ of 
$\frg$ as follows:
\eqn\quantumrep{
\rho_i:=\rho_X(\alpha_i) := \pa{y_i}+\prt{ S(y, \underline x) }  {y_i}, \text { for $i=-1, 0, \cdots, n$},
}
where $S(y, \underline x) = y \cdot G(\underline x)$. 
In other words, this is the quantum Jacobian Lie algebra representation $\rho_{S(y,\ud x)}$ 
associated to $S(y, \ud x)=y\cdot G(\ud x)$ in \QJR.
Of course, we extend this $\bC$-linearly to
get a map $\rho_X: \frg\to\End_{\bC}(A)$. This is clearly a Lie algebra representation of $\frg$. We will show that Griffiths' period integrals of the projective hypersurface $X_G$ are, in fact, period integrals of $\rho_X$ in the sense of Definition $\ref{Liedef}$ .

\subsection{Period integrals attached to $\rho_X$} \label{subs4.2}

We briefly review Griffiths' theory and find a nonzero period integral attached to $\rho_X$ . 
\bep \label{prop5}
Let $\bC$ be a non negative integer. Every rational differential $n$-form $\omega$ on $\BP^n$ with a pole order of $\leq k$ along $X_G$ (regular outside $X_G$ with a pole order of $\leq k$) can be written as
\be
\omega = \frac{F(\underline x)}{G(\underline x)^k} \Omega_n,
\ee
where $\Omega_n = \sum_{i=0}^n (-1)^i x_i (dx_0\wedge \cdots \wedge \hat {d x_i} \wedge \cdots \wedge dx_n)$ and $F$ is a homogeneous polynomial such that $\deg F + n+1= k\deg G=kd $.
\eep
Griffiths defined a surjective $\bC$-linear map, called the tubular neighborhood map $\tau$, (3.4) in \cite{Gr69}
\be
\tau: H_{n-1}(X_G, \bZ) \to H_n(\BP^n - X_G, \bZ),
\ee
where $H_i$'s are singular homology groups of the topological spaces $X_G(\bC)$ and $\BP^n(\bC)-X_G(\bC)$. It is known that this map is an isomorphism if $n$ is even. He also studied the following $\bC$-linear map, called the \textit{residue map} 
\be
Res: \cH(X_G)&\to& H^{n-1}(X_G, \bC),\\
\omega &\mapsto & \big( \g \mapsto \frac{1}{2 \pi i}  \int_{\tau(\g)} \omega \big)
\ee
where $\cH(X_G)$ is the rational de Rham cohomology group defined as the quotient of the group of rational $n$-forms on $\BP^n$ regular outside $X_G$ by the group of the forms $d \psi$ where $\psi$ is a rational $n-1$ form regular outside $X_G$. It turns out that there is an increasing filtration of $\cH(X_G)$ (see (6.1) of \cite{Gr69}):
\be
 \cH_1(X_G) \subset  \cdots \subset \cH_{n-1}(X_G) \subset \cH_n(X_G) = \cH_{n+1}(X_G)= \cdots
\simeq \cH(X_G),
\ee
where $\cH_k(X_G)$ is the cohomology group defined as the quotient of the group of rational $n$-forms on $\BP^n$
with a pole of order $\leq k$ along $X_G$ by the group of exact rational $n$-forms on $\BP^n$ with a pole
of order $\leq k$ along $X_G$. The isomorphism $\cH_{n}(X_G)
\simeq \cH(X_G)$ follows from Theorem 4.2, \cite{Gr69} and the fact that the natural 
map $\cH_{k}(X_G) \to \cH_{k+1}(X_G)$ is injective follows from Theorem 4.3, \cite{Gr69}.
Moreover, Griffiths proved that $Res$ sends this filtration given by pole orders to the Hodge filtration of $H^{n-1}(X_G, \bC)$. For each $k \geq 1$, if we define $\cF_{k} \subset H^{n-1}(X_G, \bC)$ to be the $\bC$-vector space consisting of 
all $(n-1, 0), (n-2, 1), \cdots, (n-1-k, k)$-forms. Then
\eqn\inhodge{
\cF_{0} \subset \cF_{1} \subset \cdots \subset \cF_{n-2} \subset \cF_{n-1} = H^{n-1}(X_G, \bC).
}
Theorem 8.3, \cite{Gr69}, says that $\cH_n(X_G)$ goes into the primitive part
of $\cF_{k-1}$ under $Res$ and $Res$ is a $\bC$-vector space isomorphism from $\cH(X_G)$ to
the primitive part $H^{n-1}_{\pr} (X_G, \bC)$ of $H^{n-1}(X_G, \bC)$. 

Now we construct a $\bC$-linear map $C_\g: A \to \bC$ for each $\g \in H_{n-1}(X_G, \bZ)$, which is a period integral attached to $\rho_X$.
For $y^{k-1} F(\underline x)$ where $F(\underline x)$ is a homogeneous polynomial of degree
$dk-(n+1)$ and $k \geq 1$, define
\be
C_\g (y^{k-1} F(\underline x)) &=& -\frac{1}{2 \pi i} \int_{\tau(\g)}\bigg( \int_{0}^{\infty} y^{k-1} F(\underline x) \cdot e^{y G(\underline x)}  dy\bigg) \Omega_n\\
&=&\frac{(-1)^{k-1}(k-1)!}{2 \pi i} \int_{\tau(\g)}\frac{F(\underline x)}{G(\underline x)^k} \Omega_n.
\ee
In the second equality, we used the Laplace transform:
\be
 \int_{0}^{\infty} y^{k-1} e^{y T}dy = (-1)^k \frac{(k-1)!}{T^k}.
\ee
Note that $\frac{F(\underline x)}{G(\underline x)^k} \Omega_n$ is a representative
of an element of $\cH_k(X_G)$ and $\int_{\tau(\g)}\frac{F(\underline x)}{G(\underline x)^k} \Omega_n$ is well-defined.
If $x \in A$ is of the form $y^{k-1}F(\ud x)$, but with $F(\ud x)$ 
homogeneous of degree not equal to $dk-(n+1)$, we simply define $C_\g(x)=0$. Because the elements $y^{k-1} F(\underline x)$, where $k \geq 1$ and $F(\underline x)$ varies over any homogeneous polynomials in $\bC[\underline x]$, constitute a $\bC$-basis of $A$, the above procedure, extended $\bC$-linearly, gives a map $C_\g: A \to \bC$. Then we claim that this is a period integral attached to the Lie algebra representation $\rho_X$.

\bep \label{main}
Let $\g \in H_{n-1}(X_G, \bZ)$. The $\bC$-linear map $C_\g$ is a period integral attached to $\rho_X$, i.e.
$C_\g \left(\rho_i (f)\right) =0$ for every $f \in A$. 
\eep
\begin{proof}
Recall that
\be
\rho_{-1}&=& G(\underline x)  + \pa{y},\\
\rho_{i} &=&y \pa{x_i} G(\underline x) + \pa{x_i}, \quad i=0, \cdots, n.
\ee
It is enough to check the statement when $\rho_i(x)$ is a $\bC$-linear combination of the forms $y^{k-1} F(\underline x)$, where $k \geq 1$ and $F(\underline x)$ varies over any homogeneous polynomials in $\bC[\underline x]$ such that $\deg F + n+1 = k d$, since $C_\g$ is already zero for other forms of polynomials.
Let $f( \underline x) \in \bC[\underline x]$ be a homogeneous polynomial of degree $kd -n$. Then for each $k \geq 1$ we have
\be
\rho_i (y^{k-1} f(\underline x)) = \prt{G(\underline x)}{x_i} \cdot y^{k} f(\underline x) +y^{k-1}\prt{ f(\underline x)}{x_i}.
\ee
We compute
$$
\eqalign{
2 \pi i \cdot C_{\g} \biggl(\prt{G(\underline x)}{x_i} &\cdot y^{k} f(\underline x) +y^{k-1}\prt{ f(\underline x)}{x_i}\biggr)
\cr
=& (-1)^{k}k!\int_{\tau(\g)}\frac{\prt{G(\underline x)}{x_i} \cdot f(\underline x)}{G(\underline x)^{k+1}} \Omega_n +(-1)^{k-1}(k-1)!\int_{\tau(\g)}\frac{\prt{ f(\underline x)}{x_i}}{G(\underline x)^{k}} \Omega_n
\cr
=&  (-1)^k (k-1)!\int_{\tau(\g)}  \frac{k\cdot f(\underline x) \cdot\prt{G(\underline x)}{x_i} 
-G(\underline x)\cdot \prt{f(\underline x)}{x_i}} {G^{k+1}}\Omega_n
\cr
=& \int_{\tau(\g)} d \left( \frac{(-1)^k (k-1)!}{G(\underline x)^k} \sum_{0 \leq h < j \leq n} 
\left(x_h A_j(\underline x) - x_j A_h(\underline x)\right) (\cdots \hat{d x_h} \cdots \hat{d x_j} \cdots)   \right) ,
}
$$
where $A_i(\underline x) = f(\underline x)$ and $A_j (\underline x) =0$ for $j \neq i$. The last equality is a simple differential calculation, which can be found in (4.4) and (4.5) of \cite{Gr69}.
Therefore this expression has the form $\int_{\tau(\g)} d \omega$ and this is zero, which is what we want.
Now we look at  
\be
\rho_{-1}(y^{k-1} f(\underline x))=G(\underline x)
y^{k-1} f(\underline x)+ (k-1)y^{k-2} f(\underline x).
\ee
Then we can similarly check (this is much easier and follows from the factor $(-1)^{k-1}(k-1)!$ and sign) that
\be
C_\g \big(G(\underline x)y^{k-1} f(\underline x)+ (k-1)y^{k-2} f(\underline x)\big) =0,
\ee
for all $k \geq 1$ and any homogeneous polynomial in $\bC[\underline x]$.
 \end{proof}

\subsection{The commutative differential graded algebra attached to a smooth hypersurface} \label{subs4.3}

We now have two data: the Lie algebra representation $\rho_X$ attached to 
the smooth projective hypersurface $X_G$ and a period integral $C_\g: A \to \bC$ attached to $\rho_X$ for $\g \in H_{n-1}(X_G,\bZ)$. In subsection \ref{subs2.2}, we constructed  a cochain complex $(\cA_{\rho_X}, \cdot, K_{\rho_X})$ with super-commutative product $\cdot$ 
whose degree 0 part is $A$ (in fact, an object in the category $\frC_\bC$) and a cochain map 
$(\cA_{\rho_X}, \cdot, K_{\rho_X}) \to (\bC,\cdot, 0)$. 
We introduce another Lie algebra representation $\pi_X$ of the abelian Lie algebra $\frg$ of dimension $n+2$ related to $X_G$.
Let 
\be
\pi_i:=\pi_X (\alpha_i):= \prt{S(y, \underline x)}{y_i}, \quad i=-1, 0, \cdots, n.
\ee
This defines a representation $\pi_i: \frg \to \End_\bC(A)$. The same procedure gives a cochain complex
$(\cA_{\pi_X}, K_{\pi_X})$. From now on, we denote $K_{\pi_X}$ by $Q_X$. Note that the $\bC$-algebra $\cA$ appearing in the cochain complexes attached to $\rho_X$ and $\pi_X$ should
be the same, since we use the same Lie algebra $\frg$. But the differentials $K_X:=K_{\rho_X}$ and $Q_X:=K_{\pi_X}$ are different.
 The differentials $K$ and $Q$ are given as follows:
\be
\cA_X&:=&\cA_{\rho_X}= \bC[\ud y][\ud \eta]=\bC[y_{-1}, y_0, \cdots, y_n][\eta_{-1},\eta_0, \cdots, \eta_n] \\
K_X&:=&K_{\rho_X}=\sum_{i=-1}^n \left(\prt{S(y, \underline x)}{y_i} + \pa{y_i}\right) \pa{\eta_i}:\cA\to \cA,\\
Q_X&:=&K_{\pi_X}=\sum_{i=-1}^n \prt{S(y, \underline x)}{y_i} \pa{\eta_i}: \cA \to \cA.
\ee

Since $\prt{S(y, \underline x)}{y_i} \pa{\eta_i}$ is a differential operator of order 1, the differential $Q_X$ is a derivation of the product of $\cA_X$. Thus $(\cA_X, \cdot, Q_X)$ is a CDGA (commutative
differential graded algebra). But $K$ is \textit{not a derivation of the product}, as we have already pointed out: the differential operator $\pa{y_i} \pa{\eta_i} $ has order 2.
We also introduce the $\bC$-linear map
\be
\Delta:= K_X-Q_X=\sum_{i=-1}^n  \pa{y_i} \pa{\eta_i}:\cA \to \cA.
\ee
Note that $\Delta$ is a also a differential of degree 1, i.e. $\Delta^2=0$. Therefore we have
\be
\Delta  Q_X +Q_X  \Delta =0.
\ee
It follows from this and Proposition \ref{QJRD} that $(\cA_X, \cdot, K_X, \ell_2^{K_X})$ satisfies all the axioms in Definition \ref{bvd}, and
thus we get the following theorem:
\bet\label{bvtheorem}
The triple $(\cA_X, \cdot, K_X, \ell_2^{K_X})$ is a BV algebra over $\bC$.
\eet

\bep \label{Qcohomology}
Let $H_{Q_X}^n(\cA_X)$ be the $n$-th cohomology group
of the CDGA $(\cA_X,\cdot, Q_X)$. Then $H^{0}_{Q_X}(\cA_X)$ has an induced $\bC$-algebra structure and we have 
\be
H^{0}_{Q_X}(\cA_X) \simeq \bC[y, \underline x]/ J_S,
\ee
where $J_S$ is the Jacobian ideal defined as the ideal of $A=\bC[y, \underline x]$ generated by $G(\underline x), y\prt{G}{x_0}$, $\cdots, y\prt{G}{x_n}$.
\eep

\begin{proof}
This is clear from the construction.
\end{proof}

 \bep \label{modstr}
 The cohomology group $H^{-1}_{Q_X}(\cA_X)$ is both an $\cA^0$-module and an $H^{0}_{Q_X}(\cA_X)$-module. 
 \eep
 
 \begin{proof}
 We consider $R \in \cA^{-1}_X$ such that $Q_X R =0$. For any $f \in \cA^0_X=A$, we have
 $Q_X(f \cdot R) =0$, since $A \subseteq \Ker Q_X$. Let $S=Q_X \sigma$ where $\sigma \in \cA_X^{-2}$.
 Then we have
 \be
 S \cdot f = Q_X(\sigma \cdot f), \quad \text{ for } f \in \cA_X^0.
 \ee
 Therefore $H^{-1}_{Q_X}(\cA_X)$ is an $A$-module.
 Note that $H^0_{Q_X}(\cA_X)$ has a $\bC$-algebra structure inherited from $\cA$, since $Q_X$ is a 
 derivation of the product of $\cA_X$. Then one can similarly check that $H^{-1}_{Q_X}(\cA_X)$ is also a $H^0_{Q_X}(\cA_X)$-module.
  \end{proof}

But notice that $H^{-1}_{K_X}(\cA_X)$ is not an $\cA^0_X$-module under the product $\cdot$ of $\cA$. 
Indeed, consider $R \in \cA^{-1}$ such that $K_X (R)=0$.
For any $f\in \cA_X^0$, the equation \elltwo\ says that
\eqn\deltainclusion{
K_X( R\cdot f) =  \ell^{K_X}_2(R, f)+K_XR \cdot f.
}
Because $\ell^{K_X}_2$ is not zero, $H^{-1}_{K_X}(\cA_X)$ does not necessarily have an $A$-module structure.
In fact, this will play an important role in understanding the complex $(\cA_X, K_X)$.

\subsection{The BV algebra attached to a smooth hypersurface} \label{subs4.31}

Here we prove parts $(a),(b)$, and $(e)$ of Theorem \ref{firsttheorem}. 
We drop $X$ from the notation for simplicity if there is no confusion; $(\cA, \cdot, K)=(\cA_X, \cdot, K_X)$.
We start by recalling the decomposition of $\cA$ in \depa;
$$
\cA=\bigoplus_{\gh,{\ch}, \wt} \cA_{{\ch}, (\wt)}^{gh}= \bigoplus_{-n-2\leq j \leq 0}\bigoplus_{w \in \Z^{\geq 0}}\bigoplus_{{\l} \in \Z}\cA^j_{{\l},(w)}.
$$
Associated with the charges ${\ch}$, we define the corresponding Euler vector field
$$
\hat E_{\ch} = -d y\frac{\rd}{\rd y} + \sum_{i=0}^n x_i \frac{\rd}{\rd x_i}
+d \eta_{-1}\frac{\rd}{\rd \eta_{-1}} - \sum_{i=0}^n \eta_i \frac{\rd}{\rd \eta_i}
$$
Associated with the weights $\wt$, we
define the corresponding Euler vector field
$$
\hat E_{\wt} =  y \frac{\rd}{\rd y} +\sum_{i=0}^n\eta_i \frac{\rd}{\rd \eta_i} .
$$
 Then $u \in \cA^j_{{\l},(w)}$ if and only if $\hat E_{\ch} (u) = \l \cdot u$, $\hat E_{\wt} (u) = w \cdot u$, and $\gh(u)=j$.
Note that $Q$ preserves the charge and the weight, and commutes with $\hat E_{\ch}$
and $\hat E_{\wt}$.
 The differential $K$ also commutes with $\hat E_{\ch}$ and preserves the charge but $K$ does \textit{not} preserve the weight. The operator $\Delta$ decreases the weight $\wt$ by 1.
Also note that $\gh(S(\ud y))=0$, ${\ch}(S(\ud y))=\underline{0}$ and $\wt(S(\ud y))=1$.

If we define
\eqn\RRR{
R:=- d \cdot y \eta_{-1}+ \sum_{i=0}^n x_i \eta_i  \in \cA^{-1},
}
then 
\be
Q R = \big( \sum_{i=0}^n x_i \pa{x_i} - d\cdot y \pa{y}  \big) S(y, \underline x)=0, 
\ee
which follows from the fact that $G(\underline x)$
is a homogeneous polynomial of degree $d$. 
Moreover, $R$ can not be $Q$-exact for degree reasons. Then a straightforward computation says that
\be
KR = n+1 -d.
\ee

We define a $\bC$-linear map $\delta_R: \cA \to \cA$
by 
\eqn\delr{
\delta_R (x) =   \ell^K_2(R, x)+ KR \cdot x =  \ell_2^K(R, x)+ (n+1-d) \cdot x, \quad x \in \cA.
}
It is clear that $\delta_R$ preserves the ghost number, since
$R \in \cA^{-1}$ and $\ell_2^K$ is a degree one map; for any $m \in \bZ$, we have $\delta_R:\cA^m\to \cA^m$.
Proposition \ref{QJRD} implies that $\ell_2^K$ is a derivation of the product;
\be
\ell_2^K(a \cdot b, c)= (-1)^{|a|} a \cdot \ell_2^K(b, c) +(-1)^{|b|\cdot |c|}  \ell_2^K(a,c)\cdot b,
\ee
for any homogeneous elements $a,b,c \in \cA$.
Using this one can compute that
\be
\ell_2^K(R,F)= -d \cdot\left(y\prt{F}{y}- \eta_{-1}\prt{F}{\eta_{-1}}\right) + 
\sum_{i=0}^n \left(x_i \prt{F}{x_i}  - \eta_i\prt{F}{\eta_i}\right)=\hat E_{\ch}(F),
\ee
for any $F \in \cA$.

\bel \label{tworelations}
The map $\delta_R$ preserves the degree of $\cA$. Moreover, we have 
\be
\delta_R \circ K &=& K\circ \delta_R \\
\delta_R(\cA^m) \cap \Ker K&\subseteq& K(\cA^{m-1}) \subseteq \cA^m
\ee
 for each $m \in \bZ$.
\eel

\begin{proof}
 We compute, for $x \in \cA$,
\be
K(\delta_R (x)) &=& K\big(\ell_2^K(R, x)+ KR \cdot x \big)\\ 
&=&  K\big(K(R\cdot x)+ R\cdot Kx - KR \cdot x \big)  + K ( KR \cdot x)  \\
&=& K(R \cdot Kx)\\
&=& \ell_2^K( R, Kx) + KR \cdot Kx  = \delta_R (Kx).
\ee
Note that $\delta_R (x)= \hat E_{\ch}(x) + KR \cdot x=\hat E_{\ch} (x)+ (n+1-d)x$ for $x\in \cA$ and 
\be
K(R\cdot x) = \ell_2^K(R, x) + KR\cdot x -R \cdot Kx=\delta_R(x) - \cdot Kx, \quad x\in \cA^m.
\ee
This implies that $\delta_R(\cA^m_{c}) \cap \Ker K\subseteq K(\cA^{m-1}_{c})$ for each $m \in \bZ$ and any charge $c \in \Z$, since $\delta_R (u)= \left(c-d+(n+1)\right) \cdot u$
if and only if $u \in \cA_c$.
 \end{proof}

We define the background charge $c_X$ 
of $(\cA, \cdot, K)$ by
$$
c_X := d-(n+1).
$$
Then it is clear that $\Ker \delta_R = \cA_{c_X}$.
\bel \label{czero}
The pair $(\Ker \delta_R, K)=(\cA_{c_X}, K)$ is a cochain complex and the natural inclusion map from $(\Ker \delta_R, K)$ to $ (\cA, K)$ is a quasi-isomorphism. In other words, the $K$-cohomology is concentrated in the background charge $c_X$. 
\eel

\begin{proof}
The relation $\delta_R \circ K =K\circ \delta_R $, Lemma \ref{tworelations}, says that if $x \in \Ker (\delta_R)$ then $Kx \in \Ker(\delta_R)$.
Thus $K$ is a $\bC$-linear map from $\Ker \delta_R$ to $\Ker \delta_R$. Since $K^2=0$,  we see that $(\Ker \delta_R, K)$ is a cochain complex. If we index $K$ by $K_m: \cA^{m} \to \cA^{m+1}$ for each $m\in \bZ$, then the inclusion map  from $(\Ker \delta_R, K)$ to $ (\cA, K)$ induces a $\bC$-linear map
\be
H^m_K(\Ker(\delta_R)):=\frac{\Ker(K_{m}) \cap \Ker (\delta_R)}{K_{m-1}(\cA^{m-1}) \cap \Ker(\delta_R)} \lra \frac{\Ker(K_{m})}{K_{m-1}(\cA^{m-1})}=: H^m_K(\cA).
\ee
Injectivity is immediate from the definitions. Surjectivity follows from the decomposition $\cA = \Ker(\delta_R) \oplus \im(\delta_R)$ and $\delta_R(\cA^m) \cap \Ker K\subseteq K_{m-1}(\cA^{m-1})$ in Lemma \ref{tworelations}.
Therefore we conclude that the inclusion map is a quasi-isomorphism. 
 \end{proof}

%

Let us denote the complex $(\Ker \delta_R, K)=(\cA_{c_X},K)$ by $(\cB,K)$. Then
\be
\cB = \cA_{c_X}= \bigoplus_{m\in \bZ} \cB^{m}=\cB^{-n-2} \oplus \cdots \oplus \cB^{0}
\ee
where $\cB^{m}$ is the degree $m$ (ghost number $m$) part of $\cB$. We use this complex $(\cB, K)$ to relate the 0-th cohomology group of $(\cA, K)$ to the middle dimensional primitive cohomology of the smooth projective hypersurface $X_G$.
The main result of this subsection is the following:
\bep\label{mainhodge}
Let $H_K^n(\cA)$ be the $n$-th cohomology group of the cochain complex $(\cA,K)=(\cA_X, K_X)$.
Then 
$H_K^0(\cA)$ is isomorphic to $\H$ as a $\bC$-vector space, where $\H=H_{\pr}^{n-1}(X_G, \bC)$. 
\eep

\begin{proof}
A simple computation shows that $\cB^0$ is spanned (as a $\bC$-vector space) by homogeneous
polynomials of the form $y^{k-1} F(\underline x)$, where the degree of $F(\underline x)$
is $kd- (n+1)$ with $k \geq 1$. Then we define a $\bC$-linear map $J$ by
$$
\eqalign{
J:\cB^0 &\to H^{n-1}(X_G, \bC) 
\cr
y^{k-1} F(\underline x)  &\mapsto \left\{\g \mapsto 
-\frac{1}{2\pi i} \!\! \int_{\tau(\g)} \!\!\!\left(\int_{0}^{\infty} y^{k-1}e^{yG(\underline x)} dy \right) F(\underline x) \Omega_n \right\},
}
$$
and extending it $\bC$-linearly. Proposition \ref{main} says that $K(\cB^{-1})$ goes to zero under the map $J$ and
so $J$ induces a $\bC$-linear map $H^{0}_K(\cB) \to H^{n-1}(X_G, \bC)$. Now recall that
$\cH(X_G)$ was defined to be the rational De Rham cohomology
group defined as the quotient of the group of
rational $n$-forms on $\BP^n$ regular outside $X_G$ by the group of exact rational $n$-forms on $\BP^n$
regular outside $X_G$. Theorem 8.3, \cite{Gr69}, tells us that the residue map $Res$ induces an isomorphism between $\cH(X_G)$ and $H_{\pr}^{n-1}(X_G, \bC)$. Thus, to prove the proposition, it is enough to show that the following map (extended $\bC$-linearly)
\eqn\Jprime{
\eqalign{
J':\cB^0 &\to \Omega(V)^n
\cr
y^{k-1} F(\underline x)  &\mapsto -\int_{0}^{\infty} y^{k-1}e^{yG(\underline x)}  dy \cdot F(\underline x)\Omega_n= (-1)^{k-1}(k-1)!\frac{F(\underline x)}{G(\underline x)^k} \Omega_n,
}
}
where $\Omega(V)^n$ is the group of
rational $n$-forms on $\BP^n$ regular outside $X_G$, induces an isomorphism $J':H^{0}_K(\cB) \to \cH(X_G)$, i.e. $J$ factors through the isomorphism $J'$.
This follows from Corollary 2.11, (4.4), and (4.5) in \cite{Gr69} with a computation below. An arbitrary homogeneous (as a polynomial in $\ud x$ and $y$) element
of $\cB^{-1}$ can be written as $\Lambda = \sum_{i=0}^n A_i (y, \underline x) \eta_i + B(y, \underline x) \eta_{-1}$,
where $A_i (y, \underline x)=y^k \cdot M_i (\underline x) $ and $B(y, \underline x)=y^l \cdot N(\underline x)$ are homogeneous polynomials of $A=\bC[y, \underline x]$.
Then we have
\be
K \Lambda &=& \sum_{i=0}^n A_i(y, \underline x) \prt{G}{x_i} y + G(\underline x) B(y, \underline x) +\sum_{i=0}^n 
\prt{A_i(y, \underline x)}{x_i} + \prt{B(y, \underline x)}{y}\\
&=& \sum_{i=0}^n y^{k+1}M_i(\underline x) \prt{G}{x_i} +\sum_{i=0}^n 
y^k\prt{M_i(\underline x)}{x_i} + G(\underline x) y^l N( \underline x) + l y^{l-1} N(\underline x).
\ee
If we apply $J'$ to $K \Lambda$, then a simple computation shows that 
\eqn\thisr{
J'(K\Lambda) = k! (-1)^{k-1} \frac{(k+1)\sum_{i=0}^n M_i(x) \prt{G(\underline x)}{x_i} - G(\underline x) \sum_{i=0}^n
\prt{M_i(\underline x)}{x_i}}{G(\underline x)^{k+2}} \cdot \Omega_n.
}
Note that $J'\big(G(\underline x) y^l N( \underline x) + l y^{l-1} N(\underline x)\big)=0$. The relation (4.5), \cite{Gr69} says that $J'(K\Lambda)$ is an exact rational differential form. Thus $J'$ induces a $\bC$-linear map $J':H^{0}_K(\cB) \to \cH(X_G)$. 
The surjectivity of $J'$ follows from Proposition \ref{prop5}.
Griffiths showed that any exact differential $n$-form 
$\frac{u(\ud x)}{G(\ud x)^k}\in \Omega(V)^n, k\geq 1,$ can be written
as\footnote{This fact fails to hold when $X$ is not smooth: for a singular projective hypersurface $X_G$, Dimca proved that $\frac{u(\ud x)}{G(\ud x)^k} = d \left( \frac{\beta}{G(\ud x)^{k+(n+1)m}}\right)$ for some positive integer $m$, in \cite{Dim91}. } 
$$
\frac{u(\ud x)}{G(\ud x)^k} = d \left( \frac{\beta}{G(\ud x)^{k-1}}\right)
$$
for some $(n-1)$-form $\beta$ on $\BP^n-X_G$:see the theorem 4.3, \cite{Gr69}.
Because the right hand side of \thisr\ is exactly of the form $d \left( \frac{\beta}{G(\ud x)^{k-1}}\right)$ above, exact differential $n$-forms inside $\O(V)^n$ match precisely with the image of $K(\cB^{-1})$ under the map $J'$. Thus $J'$ is injective.
 \end{proof}

We see that Theorem \ref{bvtheorem}, Proposition \ref{mainhodge}, and Proposition \ref{Qcohomology} give the proof of parts $(a)$ and $(b)$ of Theorem \ref{firsttheorem}.
Note that part (e) of Theorem \ref{firsttheorem} is straightforward by defining
 \eqn\qft{
\cC_\g(x)
 :=\vpar{ 0 \quad }{ x\in \bigoplus_{i \leq -1}  \cA^i}
{ C_\g(x) \quad } { x \in   \cA^0}
}
 
 It is a simple computation that this definition matches with \fpint\ and $\cC_\g:(\cA,K)\to (\bC,0)$ is a cochain map which induces $C_\g$ after taking the $0$-th cohomology.

\subsection{A cochain level realization of the Hodge filtration and a spectral sequence}\label{subs4.32}
This subsection we prove part (c) of Theorem \ref{firsttheorem} and construct a certain spectral sequence. Let us define a decreasing filtration $F^\bullet$ on $(\cA, \cdot, K)=(\cA_{\rho_X}, \cdot, K_{\rho_X})$
by using the weight grading $\wt$, such that the isomorphism $Res \circ J' : H^0_K(\cA) \to
\H  $, where $J'$ is given in \Jprime, sends $F^\bullet$ to the decreasing Hodge filtration on $\H=H^{n-1}_{\pr}(X_G,\C)$.
Then we analyze a spectral sequence, which we call \textit{the classical to quantum
spectral sequence}, associated to the filtered complex $(F^\bullet\cA, K)$. Then
this spectral sequence gives a precise relationship between the $Q$-cohomology $H_Q(\cA)$ 
(which we view as classical cohomology) and the $K$-cohomology $H_K(\cA)$ (which
we view as quantum cohomology).

In this subsection we shift the degree of $(\cA, K)$ to consider $\cC=\cA[-n-2]$ so that 
$\cC^i = \cA^{i+n+2}$ for each $i \in \Z$. \footnote{The reason for ghost number shifting is to get a spectral
sequence in the first quadrant} Then we have a cochain complex $(\cC, K)$;
$$
0 \to \cC^0 \mapto{K} \cC^1 \mapto{K} \cdots \mapto{K} \cC^{n+2} \to 0.
$$ 
We define a filtered complex as follows;
$$
\cC^{} =: F^0 \cC^{} \supset F^1 \cC^{} \supset \cdots \supset F^{n-1} \cC^{} \supset F^{n} \cC^{}=\{0\}
$$
where the decreasing filtration is given by the weights
$$
F^i \cC^{} =\bigoplus_{k \leq n-1-i} \cC^{}_{(k)}, \quad i \geq 1. 
$$

The associated graded complex to a filtered complex $(F^p \cC^{}, K)$ is the complex
\eqn\filc{
Gr \cC^{} = \bigoplus_{p \geq 0} Gr^p \cC^{}, \quad Gr^p \cC^{} = \frac{F^p \cC^{}}{ F^{p+1} \cC^{}},
}
where the differential is the obvious one $d_0$ induced from $K$.
We observe that this 
differential $d_0$ is, in fact, induced from $Q$, because
we have  that $K=\Delta+Q$, $\Delta:F^p\cC^{} \to F^{p+1}\cC^{}$ and $Q:F^p\cC^{} \to F^{p}\cC^{}$.

The filtration $F^p \cC^{}$ on $\cC^{}$ induces a filtration $F^p H^{}_K(\cC)$ on the cohomology $H_K(\cC)$ by
$$
F^p H^q_K(\cC)= \frac{F^p Z^q}{F^p B^q},
$$
where $Z^q=\ker (K:\cC^q \to \cC^{q+1})$ and $B^q=K(\cC^{q-1})$. The associated graded cohomology is
\eqn\fild{
Gr H^{}_K(\cC)= \bigoplus_{p,q} Gr^p H^q_K(\cC), \quad 
Gr^p H^q_K(\cC)= \frac{F^p H^q_K(\cC)}{F^{p+1} H^q_K(\cC)}.
}

Then the general theory of filtered complexes implies that there is a spectral sequence $\{E_r, d_r\}$ $(r \geq 0)$ with
$$
E_0^{p,q}= \frac{F^p \cC^{p+q}}{F^{p+1} \cC^{p+q}}, \quad E_1^{p,q} = H^{p+q} (Gr^p \cC^{}), \quad  E^{p,q}_{\infty}= Gr^p H^{p+q}_K(\cC).
$$
Note that
$$
E_r = \bigoplus_{p,q \geq 0} E_r^{p,q}, \quad d_r:E_r^{p,q} \to E_r^{p+r, q-r+1}, d_r^2=0,
$$
and $H^{}(E_r)=E_{r+1}$.

\begin{proposition} \label{wspec}
The classical to quantum spectral sequence  $\{ E_r\}$ satisfies
$$
\eqalign{
E_1^{p,q}&\simeq Gr^p H_Q^{p+q}(\cC^{}),
\cr
E_2^{p,q} &\mapto{\simeq} E^{p,q}_{\infty}= Gr^p H^{p+q}_K(\cC).
}
$$
In particular, $\{ E_r\}$ degenerates at $E_2$.
\end{proposition}

\begin{proof}
Note that the differential $d_0$ of $E_0=Gr \cC^{}$ is induced from $Q$, because
we have  that $K=\Delta+Q$, $\Delta:F^p\cC^{} \to F^{p+1}\cC^{}$ and $Q:F^p\cC^{} \to F^{p}\cC^{}$. Using this observation, we compute
$$
\eqalign{
E_1^{p,q}&:=\frac{\{ a \in F^p \cC^{p+q}: \ K(a) \in F^{p+1} \cC^{p+q+1}\}}{ K(F^p\cC^{p+q-1})+ F^{p+1}\cC^{p+q}}
\cr
&\simeq 
\frac{\{ a \in F^p \cC^{p+q}: \ Q(a)=0\}}{ Q(F^p\cC^{p+q-1})+ F^{p+1}\cC^{p+q}}
\simeq
H_Q^{p+q} (Gr^p \cC^{}).
}
$$
Since $Q$ preserves the weight and $Q\circ \hat E_{\wt} =\hat E_{\wt} \circ Q$, we conclude that
$H_Q^{p+q} (Gr^p \cC^{}) \simeq Gr^p H_Q^{p+q}(\cC^{}).$

By a general construction of the spectral sequence  we can also describe $E_2^{p,q}$:
$$
E_2^{p,q}:=\frac{ \{ a \in F^p \cC^{p+q}: \ K(a) \in F^{p+2} \cC^{p+q+1}\} }
{ \left( K(F^{p-1}\cC^{p+q-1}) + F^{p+1} \cC^{p+q} \right) \bigcap  \{ a \in F^p \cC^{p+q}: \ K(a) \in F^{p+2} \cC^{p+q+1}\} }.
$$
Since $K(F^p \cC^{} \setminus F^{p+1}\cC^{} ) \subseteq F^{p} \cC^{} \setminus F^{p+2}\cC^{}$
and $F^p \cC^{p+q}/F^{p+1} \cC^{p+q} = \cC^{p+q}_{(n-1-p)}$, we have
$$
\eqalign{
E_2^{p,q}
&\simeq 
\frac{ \{ a \in F^p \cC^{p+q}: \ K(a)=0\} }
{ \left( K(F^{p-1}\cC^{p+q-1}) + F^{p+1} \cC^{p+q} \right) \bigcap  \{ a \in F^p \cC^{p+q}: \ K(a) =0\} }
\cr
&\simeq
\frac{\ker K \cap \cC^{p+q}_{({n-1}-p)}}{ \cC^{p+q}_{({n-1}-p)} \cap K(\cC^{p+q-1}_{({n-1}-p)} \oplus \cC^{p+q-1}_{({n-1}-p+1)})} 
\simeq
 \frac{F^p H^{p+q}_K(\cC)}{F^{p+1} H^{p+q}_K(\cC)}.
 }
$$
Since $\frac{F^p H^{p+q}_K(\cC)}{F^{p+1} H^{p+q}_K(\cC)} = Gr^p H^{p+q}_K(\cC)$, we are done.
 \end{proof}

Since $\cC[n+2]=\cA$, we also define a filtered complex $(F^\bullet \cA, K)$ in the same way:
$$
\cA^{} =: F^0 \cA^{} \supset F^1 \cA^{} \supset \cdots \supset F^{n-1} \cA^{} \supset F^{n} \cA^{}=\{0\}, \quad
F^i \cA^{}=\bigoplus_{k \leq n-1-i} \cA^{}_{(k)}, \quad i \geq 1. 
$$
Then we have 
$$H_K^{p}(\cC)=H_K^{p-n-2}(\cA), \quad 
Gr^p H^{p+q}_K(\cC)=Gr^p H^{p+q-n-2}_K(\cA), \quad p,q \geq 0.
$$

Recall the decreasing Hodge filtration on $\H=H_{\pr}^{n-1}(X_G, \bC)$:
\be
\cF^{q} \H = \bigoplus_{\substack{i+j=n-1\\i  \geq q}}H^{i,j}=H^{n-1,0}
\oplus H^{n-2, 1} \oplus \cdots \oplus H^{ q+1,n-q-2}\oplus H^{q,n-1-q},
\ee
where $H^{i,j}$ is the cohomology of $(i,j)$-forms on the hypersurface $X_G$.
Thus $\cF^0 \H=\H$ and $\cF^{n}\H=0$. Then the following proposition is clear.

\bep \label{homotopyhf}
The isomorphism $Res \circ J' : H^0_K(\cA) \to
\H $ sends $F^q H_K^0(\cA)$ to $\cF^q \H$ for each $q \geq 0$.
\eep

This proposition proves part $(c)$ of Theorem \ref{firsttheorem}.

\subsection{Computation of the $K_X$-cohomology of $(\cA_X, K_X)$}\label{subs4.40}

Here our goal is to compute $H^i_{K_X}(\cA_X)$ for every $i \in \bZ$ in addition to $H^0_{K_X}(\cA_X) \simeq \H$.
We achieve this by showing that 
$(\cA_X, K_X)$ is degree-twisted cochain isomorphic to a twisted de Rham complex 
appearing in \cite{AdSp06}. Adolphson and Sperber computed the cohomology of a certain de Rham type complex, which we briefly review now. They considered the complex of differential forms $\O^\bullet_{\bC[{y}, \ud x]/\bC}$ with boundary map $\partial_S$ defined
by $\partial_S(\omega) = d S \wedge \omega$ where $S=y \cdot G(\underline{x})$.
They introduced the bigrading $(\deg_1, \deg_2)$ on $\O^\bullet_{\bC[{y}, \ud x]/\bC}$ as follows;

\eqn\bigr{
\eqalign{
\deg_1(x_i)&=\deg_1(dx_i)=1,\quad i=0,\cdots, n, \qquad \deg_2 (x_i)=\deg_2(d x_i) = 0,\quad  i=0, \cdots, n,
\cr
\deg_1(y)&=\deg_1(dy)=-d,\qquad \qquad \qquad \quad   \deg_2(y)=\deg_2(dy)= 1.
 }
}
Then $(\O^\bullet_{\bC[{y}, \ud x]/\bC}, \partial_S)$ is a bigraded cochain complex of bidegree $(0,1)$.
One can also consider the following twisted de Rham complex $(\O^\bullet_{\bC[{y}, \ud x]/\bC}, D_S:=d+\partial_S)$, a so-called algebraic Dwork complex, where $d$ is the usual exterior derivative. For $(u,v) \in \bZ^2$, let us denote by $\O_{\bC[{y}, \ud x]/\bC}^{s,(u,v)}$ the submodule of homogeneous elements of bidegree $(u,v)$ in $\O_{\bC[{y}, \ud x]/\bC}^{s}$.

\begin{lemma} \label{clem}
We have the following relationship between $(\O^\bullet_{\bC[{y}, \ud x]/\bC}, D_S)$ and $(\cA_X, K_X)$:

(a) For each $s \in \bZ$, if we define a $\bC$-linear map $\Phi: (\O^s_{\bC[{y}, \ud x]/\bC}, D_S) \to (\cA_X^{s-(n+2)}, K_X)$ by sending
$dy_{i_1} \cdots d y_{i_s}$ to $(-1)^{i_1 + \cdots +i_s - s} ( \cdots \hat{\eta_{i_1}} \cdots \hat{\eta_{i_s}}\cdots)$ for $-1 \leq i_1 < \cdots < i_s \leq n$ and extending
it $\bC$-linearly, then $\Phi \circ D_S = K_X \circ \Phi$ and $\Phi$ induces an isomorphism
$$
H^s_{D_S}(\O^\bullet_{\bC[{y}, \ud x]/\bC}) \simeq H^{s-(n+2)}_{K_X} (\cA^\bullet_X).
$$
for every $s \in \bZ$.

(b) The map $\Phi$ induces a $\bC$-linear map from $\O_{\bC[{y}, \ud x]/\bC}^{s,(u,v)}$ to $\cA_{X,c,(w)}^{s-(n+2)}$ where
$c=u+c_X$ and $w+s-(n+2)= v-1$. 

(c) The map $\Phi$ satisfies that  $\Phi \circ \partial_S = Q_X \circ \Phi$ and $\Phi \circ d = \Delta \circ \Phi$.
\end{lemma}

\begin{proof}
These follow from straightforward computations.
\end{proof}

\begin{remark}
Lemma \ref{clem} implies that two cochain complexes $(\O^\bullet_{\bC[{y}, \ud x]/\bC}, D_S)$ and $(\cA_X, K_X)$ are degree-twisted isomorphic each other. \footnote{In fact, $(\O^\bullet_{\bC[{y}, \ud x]/\bC}, D_S)$ is the Chevalley-Eilenberg complex associated to the Lie algebra representation $\rho_X$ which computes the Lie algebra cohomology and, on the other hand,
$(\cA_X, K_X)$ is the dual Chevalley-Eilenberg complex which computes the Lie algebra homology.}
But we emphasize that the natural product structure, the wedge product, on $\O^\bullet_{\bC[{y}, \ud x]/\bC}$ and  the super-commutative product $\cdot$ on $\cA_X$
are quite different and $\Phi$ is \textit{not} a ring isomorphism. It is crucial for us to use the super-commutative product $\cdot $ on 
$\cA_X$ to get all the main theorems of the current article.
\end{remark}

\begin{table}[h!] \normalsize
 \begin{center}
  \begin{tabular} 
 {| p{ 6.5cm} |  p{6.5cm}  |  }
\hline
 $(\O^\bullet_{\bC[{y}, \ud x]/\bC}, \wedge, D_S)$ is a Dwork complex &  $(\cA_X, \cdot, Q_X, K_X, \ell_2^{K_X}) $ is a BV algebra \\
\hline
\hline
$(\O^\bullet_{\bC[{y}, \ud x]/\bC}, \wedge, \partial_S)$ is not a CDGA &  $(\cA_X, \cdot, Q_X) $ is a CDGA \\
\hline
 $(\O^\bullet_{\bC[{y}, \ud x]/\bC}, \wedge, d)$ is a CDGA & $(\cA_X, \cdot, \Delta) $ is not a CDGA \\  
\hline
     \end{tabular}
   \end{center}
  \caption{Comparison with the algebraic Dwork complex} \label{cpn}
\end{table}

We proved that $H^{0}_{K_X}(\cA_X)$ is canonically isomorphic to $H^n(\BP^n \setminus X, \bC) \simeq \H$.
Now we describe all the other cohomologies.

\begin{proposition}
Assume that $n > 1$.\footnote{If $n=1$, then the same result holds except $\dim_\bC H^{-1}_{K_X}(\cA_X) =\dim_\bC H^{0}_{K_X}(\cA_X)+1$; see (1.12), \cite{AdSp06}.} Then we have the following description of the total cohomology of $(\cA_X, K_X)$:
\eqn\ccmp{
\eqalign{
H^s_{K_X}(\cA_X) &= 0, \quad \text{for} \quad s\neq -n, -1, 0,  
\cr
H^{-1}_{K_X}(\cA_X)& \simeq H^{0}_{K_X}(\cA_X),
\quad \text{ and } \quad
H^{-n} _{K_X}(\cA_X) \simeq \bC. 
}
}
\end{proposition}

\begin{proof}

This follows from the classical to quantum spectral sequence associated to the filtered cochain complex $(F^\bullet_{\wt}\cA_X, K_X)$ and Theorem 1.6 in \cite{AdSp06}.
We shift the degree of $(\cA_X, K_X)$ to consider $\cC=\cA_X[-n-2]$ so that 
$\cC^i = \cA_X^{i+(n+2)}$ for each $i \in \Z$. 
According to Proposition \ref{wspec} we have
the classical to quantum spectral sequence  $\{ E_r\}$ which satisfies
\eqn \cqcq{
\eqalign{
E_1^{p,q}&\simeq Gr^p H_Q^{p+q}(\cC^{}),
\cr
E_2^{p,q} &\mapto{\simeq} E^{p,q}_{\infty}= Gr^p H^{p+q}_K(\cC).
}
}
In particular, $\{ E_r\}$ degenerates at $E_2$.
Then we have 
$$
H_K^{p}(\cC)=H_{K_X}^{p+n+2}(\cA_X), \quad 
Gr^p H^{p+q}_K(\cC)=Gr^p H^{p+q+n+2}_{K_X}(\cA_X), \quad p,q \geq 0.
$$

Note that Lemma \ref{clem} implies that $H_Q^{p}(\cC) \simeq H^p_{\partial_S}(\O^\bullet_{\bC[{y}, \ud x]/\bC})$ for $p \geq 0$
and Lemma \ref{czero} says that $H^q_{K_X}(\cA_X) =H^q_{K_X}(\cA_{X, c_X})$ for $q \leq 0$.
Moreover it is known that $\dim_\bC(H_{Q_X}^{0}(\cA_{c_X}))$ $=\dim_\bC \H$ (see page 1194, \cite{AdSp06}). 
Now it is easy to see that Theorem 1.16 in \cite{AdSp06} implies the desired result combined with Lemma \ref{czero}, \cqcq, and Lemma \ref{clem}.
In fact, the isomorphism $H^{0}_{K_X}(\cA_X) \simeq H^{-1}_{K_X}(\cA_X)$ is given by $[f] \mapsto [R \cdot f]$
where $f \in \cA^0_{X}$ and $R= \sum_{\m=-1}^{n} \ch(y_\m) y_\m \eta_\m \in \CA^{-1}_{0}$, and $H^{-n}_{K_X}(\cA_X)$ is generated by
$$
[\Phi \left(dy\wedge dG\right)]
$$
where $\Phi$ is the map in Lemma \ref{clem}.
 \end{proof}

\subsection{Lifting of a polarization}\label{subs4.33}
The goal here is to prove part $(d)$ of Theorem \ref{firsttheorem}. The primitive cohomology $\H$ behaves well with respect to the cup product pairing.
Recall that the following bilinear pairing $\langle \cdot, \cdot \rangle$
$$
\langle \omega, \eta \rangle := (-1)^{\frac{(n-1)(n-2)}{2}} \int_{X_G} \omega \wedge \eta, \quad \omega, \eta \in 
\H=H^{n-1}_{\pr}(X_G, \C),
$$
provides \textit{a polarization} of a Hodge structure of weight $n-1$ on $\H$, i.e.

 \begin{enumerate}
 \item $\langle \omega, \eta \rangle = (-1)^{n-1} \langle \eta, \omega \rangle$,
 \item $\langle \omega, \eta \rangle \in \Z$ if $\omega, \eta \in H^{n-1}(X_G, \Z)$,
 \item $\langle \cdot, \cdot \rangle$ vanishes on $H^{p,q} \otimes H^{p',q'}$ unless $p=q', q=p'$,
 \item $(\sqrt{-1})^{p-q} \langle \omega, \overline{\omega} \rangle  > 0$, if $\omega \in H^{p,q}$ is non-zero,
 \end{enumerate}
where $\H =\bigoplus_{p+q=n-1}H^{p,q}$ is the Hodge decomposition and $\overline{\omega}$ is the complex conjugation of $\omega$.

\begin{definition} \label{polgo}
We define a $\C$-linear map
$\oint: \cA_X\rightarrow\C$
such that $\oint$ is a zero map on  $\cA_X^j$
if $j\neq 0$, otherwise:
$$
\eqalign{
\oint {u} 
:= &
\frac{1}{(2\pi i)^{n+2}}
\int_{X(\e)}
\left(\oint_C \frac{{u}}{  \frac{\rd S}{\rd x_0}\cdots\frac{\rd S}{\rd x_n} } y d\!y
\right) dx_0 \wedge \cdots \wedge dx_n
\cr
= &
\frac{1}{(2\pi i)^{n+2}}
\int_{X(\e)}
\left(\oint_C \frac{u}{  y^{n} }dy\right)
\frac{dx_0 \wedge \cdots \wedge dx_n}{  \frac{\rd G}{\rd x_0}\cdots\frac{\rd G}{\rd x_n} }
}
$$
for all ${u} \in \cA_X^0$,
where
$C$ is a closed path on $\C$ with the standard orientation around $y=0$ 
and
$$
X(\e)=\left\{ \underline{x} \in \C^{n+1}\left| \Big|\frac{\rd G(\ud x)}{\rd x_i}\Big|=\e > 0, i=0,1,\cdots, n\right.\right\},
$$
 which is oriented by $d(\arg \frac{\rd G}{\rd x_0})\wedge \cdots \wedge d(\arg \frac{\rd G}{\rd x_n}) >0$.
\end{definition}
Our definition of $\oint$ is motivated by the Grothendieck residue. 
We construct a lifting of $\langle \cdot, \cdot \rangle$ to $\cA_X$ by using $\oint$; see \liftingofp.

\begin{theorem}\label{polgood}

(a) The $\C$-linear map $\oint:\cA_X \to \C$ is concentrated in weight $n-1$, charge $2c_X=2d-2(n+1)$, and ghost number $0$, i.e. 
$\oint$ factors through the projection map from $\cA_X$ to $\cA^0_{X, (n-1), 2c_X}$.\footnote{In fact, one can show that $\cA^0_{X, (n-1), 2c_X}/Q_X(\cA^{-1}_{X, (n-1), 2c_X})$ is isomorphic to $\C$.}

(b) We have 
$$
\oint  Q_X(u) =0, \quad \oint Q_X(u)\cdot v =\oint (-1)^{|u|+1}u \cdot Q_X(v), \quad \forall u, v\in \cA_X.
$$

(c) Under the map $J:(\cA_{X, c_X},Q_X) \to (\H,0)$, we have
\eqn\bosides{
 \frac{c_{ab}}{\l} \oint u \cdot v = \int_{X_G} J(u) \wedge J(v)=:(-1)^{\frac{(n-1)(n-2)}{2}} \langle J(u), J(v) \rangle,
 \quad u\in \cA_{X,(a),c_X}, v \in \cA_{X,(b),c_X},
 }
where $c_{ab}= d\cdot (-1)^{\frac{a(a+1)}{2} +\frac{b(b+1)}{2}+b^2}$ and $\l$ is the residue 
of the fundamental class of $\BP^n$, viewed in $H^n(\BP^n, \Omega^n)$.
\end{theorem}

\begin{proof}
We use the notation $(\cA, K=Q+\Delta)= (\cA_X, K_X=Q_X +\Delta)$ for simplicity.

(a)
Note that $\oint  u=0$ for $u$ homogeneous (with respect to the weight) of some
weight other than $n-1$, since 
$\oint_C \Fr{1}{ y^{m} }dy =0$ unless $m=1$. 
When we write $u \in \cA^0_{\l}$ as $u = \sum_{k\geq 0} y^k\cdot u_k(\ud{x})$,
where $u_k \in \C[\ud x]_{\l + kd}$, a simple computation confirms that
$$
\oint u =\hbox{Res}_0 \Bigg\{{{u_{n-1}}\atop {\Fr{\rd G}{\rd x_0}\cdots\Fr{\rd G}{\rd x_n}} }\Bigg\}
$$,
where the right hand side is the Grothendieck residue given in (12.3), \cite{PS}.
Then Lemma (12.4), \cite{PS} implies that if $\l \neq 2 c_X$ then $\oint u=0$.
Hence $\oint$ is concentrated in charge $2 c_X$. Therefore we get the result.

(b) It suffices to consider the case when
the integrand $Q(\Lambda)$ is an element of
$\cA^0_{(n-1),2c_X}$, when we prove that $\oint \circ {Q}=0$.
Thus it is enough to check that $\oint {Q} (\La)=0$ for ${\La} \in
\cA^{-1}_{(n-1),2c_X}$.
An arbitrary element ${\La} \in \cA^{-1}_{(n-1),2c_X}$ 
can be written as a $\C$-linear
combination of terms like
$$
{\La} =y^{n-1}{M}(\underline{x})\eta_{-1} +\sum_{i=0}^n y^{n-2} {N}_i(\underline{x})\eta_i
$$
where $M$ and $N_i$, $i=0,1,\cdots, n$ are monomials in $\underline{x}$ such that
$\ch({M})=\deg(M)= 2c_X+d(n-2)$ and $\ch({N}_i)=\deg(N_i)=2c_X +d(n-2)+1$.
Then
$$
\eqalign{
Q{(\La)} =&
y^{n-1}{M}\cdot G 
+y^{n-1}\sum_{i=0}^n {N}_i\Fr{\rd G}{\rd x_i}.
}
$$
Note that $\hat E_{ch}(G) = {\ch}(G) \cdot G$ and $\ch(G)=d\neq 0$, and so we get
$$
G=\Fr{1}{\ch(G)} \sum_{i=0}^n x_i\Fr{\rd G}{\rd x_i}.
$$
Hence
$$
\eqalign{
Q{\La} =
y^{n-1} \sum_{i=0}^n \Fr{1}{\ch(G)}{M}\cdot x_i\Fr{\rd G}{\rd x_i}
+y^{n-1}\sum_{i=0}^n {N}_i\Fr{\rd G}{\rd x_i}
=y^{n-1}\sum_{i=0}^n \tilde{N}_i\Fr{\rd G}{\rd x_i}.
}
$$
where
\eqn\tiddef{
\tilde N_i(\underline{x})=
 \Fr{1}{\ch(G)}{M}\cdot x_i+{N}_i.
}
It follows that
$$
\eqalign{
 \oint Q\La 
= & \oint y^{n-1}\sum_{i=0}^n \tilde{N}_i(\ud x)\Fr{\rd G}{\rd x_i}
\cr
=&\Fr{1}{(2\pi i)^{n+1}}
\int_{X(\e)}
\left({\sum_{i=0}^n \tilde{N}_i(\ud x)\Fr{\rd G}{\rd x_i}}
\right)\Fr{dx_0 \wedge \cdots \wedge dx_n}{  \Fr{\rd G}{\rd x_{0}}\cdots\Fr{\rd G}{\rd x_{n}} }
\cr
=&  \hbox{Res}_0 \Bigg\{ {{{\sum_{i=0}^n \tilde{N}_i(\ud x){\Fr{\rd G}{\rd x_i}}}}
 \atop  {\Fr{\rd G}{\rd x_0}\cdots\Fr{\rd G}{\rd x_n}}} \Bigg\}=0,
}
$$
where we used again Lemma 12.4, \cite{PS}, in the final 
step (${\sum_{i=0}^n \tilde{N}_i(\ud x)\Fr{\rd G}{\rd x_i}}$ belongs to the Jacobian ideal). 
Hence $\oint$ is a cochain map from $(\cA, Q)$ to $(\C, 0)$.

The second equality
$$
\oint \left( Q\a\cdot \b +(-1)^{|\a|}\a\cdot Q\b \right)=0,
$$
also follows easily from the same computation above.

(c) Note that
$$
J( [y^{k} F(\ud x)])=Res\left( \big[ (-1)^k k! \frac{F(\ud x) \Omega_n}{ G^{k+1}}\big] \right).
$$
Let $u=y^{a} A(\ud x)$ and $v=y^{b} B(\ud x)$ where $a+b=n-1$ (if $a+b \neq n-1$, then both sides
in \bosides\ are zero). Then Remarks on page 19 of \cite{CG} imply that
$$
\eqalign{
\int_{X_G} J(u) \wedge J(v) 
&=
(-1)^{(n-1)}a! b! \int_{X_G} Res \big[  \frac{A(\ud x) \Omega_n}{ G^{a+1}}  \big] \wedge
Res \big[ \frac{B(\ud x) \Omega_n}{ G^{b+1}}\big]
\cr
&=\frac{c_{ab}}{\l} \hbox{Res}_0 \Bigg\{ {{A(\ud x) B(\ud x)} \atop {\Fr{\rd G}{\rd x_0}\cdots\Fr{\rd G}{\rd x_n}} }\Bigg\}
\cr
&=\frac{c_{ab}}{\l}  \oint y^{n-1} A(\ud x) B(\ud x)=\frac{c_{ab}}{\l}  \oint u\cdot v.
}
$$
This finishes the proof.
\end{proof}

Note that $\oint \circ K_X\neq 0$. For example, if we assume that $d=n-1$, then there exists $\mu \in \cA^{-1}_{(n-1),0}$
such that 
$$
\Delta(\mu) = y^{n-1} \det \left( \frac{\partial^2{G(\ud x)}}{\partial x_i \partial x_j} \right) \in \cA^{0}_{(n-1),0},
$$
(since $\Delta:\cA^{-1} \to \cA^{0}$ is surjective) and $\int \Delta(\mu) \neq 0$ by (12.5) in \cite{PS}.

\subsection{The $L_\infty$-homotopy structure on period integrals of smooth projective hypersurfaces} \label{subs4.4}

The goal here is to prove Theorem \ref{thirdtheorem}, which describes how to understand the 
period integrals of smooth projective hypersurfaces in terms of homotopy theory, by applying
the general theory developed in sections \ref{section2} and \ref{section3}.
Let $\{[r_{\alpha}]\}_{\a \in J}$ be a $\bC$-basis of $\bigoplus_{m < 0} H^m_{K_X}(\cA_X)$. Note that $J$ is a finite index set.
For each $\a \in J$, we introduce new indeterminates $\theta_{\alpha}$ corresponding to $r_{\alpha}$ such that $\gh(\theta_{\alpha})=|\theta_{\alpha}| = |r_{\alpha}|-1$
in order to find a resolution of $(\cA_X, \cdot, K_X)$ in the category $\frC_\bC$.
We define a super-commutative algebra
\eqn\resl{
\tilde \cA_X = \cA_X[\theta_{\alpha} \ :  \ \a \in J]/ \cI,
}
where $\cA_X[\theta_{\alpha} \ :  \ \a \in J]$ is the super-commutative algebra generated by $\theta_{\alpha}$ over $\cA_X$ and
$\cI$ is the ideal generated by $\theta_{\alpha}^2$ and $\theta_{\alpha} \cdot  f$ for $f \in \cA_X \setminus \bC$.
If we define $\tilde K(\theta_{\alpha}) = r_{\alpha}$ for each $\a \in J$ and $\tilde K(f) = K(f)$ for $f \in \cA_X$, then
this defines a $\bC$-linear map $\tilde K: \tilde \cA_X^\bullet \to \tilde \cA_X^\bullet$ of degree 1.
\begin{lemma}
The triple $(\tilde \cA_X, \cdot, \tilde K_X)$ is an object of the category $\frC_\bC$. We have $H_{\tilde K_X}^p (\tilde \cA_X)=0$
for every $p \neq 0$ and $H_{\tilde K_X}^0(\tilde \cA_X)=H_{K_X}^0 (\cA_X)$. In other words, $(\tilde \cA_X, \cdot, \tilde K_X)$
is a resolution of $(H_{K_X}^0 (\cA_X), 0)$ inside $\frC_\bC$.
\end{lemma}

\begin{proof}
This is clear from the construction.
\end{proof}

We remark that $(\tilde \cA_X, \cdot, \tilde K_X)$ does not give a BV algebra: since one can easily check that $\ell_{m}^{\tilde K_X} \neq 0$ for every $m \geq 1$, the differential $\tilde K_X$ can not be decomposed as $\tilde Q+ \tilde \Delta$ satisfying the axioms of a BV algebra.\footnote{Jeehoon Park and Donggeon Yhee proved that one has to add infinitely many new formal variables to make a \textit{BV resolution} of $(\cA_X, \cdot, K_X)$. See \cite{PY} for a precise meaning of such a BV resolution.}
We apply the descendant functor (Theorem \ref{Desc}) to $(\tilde \cA_X, \cdot, \tilde K_X)$ and put $\left(\tilde \cA, \tilde{\ud \ell}\right)_X=\Des (\tilde \cA_X, \cdot, \tilde K_X)$. The above lemma, together with Proposition \ref{mainhodge}, implies the following proposition.

\bep \label{ldgla}
The descendant $L_\infty$-algebra $\left(\tilde \cA, \tilde{\ud \ell}\right)_X$ is
quasi-isomorphic to $(\H,0)$, where $\H = H^{n-1}_{\pr}(X_G, \bC).$ 
\eep

The strategy we have developed so far suggests that we have to understand $C_{[\g]}$ using the composition of $L_\infty$-morphisms; Proposition \ref{enhancetochain}, Proposition \ref{comLi}, Theorem \ref{homotopyperiod}, Proposition \ref{main}, and Proposition \ref{mainhodge} imply the following theorem (a restatement of Theorem \ref{thirdtheorem}).
\bet \label{undper}
For each $\g\in H_{n-1}(X_G, \bC)_0$, 
the period integral $C_{[\g]}$ can be enhanced to the composition of $L_\infty$-morphisms through $(\tilde \cA, \tilde{\ud \ell})_X$: we have the following diagram of $L_{\infty}$-morphisms of $L_{\infty}$-algebras:
\eqn\dgm{
\xymatrix{
(\H, \ud 0)  \ar@/^0.0pc/[dd]^{\varphi_1^\H} \ar[rd]_{C_{[\g]}} \ar@{.>}@/^/[drr]^{\ud \kappa:= \ud \phi^{\cC_\g} \bullet \ud \varphi^\H}&&
\\
&(\bC, \cdot, 0) \ar@{=>}[r] &( \bC, \ud 0)
\\
(\tilde \cA_X, \cdot, \tilde K_X)
\ar[ur]_{\cC_\g} &&
\\
(\tilde \cA, \tilde{\ud \ell})_X \ar@{<.}@/^3pc/[uuu]^{\ud \varphi^\H} \ar@{<=}[u] \ar@{.>}@/_/[uurr]_{\ud \phi^{\cC_\g}}&&
}
}
where $C_{[\g]}$ is the same as $(\ud \phi^{\cC_\g} \bullet \ud \varphi^\H)_1 = \cC_\g \circ \varphi_1^\H$. 
The $L_\infty$-morphism $\ud \kappa:= \ud \phi^{\cC_\g} \bullet \ud \varphi^\H$ depends only on the $L_\infty$-homotopy 
types of $\ud \varphi^\H$ and $\ud \phi^{\cC_\g}$. Here the notation $\implies$ means that we take the descendant functor.

\eet

%

The composition $\ud \phi^{\cC_\g} \bullet \ud \varphi^\H$ is not a descendant $L_\infty$-morphism: we do not have a super-commutative associative binary product on $\H$ (the differential $K$ is not a derivation of the product of $\cA$ so it does not
induce a product on $\H$). Note that $\ud \varphi^\H$ is an $L_\infty$-quasi-isomorphism; such quasi-isomorphisms are
classified by the versal solutions $\G \in (\mm_{\widehat{S\H}}\otimes \tilde \cA)^0=\mm_{\widehat{S\H}} \otimes \tilde \cA^0$ to the Maurer-Cartan equation (see Proposition \ref{defmoduli} and Definition \ref{doit}):
\be
\tilde K(\G) + \frac{1}{2} \ell^{\tilde K}_2(\G, \G)+\frac{1}{3!}\ell^{\tilde K}_3(\G,\G)+ \cdots = 0.
\ee
Note that $ (\mm_{\widehat{S\H}}\otimes \tilde \cA)^0= (\mm_{\widehat{S\H}}\otimes \cA)^0$. So we found a hidden structure of the period integral $C_{[\g]}:\H\simeq \cH(X_G) \to \bC$ for a fixed $\g \in H_{n-1}(X_G, \bZ)_0$: there is an $L_\infty$-quasi-isomorphism from the 
$L_\infty$-algebra $(\H, \ud 0)$ to the $L_\infty$-algebra $(\tilde \cA, \tilde{\ud \ell})_X$, and a sequence of $\bC$-linear 
maps $(\ud \phi^{\cC_\g} \bullet \ud \varphi^\H)_m: S^m(\H) \to \bC$, which reveals hidden correlations and deformations of $C_{[\g]}$, such that $C_{[\g]}=(\ud \phi^{\cC_\g} \bullet \ud \varphi^\H)_1$. 
Then the theory of $L_\infty$-algebras suggests that we can study a (new type of) formal variation of the Griffiths period integral. We discuss this issue in the next subsections.

%

\subsection{Extended formal deformations of period integrals of $X_G$} \label{subs4.6}

We prove Theorem \ref{fourththeorem} here.
Proposition \ref{defmoduli} implies that the extended deformation functor
attached to $(\tilde \cA, \tilde{\ud \ell})_X$ is {\it pro-representable} by the completed symmetric algebra
$\widehat{S\H}$. 
Let $\cM_{X_G}$ be the associated formal moduli space.
Now we consider the generating series $\cZ_{[\g]}\left(\big[\ud \varphi^\H\big]\right)$ in \mgfcy. 
Note that
$L_\infty$-homotopy types of quasi-isomorphisms $\xymatrix{\ud \varphi^\H:(\H,\ud 0)\ar@{..>}[r]& (\cA, \ud \ell)_X}$ are not unique, though  the $L_\infty$-homotopy type of  $\ud \phi^{\cC_\g}$ is uniquely determined by 
$[\g] \in H_{n-1}(X_G, \bZ)_0$.
The results in Section \ref{section2} specialized to $(\tilde \cA, \tilde{\ud \ell})_X$ gives us the following theorem (see Proposition \ref{defmoduli}, Definition \ref{gpseries}, and Lemma \ref{invhom}).

\begin{theorem}\label{tmthm}
Let $\{e_{\alpha}\}_{\a \in I}$ be a basis of $\H$ with dual basis $\{t^\a\}_{\a\in I}$. Then
for any $L_\infty$-quasi-isomorphism $\ud \varphi^\H$ from $(\H, \ud 0)$ to $(\tilde \cA, \tilde {\ud \ell})_X$,
we have the following versal solution
to the Maurer-Cartan equation of $(\tilde \CA,\tilde {\ud \ell})_X$:
$$
\G(\ud t)_{\ud \varphi^\H} = \sum_{\a \in I} t^\a \varphi^\H_1(e_{\alpha}) + 
 \sum_{k=2}^\infty\sum_{\a_1,\cdots, \a_k \in I} t^{\a_k}\cdots t^{\a_1}\otimes  \varphi^\H_k\big(e_{\a_1},\cdots,e_{\a_k}\big)
 \in (\C[[\ud t]]\hat\otimes \CA)^0
$$
such that 
$\G(\ud t)_{\ud \varphi^\H}$ is gauge equivalent to 
$\G(\ud t)_{\ud {\tilde\varphi}^\H}$ if and only if
 $\ud \varphi^\H$ is $L_\infty$-homotopic to
 ${\ud {\tilde\varphi}^\H}$, and 
 $$
\cZ_{[\g]}\left(\big[\ud \varphi^\H\big]\right)=\cC_\g\left(e^{\G(\ud t)_{\ud \varphi^\H }}-1\right).
 $$
\end{theorem}
Note that the dual basis $\{t^\a\}_{\a\in I}$ is an affine coordinate on $\H$.
Let $X_{G_{\ud {T}}} \subset \BP^n$ be a formal family of smooth hypersurfaces defined by
\eqn\formald{
G_{\ud{T}} (\ud x) = G(\ud x) +F(\ud T),
}
where $F(\ud T) \in \bC [[\ud T]] [\ud x]$ is a homogeneous polynomial of degree $d$ with 
coefficients in $\bC[[\ud T]]$ with $F(\ud 0) =0$ and $\ud T=\{T^\a \}_{\a \in I'}$ are formal variables with some index set $I' \subset I$. 
 Recall that by a standard basis of $\H$ we mean a choice of basis $e_1,\cdots, e_{\d_0},$ $ e_{\d_0+1},\cdots,
 $ $ e_{\d_1},\cdots,e_{\d_{n-2}+1},\cdots,
e_{\d_{n-1}}$ for the flag $\cF_\bullet \H$ in \pog\ such that
$e_1,\cdots, e_{\d_0}$ gives a basis for the subspace $\H^{n-1 - 0, 0}:=H_{\pr}^{n-1,0}(X_G,\C)$
and $e_{\d_{k-1} +1},\cdots, e_{\d_{k}}$, $1\leq k\leq n-1$, gives a basis for the subspace 
$\H^{n-1 - k, k}=H_{\pr}^{n-1-k,k}(X_G,\C)$. We denote such a basis by 
$\{e_{\alpha}\}_{\a \in I}$ where $I= I_{0}\sqcup I_{1}\sqcup\cdots\sqcup I_{n-1}$ with the notation
$\{e^{j}_{a}\}_{a\in I_j}= e_{\d_{j-1} +1},\cdots, e_{\d_{j}}$ 
and $\{t^a_{j}\}_{a \in I_{j}}= t^{\d_{j-1} +1},\cdots, t^{\d_{j}}$.
We need to assume that $I' \subset I_1$ when $X_G$ is Calabi-Yau.

For $(a)$ of Theorem \ref{fourththeorem} (recall that $X_G$ is assumed to be Calabi-Yau), we define $\bC$-linear maps $\ud f: (\H,\ud 0) \to (\tilde \cA, \tilde {\ud \ell})_X$:
\eqn\hlq{
\ud f = f_1, f_2, f_3\cdots,  \quad f_1(e_a^{k}) := y^k \cdot F_{[k] a}(\ud x) \in \tilde \cA^0_0,
}
where $F_{[k] a}(\ud x)$ can be chosen to be any homogeneous polynomial of degree $d(k+1)-(n+1)=dk$ 
such that
$  \{{F_{[k]a}(\ud x)\over G(\ud x)^{k+1}}\Omega_n  \ : \ a \in I, \ 0 \leq k \leq n-1 \}$ is a set of 
representatives of a basis of $\cH(X_G) \simeq \H$ and $F_{[1]a}(\ud x)$ is the $t^a$-coefficient 
of $F(\ud T)$ with $t^a=T^a$; we define $f_m: S^m(\H) \to \tilde \cA^0, m \geq 2$ so that $f_m(e_{a_1}^{1}, \cdots, e_{a_m}^{1})$ is the $t^{a_m}\cdots t^{a_1}$-coefficient of $yF(\ud T)$ with $t^a=T^a, a \in I'$.
Then $\ud f$ is clearly a $L_\infty$-quasi-isomorphism satisfying $(a)$ of Theorem \ref{fourththeorem}
by the Griffiths theorem on $\cH(X_G)$ and our general theory: the fact that $f_1$ is a quasi-isomorphism comes
from the construction; the fact that it is a morphism of $L_\infty$-algebras is completely trivial because $\H$ is concentrated in degree 0 and $\tilde \cA$ is concentrated in non-negative degree, so all relations that need to be checked are easily seen to be zero.
Then $(b)$ of Theorem \ref{fourththeorem} follows from the following computation;
$$
\eqalign{
& \frac{1}{2 \pi i} \left( {\int_{\tau(\g)} {\Omega_n\over G_{\ud{T}}(\ud x) }  - 
\int_{\tau(\g)} \frac{\Omega_n}{G(\ud x)}   }   \right)
\cr
&= -\frac{1}{2 \pi i} \left( \int_{\t(\g)} \int_0^\infty e^{y \cdot G_{\ud{T}}(\ud x)} dy \Omega_n 
- \int_{\t(\g)} \int_0^\infty e^{y \cdot G(\ud x)} dy \Omega_n  \right)
\cr
&= \cC_\g (e^{\sum_{a\in I'} T^a y \cdot F_{[1]a}(\ud x)} - 1)
\cr
&= \exp \left({\Phi^{\cC_\g} \left(\sum_{a\in I'} T^{a}\cdot  f_1(e_a^{[1]})  \right)  } \right) -1
\cr
&=\cZ_{[\g]}([\ud f])(\ud t)\Big|_{\substack{t^\b=0, \b\in I\setminus I' \\  t^\a=T^\a, \a \in I'}}\cr
}
$$
where we used the definition of $\cC_\g$ and the equalities \kexp\ and \hper.
Let
\eqn\efd{
\G(\ud t)_{f_1}=
\sum_{a\in I_0} t^{a}_0F_{[0]a}(\ud x)
+ y\cdot\!\!\sum_{a\in I_1} t_{1}^{a}F_{[1]a}(\ud x) 
+\cdots
+ y^{n-1}\cdot\!\!\!\sum_{a\in I_{n-1}} t^{a}_{n-1}F_{[n-1]a}(\ud x).
}
We can do a similar computation (a period integral of an extended formal deformation of $X_G$) ;
$$
\eqalign{
&\cC_\g(e^{\G(\ud t)_{f_1}}-1)
\cr
&=-\frac{1}{2 \pi i}\int_{\tau(\g)}\!\! \left(\int_0^\infty
e^{\G(\ud t)_{f_1}}\cdot e^{y\cdot G(\ud x)}
d\!y\right)\Omega_n    -\frac{1}{2 \pi i} \int_{\tau(\g)} \frac{\Omega_n}{G(\ud x)}
\cr
&= \exp \left({\Phi^{\cC_\g} \left(\sum_{a\in I_0} t^{a}_0F_{[0]a}(\ud x)
+ y\cdot\!\!\sum_{a\in I_1} t_{1}^{a}F_{[1]a}(\ud x) 
+\cdots
+ y^{n-1}\cdot\!\!\!\sum_{a\in I_{n-1}} t^{a}_{n-1}F_{[n-1]a}(\ud x) \right)  } \right) -1
\cr
&=
\exp\left(\sum_{n=1}^\infty \frac{1}{n!} \sum_{\alpha_1, \cdots, \alpha_n} t^{\alpha_n} \cdots t^{\alpha_1} 
\left(\ud\phi^{\!\!\cC_\g}\bullet \ud{f}\right)_n(e_{\alpha_1}, \cdots, e_{\alpha_n} )\right) -1
=\cZ_{[\g]}([\ud f]).
}
$$

%
The properties, which can be easily checked, 
$$
\eqalign{
\int_0^\infty
e^{\G(\ud t)_{f_1}}\cdot e^{y\cdot G(\ud x)}
d\!y\Bigg|_{\ud t =\ud 0} \Omega_n
&=\frac{-\Omega_n}{G(\ud x)}
,\cr
\frac{\rd}{\rd t^{a}_{k}}\int_0^\infty
e^{\G(\ud t)_{f_1}}\cdot e^{y\cdot G(\ud x)}
d\!y \Bigg|_{\ud t =\ud 0} \Omega_n
&= (-1)^{k+1} k! \frac{F_{[k]a}\cdot\Omega_n}{G(\ud x)^{k+1}},\qquad \forall a \in I_k,
}
$$
for example, imply that Griffiths transversality is violated for $k\geq 2$. 
Note that the deformation in part $(b)$ is 
a geometric deformation of the complex structure of $X_G$, a family of hypersurfaces.
Though this extended deformation
does not have a clear geometric meaning yet, the above properties 
are the key components to prove Theorem \ref{fifththeorem}, which 
demonstrates the usefulness of $L_\infty$-homotopy theory to compute
an extended period matrix 
$\frac{\rd}{\rd t^\b} \left( \cZ_{[\g_{\alpha}]}\left(\big[\ud f\big]\right)(\ud t)\right)_{\{\a, \b \in I \}}$.
Also, the generating power series $\cZ_{[\g]}([\ud f])(\ud t)$ is a natural generalization of 
the geometric invariant. 

\subsection{Extended formal deformations of the period matrix of $X_G$} \label{subs4.61}

We prove Theorem \ref{fifththeorem}, which demonstrates
usefulness of the $L_\infty$-homotopy theory
to compute the period matrices of a deformed hypersurface and an extended formal deformation. 
Since Lemma \ref{onetensor} directly implies $(a)$ and $(b)$ of Theorem \ref{fifththeorem}, we
concentrate on proving the second part of Theorem \ref{fifththeorem}, i.e.
$$
\eqalign{
\omega^\a_\b (X_{G_{\ud {T}}}) 
&= 
\frac{\rd}{\rd t^\b} \left( \cZ_{[\g_{\alpha}]}\left(\big[\ud f\big]\right)(\ud t)\right) 
\Big|_{\substack{t^\b=0, \b\in I\setminus I' \\  t^\a=T^\a, \a \in I'}}
\cr
&=\sum_{\rho\in I} \left( \frac{\rd}{\rd t^\b} T^{\rho}(\ud t)_{\ud f}\right) \omega^\a_\rho(X_G)
\Big|_{\substack{t^\b=0, \b\in I\setminus I' \\  t^\a=T^\a, \a \in I'}},
}
$$
for each $\a, \b \in I$.
Since we have
$$
\eqalign{
2 \pi i \cdot \cZ_{[\g_{\alpha}]}([\ud f])
&=-\int_{\tau(\g_{\alpha})}\!\! \left(\int_0^\infty
e^{\G(\ud t)_{\ud f}}\cdot e^{y\cdot G(\ud x)}
d\!y\right)\Omega_n    - \int_{\tau(\g)} \frac{\Omega_n}{G(\ud x)},
}
$$
we have, for each $\a, \b \in I_k, 0 \leq k \leq n-1$,
$$
\eqalign{
\frac{\rd}{\rd t^\b} \cZ_{[\g_{\alpha}]}([\ud f]) (\ud t)\Big|_{\substack{t^\b=0, \b\in I\setminus I' \\  t^\a=T^\a, \a \in I'}}
&=
\frac{1}{2 \pi i} \int_{\tau(\g_{\alpha})}   \frac{ (-1)^k k! F_{[k]\b}(\ud x)}
{\left(G(\ud x)+ \sum_{a\in I'} T^a F_{[1]a}(\ud x)\right)^{k+1}}\Omega_n
\cr
&=
\omega^{\alpha}_{\beta}(X_{G_{\ud {T}}}).
}
$$
This finishes the proof.\\

\subsection{The Gauss-Manin connection and extended formal deformation space} \label{subs4.62}

Let $H_K= H_K(\cA)$ be the cohomology group of $(\cA_X, K_X)$. This subsection we assume that $X$ is Calabi-Yau.
Let $\ud \varphi^H$ be an $L_\infty$-quasi-isomorphism from $(H_K,\ud 0)$ to $(\tilde \cA, \tilde {\ud \ell})_X$ by
Proposition \ref{ldgla}. Then we see
that the system of second order partial differential equations \hodi\ holds for a uniquely determined 3-tensor $A_{\a\b}{}^\g(\ud t)_\G \in \bC[[\ud t]] \simeq \widehat{SH}$ for $\G=\G(\ud t)_{\ud \varphi^H}$, where $\widehat{SH}:=\varprojlim_n $ $\bigoplus_{k=0}^nS^k(H_K^*)$ with $H_K^* = \Hom(H_K, \bC)$; the assumption in Theorem
\ref{diff} can be checked for $(\cA, \cdot, K)$.

In subsection \ref{subs3.7}, we reinterpreted the 3-tensor $A_{\a\b}{}^\g(\ud t)_\G$ as a flat (integrable) connection
$D_\G$ on $\TM_{X_G}=\cM_{X_G} \times V$, where $V$ is isomorphic to $H_K^*$. Here we provide an explicit
relationship between $D_\G$ and the Gauss-Manin connection on a geometric formal deformation given by
an element in $\H^{n-2,1}$. \\ 

Recall the notation in \hlq\ and \efd.
The Gauss-Manin connection is defined on a formal (geometric) deformation space $\cX$ over the formal spectrum
of $\bC[[\ud t_1^a : a \in I_1]]$ such that the fibre $\cX_0$ is isomorphic to $X_G$ and the fibre $\cX_p$ at a general $p$ is isomorphic to a smooth projective hypersurface
of degree $d$ in $\BP^n$. Let $\pi$ be a morphism from $\cX$ to the formal spectrum 
$Spf\left(\bC[[\ud t_1^a : a \in I_1]]\right)$.
The algebraic de Rham primitive cohomology (locally free) sheaf 
$\cH^{n-1}_{dR,\pr}(\cX/ \bC[[\ud t_1^a : a \in I_1]])$ on $Spf\left(\bC[[\ud t_1^a : a \in I_1]]\right)$
corresponds to our flat connection $D_\G$ on the tangent bundle $\TM_{X_G}$ restricted from  
$Spf\left(\bC[[\ud t]]\right)$ to $Spf\left(\bC[[\ud t_1^a : a \in I_1]]\right)$.
Note that the stalk $\cH^{n-1}_{dR,\pr}(\cX/ \bC[[\ud t_1^a : a \in I_1]])_p$ at $p$ such that 
$\cX_p$ is smooth,
is isomorphic to the $\H=H^{n-1}_{\pr}(X_G, \bC)$.

Then the following matrix of 1-forms with coefficients in power series in $\ud T$
$$
(A_{\G_{\ud f}})_\b{}^{\g}:= -\sum_{\a\in I_1} d\! t_1^\a \cdot A_{\a\b}{}^{\g}(\ud t)_{\G_{\ud f}}
\Big|_{\substack{t^\b=0, \b\in I\setminus I' \\  t^\a=T^\a, \a \in I'}},  \quad \b, \g \in I,
$$
becomes the connection matrix of the Gauss-Manin connection along the geometric deformation 
given by the $\H^{n-2,1}$-component of 
the $L_\infty$-homotopy type of $\ud f$ in Theorem \ref{fourththeorem}.
Then Theorem \ref{diff} implies the following proposition;

\begin{proposition} \label{gaussmanin}
We have 
$$
\partial_\g \omega^\a_\b (X_{G_{\ud {T}}})-  
\sum_{\rho \in I}A_{\g\b}{}^{\rho}(\ud t)_{\G_{\ud f}}
\cdot \omega^\a_\rho (X_{G_{\ud {T}}})=0, \quad  \g \in I_1, \a, \b \in I,
$$
where $\omega^\a_\b (X_{G_{\ud {T}}})$ is the period matrix of a deformed hypersurface $X_{G_{\ud {T}}}$.
\end{proposition}

\subsection{Explicit computation of deformations of the Griffiths period integrals} \label{subs4.5}

We use the same notation as subsection \ref{subs4.62}. Here we provide
an algorithm to compute $A_{\a\b}{}^\g(\ud t)_\G$, generalizing the method (a systematic way
of doing integration by parts in the one-variable case) in Proposition \ref{Etoy}.
In subsection \ref{subs3.6}, we saw that the explicit computation problem of the generating power series $\cC(e^\G-1)$ reduces to the problem of 
computing $A_{\a\b}{}^\g(\ud t)_\G$, in addition to the data $\cC(\varphi_1^H(e_{\alpha})), \a \in I$. 
Because of the relationship \onediff\, we can also determine the matrix 
$\CG_{\alpha}{}^\b(\ud t)_{\ud \varphi^H}  
:=\rd_{\alpha} T^\b(\ud t)_{\ud \varphi^H}$ from $A_{\a\b}{}^\g(\ud t)_\G$.
We recall that
$\cA = \cA^{-n-2} \oplus \cdots \oplus \cA^{-1}\oplus \cA^0 $, $K=\Delta + Q$, and
$$
K_\G =  \Delta + Q_\G,\qquad\left\{
\eqalign{
\Delta &=\sum_{i=-1}^n  \pa{y_i} \pa{\eta_i},\cr
Q_\G &= \sum_{i=-1}^n \prt{(yG(\ud x))}{y_i} \pa{\eta_i}+\ell^K_2 (\G, \cdot)=Q+\ell^K_2 (\G, \cdot).
}\right.
$$
If $\Lambda = \sum_{i=-1}^n \l_i \eta_i \in \cA^{-1}$, where $\l_i \in \cA^0, i =-1, 0, \cdots, n$, then
\be
Q_\G(\Lambda) = \sum_{i=-1}^n \l_i \prt{(y G(\ud x))}{y_i}+\ell^K_2(\G, \Lambda), \quad 
\Delta(\Lambda) = \sum_{i=-1}^n \prt{\l_i}{y_i}.
\ee
Note that
\be
K_\G^2= \Delta^2 = 0, \quad K(1_\cA)=Q(1_\cA)=\Delta(1_\cA)=0,
\ee
where $1_\cA$ is the identity element in $\cA^0=\bC[y, \ud x]$.
%

For any $L_\infty$-quasi-isomorphism $\ud \varphi^H$ from $(H_{K}(\cA), 0)$
to $(\cA, K)$, we clearly have that $\Delta (\G(\ud t)_{\ud \varphi^H}) =0$. Thus
we have
\be
Q(\G) + \frac{1}{2} \ell^K_2(\G, \G)=0,
\ee
because $K(\G) + \frac{1}{2} \ell^K_2(\G, \G)=Q(\G) +\Delta(\G)+ \frac{1}{2} \ell^K_2(\G, \G)=0$, which is equivalent to $K(e^\G-1)=0$. Then this says that
\be
\Delta Q_\G + Q_\G \Delta = Q_\G^2 =0.
\ee
So we have a cochain complex $({\widehat{SH}} \otimes \cA, \cdot, Q_\G)$ with super-commutative product. 
Note that $Q_\G$ is a derivation of the binary product of ${\widehat{SH}} \otimes \cA$. We tacitly think of $Q_\G$ as \textit{the classical component} of the differential $K_\G=Q_\G +\Delta$; we view $\Delta$ as \textit{the quantum component} of $K_\G$ on
the other hand. The key point in the algorithm of computing $A_{\a\b}{}^\g(\ud t)_\G$ is that we can
use the ideal membership problem based on Gr\"obner basis methods to compute the answer
to a problem that involves only $Q_\G$
and relate it the corresponding problem on $K_\G$, by utilizing the quantum component $\Delta$.

\bel \label{memcon}
Let $\G=\G(\ud t)_{\ud \varphi^H}$ be a versal Maurer-Cartan solution. 
Then there is an algorithm to compute
a finite sequence $ A_{\a\b}^{(m) \g}(\ud t) \in ({\widehat{SH}}\otimes \cA)^0$, where $m =0,1, \cdots, M$ for some positive integer $M$, such that
\be
A_{\a\b}{}^\g (\ud t)_\G= A_{\a\b}^{(0) \g}(\ud t)-A_{\a\b}^{(1) \g}(\ud t)-\cdots -A_{\a\b}^{(M) \g} (\ud t)\in \widehat{SH}.
\ee
where $A_{\a\b}{}^\g (\ud t)_\G $ is in \hodi.
\eel
\begin{proof}
We use the notation $ \G_{\a\b}=\partial_{\a\b} \G(\ud t)$ and $\G_{\alpha}= \partial_{\alpha} \G(\ud t)$.
One can check that $H_{Q_\G}=H_{Q_\G}(\widehat{SH} \otimes \cA)$ is a finitely generated $\widehat{SH} \otimes \cA^0$-module,
whose generators are given by $\{\G_\g :  \g \in I\}$.
 Then, for $\G_{\alpha} \cdot \G_\b \in ({\widehat{SH}}\otimes \cA)^0$, we can find 
$A_{\a\b}^{(0) \g} (\ud t)  \in {\widehat{SH}}\otimes \cA^0$ and $\Lambda_{\a\b}^{(0)}(\ud t)\in ({\widehat{SH}} \otimes \cA)^{-1}$ (this is an ideal membership problem) such that
\eqn\rone{
\G_{\alpha} \cdot \G_\b = \sum_\g A_{\a\b}^{(0) \g} (\ud t) \cdot \G_\g + Q_\G (\Lambda_{\a\b}^{(0)}(\ud t)),
}
 since $Q_\G$ is a derivation of the binary product. Then the relation \rone\ can be rewritten as
\be
\G_{\alpha} \cdot \G_\b =\sum_\g A_{\a\b}^{(0) \g} (\ud t) \cdot \G_\g + K_\G (\Lambda_{\a\b}^{(0)}(\ud t))-\Delta (\Lambda_{\a\b}^{(0)}(\ud t)).
\ee
Let $R^{(1)} = \Delta (\Lambda_{\a\b}^{(0)}(\ud t))$ and we can find $A_{\a\b}^{(1) \g} (\ud t)\in {\widehat{SH}}\otimes \cA^0$ and $\Lambda_{\a\b}^{(1)}(\ud t)\in ({\widehat{SH}} \otimes \cA)^{-1}$ (again by an ideal membership problem) such that
\be
R^{(1)} = \sum_\g A_{\a\b}^{(1) \g} (\ud t) \cdot \G_\g + Q_\G (\Lambda_{\a\b}^{(1)}(\ud t)).
\ee
Set $R^{(2)} = \Delta (\Lambda_{\a\b}^{(1)}(\ud t))$ and we can find $A_{\a\b}^{(2) \g} (\ud t)\in {\widehat{SH}}\otimes \cA^0$ and $\Lambda_{\a\b}^{(2)}(\ud t)\in ({\widehat{SH}} \otimes \cA)^{-1}$ similarly such that
\be
R^{(2)} = \sum_\g A_{\a\b}^{(2) \g} (\ud t) \cdot \G_\g + Q_\G (\Lambda_{\a\b}^{(2)}(\ud t)).
\ee
We can continue this way and observe that this process stops after a finite number of steps.
Then Theorem \ref{diff} guarantees that 
we can choose $\Lambda_{\a\b}^{(M)}(\ud t)$ such that $\Delta(\Lambda_{\a\b}^{(M)}(\ud t))=
\G_{\a\b}$ so that we have 
\be
\G_{\alpha} \cdot \G_\b = \sum_\g \BA_{\a\b}{}^\g (\ud t)_\G \cdot \G_\g- 
\Delta(\Lambda_{\a\b}^{(M)}(\ud t)) + K_\G (\BL_{\a\b}(\ud t)),
\ee
where
\be
A_{\a\b}{}^\g (\ud t)_\G=\BA_{\a\b}{}^\g (\ud t)_\G&=& A_{\a\b}^{(0) \g}(\ud t)-A_{\a\b}^{(1) \g}(\ud t)-\cdots -A_{\a\b}^{(M) \g}(\ud t) \in \widehat{SH} \\
\BL_{\a\b}(\ud t)&=&\Lambda_{\a\b}^{(0)}(\ud t)-\Lambda_{\a\b}^{(1)}(\ud t)-\cdots -\Lambda_{\a\b}^{(M)}(\ud t) \in (\widehat{SH} \otimes \cA)^{-1}
\ee
for some positive integer $M$. Note that this equality is same as \klm, which finishes the proof.
\end{proof}

This enables us to compute a new formal deformation of the Griffiths period integral, i.e. 
the generating power series $\cC_\g(e^{\G(\ud t)_{\ud \varphi^H}}-1)$ attached to $C_{[\g]}$ and a formal deformation data $\ud \varphi^H$. 

\newsec{Appendix } \label{section6}

\subsection{The quantum origin of the Lie algebra representation $\rho_X$} \label{subs6.1}

In this section, we explain how we arrive at the key definition in \quantumrep, the definition of the Lie algebra representation $\rho_X$ attached to a hypersurface $X_G$. An interesting thing is that this has a quantum origin. It is well known that the representation theory of the Heisenberg group plays a crucial role in quantum field theory and quantum mechanics. We will focus on the Lie algebra representation of the Heisenberg Lie algebra
in order to explain our motivation for the definition in \quantumrep.
For each integer $m \geq 1$, we consider the universal enveloping algebra of the Heisenberg Lie algebra $\cH_m$ over $\Bbbk$ as follows:
\be
U(\cH_m) = \Bbbk \langle q^1, \cdots, q^m ; p_1, \cdots, p_m; z\rangle/ I
\ee
where $I$ is an ideal of the free $\Bbbk$-algebra $\Bbbk\langle q^1, \cdots, q^m ; p_1, \cdots, p_m; z \rangle$, generated by 
$$
[q^i, p_j]- z\delta^i_j,\ [q^i, z], \ [p_i,z], \ [q^i,q^j], \text{ and }  [p_i,p_j]
$$ 
for all $i,j=1, \cdots, m.$ Here, $[x,y]=x\cdot y - y \cdot x$ and $\delta^i_j$ is the usual Kronecker delta symbol. Then $(\cH_m, [\cdot, \cdot])$ is a nilpotent Lie algebra whose $\Bbbk$-dimension is $2m+1$.
Let $I_P$ be the left ideal of $U(\cH_m)$ generated by $p_1, \cdots, p_m$, i.e. $U(\cH_m)I_P  \subset I_P$.
Then the Schr\"odinger  representation of $U(\cH_m)$ is given by
\eqn\Schrd{
\rho_{\Sch}: U(\cH_m) \to \End_{\Bbbk}( U(\cH_m)/I_P),  \quad
 f \mapsto \big( g+I_P \to f\cdot g +I_P \big).
}
This celebrated representation has attained much attention from both physicists and mathematicians. This
representation can be used to derive both the Heisenberg matrix formulation and the Schr\"odinger wave differential equation formulation of quantum mechanics through Dirac's transformation theory. It also plays a crucial role in the study of theta functions and modular forms via the oscillator (also called Weil) representation coming from the 
Sch\"odinger representation.
 We have the Schr\"odinger Lie algebra representation $\rho_{\Sch}$ (with the same notation) of $\cH_m$ on the same $\Bbbk$-vector space $U(\cH_m)/I_P$ from \Schrd. Let $P$ be the abelian $\Bbbk$-sub Lie algebra of $\cH_m$ spanned by $p_1, \cdots, p_m$. If we restrict $\rho_{\Sch}$ to $P$, we get
 \eqn\Prep{
 \rho_{\Sch}|_{P}: P \to \End_{\Bbbk} (U(\cH_m)/I_P),
 }
 which is a Lie algebra representation of $P$.
 This Lie algebra representation of $P$ corresponding to $\rho_{\Sch}$ is our starting point to arrive at the definition in \quantumrep. 
 Note that the representation space $U(\cH_m)/I_P$ \textit{does not} have a $\Bbbk$-algebra structure, since $I_P$ is \textit{not} 
 a two-sided ideal.
This is a very important point in a mathematical algebraic formulation of quantum field theory (see the the first author's algebraic formalism of quantum field theory \cite{Pa10} for related issues). But we decided to simplify the quantum picture by introducing the Weyl algebra.\footnote{This has the benefit of dramatically reducing 
the length of our paper, although it hides certain important quantum features of the theory. In later papers, we will pursue how this phenomenon ($I_P$ is only a left $U(\cH_m)$-ideal, not two-sided) is related to understanding period integrals, which seems important in order to connect our theory to the theory of modular forms.}
 The $m$-th Weyl algebra (which is introduced to study the Heisenberg uncertainty principle in quantum mechanics), denoted $A_m$, is the ring of differential operators with polynomial coefficients in $m$ variables. It is generated by $q^1, \cdots, q^m$ and $\pa{q^1}, \cdots, \pa{q^m}$. Then we have
a surjective $\Bbbk$-algebra homomorphism
\be
r: U(\cH_m) \to A_m, 
\ee
defined by sending $z$ to 1, $q^i$ to $q^i$, and $p_i$ to $\pa{q^i}$ for $i=1, \cdots, m$ and extending the map in an obvious way. Note that $z; q^1, \cdots q^m; p_1, \cdots p_m$ are $\Bbbk$-algebra generators of $U(\cH_m)$.
Therefore the $m$-th Weyl algebra is a quotient algebra of $U(\cH_m)$. 
We further project the representation $\rho_{\Sch}$ to $A_m$ to get a Lie algebra representation of $A_m$ on $A_m/r(I_P)$, denoted $\rho_{\Wey}$,
\eqn\Weyr{
\rho_{\Wey}: A_m \to \End_{\Bbbk}(A_m/r(I_P)) \simeq \End_{\Bbbk}(\Bbbk[q^1, \cdots, q^m]).
}
The benefit of working with the Weyl algebra is that $A_m/r(I_P)$ is isomorphic to $\Bbbk[q^1, $ $\cdots, q^m]$ as a $\Bbbk$-vector space and so $A_m/r(I_P)$ has the structure of a commutative 
associative $\Bbbk$-algebra which is inherited from $\Bbbk[q^1, \cdots, q^m]$. Recall that in Definition \ref{Liedef} we assumed the representation space of the Lie algebra representation was a commutative associative $\Bbbk$-algebra.
 \footnote{This assumption is why we encounter $L_{\infty}$-homotopy theory, when we analyze period integrals. If it is not commutative, then we would need a different type of homotopy theory such as $A_\infty$-homotopy theory.}  
Note that $\rho_{\Wey}$ restricted to $r(I_P)$ is isomorphic to the representation on $\Bbbk[q^1, \cdots, q^m]$ obtained by applying the differential operators $\pa{q^1}, \cdots, \pa{q^m}$.
Now let $m=n+2$ and put $y=q^1, x_0=q^2, \cdots, x_n=q^{n+2}$ in order to connect to $\rho_X$, where $G(\underline x)$ is the defining equation of the smooth projective hypersurface $X_G$. Then when  $X_G = \BP^{n}$, the representation $\rho_X$ is isomorphic to $\rho_{\Wey}$. 
 Recall that $A=\Bbbk[y,\underline x] = \Bbbk[y_{-1}, y_0, \cdots, y_n]$. Dirac's transformation theory in quantum mechanics suggests considering
the following deformation of $\rho_{\Wey}:A_m \to \End_{\Bbbk}(A_m/r(I_P))$:
for $F \in A$, consider the formal operators
\be
\rho[i]:=\exp(-F(\underline y))\cdot \pa{y_{i}} \cdot \exp(F(\underline y)) , \quad i=-1, 0, \cdots, n.
\ee
These operators $\rho[i], i=-1, 0, \cdots, n,$ will act on $\exp(-F(\underline y)) \cdot A/r(I_P)\simeq A_m/ r(I_P) \simeq A$ (as $\Bbbk$-vector spaces). Formally, we can write and check
\eqn\Liequantum{
\eqalign{
\rho[i]=& \pa{y_i} + \Big[\pa{y_i}, F\Big] + \frac{1}{2}\Big[\Big[\pa{y_i}, F\Big],F\Big]+ \cdots
\cr
\equiv& \pa{y_i} + \prt{F}{y_i} \pmod{r(I_P)},
}
}
where $F=F(\underline y)$. For any $F\in A=\Bbbk[\underline y]$, define $\rho[F]:r(I_P) \to \End_{\Bbbk}(A_m/r(I_P)) $ via the rule
\be
r(p_i) \mapsto \big( Q+r(I_P) \mapsto \rho[i] \cdot Q +r(I_P)   \big), \quad i=-1, 0, \cdots, n.
\ee
\bep \label{SchLie}
If $F=y\cdot G(\ud x) \in A$, then 
we have a canonical isomorphism between the Lie algebra representations $\rho[F]$, defined above, and $\rho_X$, defined in \quantumrep.
\eep
\begin{proof}
The abelian Lie algebra $\frg$ in the definition of $\rho_X$ is isomorphic to
$r(I_P)$ and it is clear from \Liequantum\ that the natural $\Bbbk$-vector space isomorphism $A_m/r(I_P) \simeq A$ realizes the Lie algebra representation isomorphism between $\rho[F]$ and $\rho_X$.
 \end{proof}


\subsection{The homotopy category of $L_{\infty}$-algebras}\label{subs6.2}

An $L_\infty$-algebra is a generalization of an $\Z$-graded Lie-algebra
such that the graded Jacobi identity is only satisfied up to homotopy. 
An $L_\infty$-algebra is also known as a
 strongly homotopy Lie algebra in \cite{Sta63}, or Sugawara Lie algebra. It  has also appeared, albeit the dual version, in \cite{Su77}. It is also the Lie version of an $A_\infty$-algebra (strongly homotopy associative algebra), which is  the original example of homotopy algebra due to Stasheff. 
In this paper we encounter a variant of $L_\infty$-algebra such that its Lie bracket has degree one.
In other words, a structure of $L_\infty$-algebra on $V$ in our paper is equivalent to that of the usual $L_\infty$-algebra
on $V[1]$, where $V[1]$ means that $V[1]^m = V^{m+1}$ for $m \in \bZ$. We should also note that the usual presentation of $L_\infty$-algebras and $L_\infty$-morphisms via generators and
relations relies on unshuffles, which can be checked to be equivalent to our presentation based on partitions.

 Let $\art_{\Bbbk}$ denote the category of $\bZ$-graded artinian local $\Bbbk$-algebras with residue field $\Bbbk$ and $\widehat{\art_{\Bbbk}}$ be the category of complete $\bZ$-graded noetherian local $\Bbbk$-algebras. Let $R \in \hbox{\it Ob}(\widehat{\art_{\Bbbk}})$ concentrated in degree zero.
For $A \in \hbox{Ob}(\art_{\Bbbk})$, $\mm_A$ denotes the maximal ideal of $A$ which is a nilpotent $\Z$-graded super-commutative and associative $\Bbbk$-algebra without unit. 
Let $V= \bigoplus_{i\in \bZ} V^i$ be a $\bZ$-graded vector space over a field $\Bbbk$ of characteristic 0. If $x \in V^i$, we say that $x$ is a homogeneous element of degree $i$; let $|x|$ be the degree of a homogeneous element of $V$. For each $n \geq 1$ let $S(V)=\bigoplus_{n=0}^\infty S^n(V)$
be the free $\bZ$-graded super-commutative and associative algebra over $\Bbbk$ generated by $V$, which is the quotient algebra of the free tensor algebra $T(V)=\bigoplus_{n=0}^\infty T^n(V)$ by the ideal generated by $x\otimes y - (-1)^{|x||y|}y\otimes x$. Here $T^0(V)=k$ and $T^n(V) = V^{\otimes n}$ for $n \geq 1$.

\begin{definition}[$L_{\infty}$-algebra]
\label{shl}
The triple $V_L=(V,\underline{\ell}, 1_{V})$ is a unital $L_\infty$-algebra over $\Bbbk$  if
$1_V \in V^0$ and $\underline{\ell}=\ell_1,\ell_2,\cdots$ be a family such that

\begin{itemize}
\item $\ell_n \in \Hom(S^n V,V)^1$ for all $n\geq 1$.
\item $\ell_n(v_1,\cdots, v_{n-1}, 1_{V})=0$,  for all $v_1,\cdots, v_{n-1} \in V$, $n\geq 1$. 
\item for any $A \in \hbox{Ob}\left(\art_{\Bbbk}\right)$ and for all $n\geq 1$
\be
\sum_{k=1}^n\frac{1}{(n-k)! k!} {\ell}_{n-k+1}\left({\ell}_k\left( \g,\cdots,  \g\right), \g,\cdots,
 \g\right)=0,
\ee
whenever $\g \in (\mm_A\otimes V)^0$, where 
\be
&&{\ell}_n\big(a_1\otimes v_1, \cdots,  a_n\otimes v_n\big) \\
&&=(-1)^{|a_1|+|a_2|(1+|v_1|) +\cdots + |a_n|(1+|v_1|+\cdots +|v_{n-1}|)}
 a_1\cdots a_n\otimes \ell_n\left(v_1,\cdots, v_n\right).
\ee
\end{itemize}
\end{definition}

Every $L_\infty$-algebra in this paper is assumed to satisfy $\ell_n =0$ for all $n > N$ for some natural number $N$.

\begin{definition} [$L_{\infty}$-morphism]
\label{shlm}
A morphism of unital $L_\infty$-algebras from $V_L$ into $V'_L$ 
is a  family $\underline{\phi}=\phi_1,\phi_2,\cdots$ 
such that
\begin{itemize}
\item $\phi_n \in \Hom (S^n V, V')^0$ for all $n\geq 1$.
\item $\phi_1(1_{V})=1_{V'}$ and $\phi_n(v_1,\cdots, v_{n-1}, 1_{V})=0$, $v_1,\cdots, v_{n-1} \in V$, for all $n\geq 2$. 
\item for any $A \in \hbox{Ob}\left(\art_{\Bbbk}\right)$ and for all $n\geq 1$
\be
&&\sum_{j_1+ j_2 =n}
\frac{1}{j_1! j_2!} {\phi}_{j_1+1}\left( \ell_{j_2}(\g,\cdots,\g), \g,\cdots, \g\right) \\
&&=
\sum_{j_1+\cdots + j_r =n}\frac{1}{r!}\frac{1}{j_1!\cdots j_r!}
 \ell'_r\left({\phi}_{j_1}(\g,\cdots,\g), \cdots, {\phi}_{j_r}(\g,\cdots,\g)\right),
\ee
whenever $\g \in (\mm_A\otimes V)^0$,
where
\be
&&{\phi}_n\big(a_1\otimes v_1, \cdots, a_n\otimes v_n\big)\\
&&=(-1)^{|a_2||v_1| +\cdots + |a_n|(|v_1|+\cdots +|v_{n-1}|)}
 a_1\cdots a_n\otimes \phi_n\left(v_1,\cdots,v_n\right).
\ee
\end{itemize}
\end{definition}

It is convenient to introduce an equivalent presentation of $L_{\infty}$-algebras and morphisms based on partitions $P(n)$ of the set $\{1, 2, \cdots, n\}$, in order to present the proof of Theorem \ref{Desc} regarding the descendant functor. 

 Let us set up some notation related to partitions. A partition $\pi= B_1 \cup B_2\cup \cdots$ of the set $[n]=\{1,2, \cdots, n\}$ is a decomposition of $[n]$ into a pairwise
disjoint non-empty subsets $B_i$, called blocks. Blocks are ordered by the minimum element of each block and each block is ordered by the ordering induced from the ordering of natural numbers. The notation $|\pi|$ means the number of blocks in a partition $\pi$ and $|B|$ means the size of the block $B$. If $k$ and $k'$ belong to the same block in $\pi$, then we use
the notation $k \sim_\pi k'$. Otherwise, we use $k \nsim_\pi k'$. Let $P(n)$ be the set of all partitions of $[n]$.
%

Let $\tau_{i,j}$ be a transposition $(i j)$. We define the Koszul sign $\epsilon (\tau, \{x_i, x_j\}):=(-1)^{|x_i||x_j|}$ for any homogeneous elements $x_i, x_j \in V$. For a permutation $\sigma$ of $[n]$, we decompose $\sigma$ as a product of transpositions 
$\sigma = \prod_k \tau_{i_k,j_k}$. Then define $\epsilon(\sigma, \{x_1, \cdots, x_n \})=\prod_k \epsilon(\tau,\{x_{i_k}, x_{j_k}\})$ for homogeneous elements $x_1, \cdots, x_n \in V$.
%
The Koszul sign $\epsilon(\pi)$ for a partition $\pi = B_1 \cup B_2 \cup \cdots \cup B_{|\pi|}\in P(n)$ is defined by the Koszul sign $\epsilon(\sigma, \{x_1, \cdots, x_n\})$ of the permutation $\sigma$ determined by
\be
x_{B_1} \otimes \cdots \otimes x_{B_{|\pi|}} = x_{\sigma(1)} \otimes \cdots \otimes x_{\sigma(n)},
\ee
where $x_B= x_{j_1} \otimes \cdots \otimes x_{j_{r}}$ if $B=\{j_1, \cdots, j_{r}\}$. Let $|x_B|= |x_{j_1}|+\cdots+ |x_{j_{r}}|$.
Note that $\epsilon(\pi)$ depends on the degrees of $x_1, \cdots, x_n$ but we omit such dependence in the notation for simplicity. Then it can be checked that the following definitions \ref{L2} and \ref{LM2} are equivalent to the above ones; for an arbitrary finite set $\{v_1,\cdots, v_N\}$ of homogeneous elements in $V$,
consider $\Bbbk[[t_1,\cdots, t_N]]$, where $|t_j|= -|v_j|$, $1\leq j \leq N$. Let $J$ be the maximal ideal of $\Bbbk[[t_1,\cdots, t_N]]$,
so that $A:=\Bbbk[[t_1,\cdots, t_N]]/J^{N+1} \in \art_{\Bbbk}$. Let $\g$ be $\sum_{j=1}^N t_j v_j$, which is an element of $(\mm_A\otimes V)^0$.
Definitions \ref{shl} and \ref{shlm} imply, after simple combinatorics, the claimed relations in Definitions \ref{L2} and \ref{LM2} for all $n$ satisfying $1\leq n \leq N$. 
It is clear that the converse is also true.

\bed[$L_{\infty}$-algebra] \label{L2}
An $L_{\infty}$-algebra is a $\bZ$-graded vector space $V=\bigoplus_{m\in \bZ} V^m$ over a field $\Bbbk$ with a family $\underline\ell= \ell_1, \ell_2, \cdots$, where $\ell_k: S^k(V) \to V$ is a linear map
of degree 1 on the super-commutative $k$-th symmetric product $S^k(V)$ for each $k\geq 1$ such that
\eqn\Lalgebra{
\sum_{\substack{\pi \in P(n)\\ |B_i|=n-|\pi|+1}} \epsilon(\pi,i) \ell_{|\pi|}(x_{B_1},
\cdots, x_{B_{i-1}}, \ell(x_{B_i}), x_{B_{i+1}}, \cdots, x_{B_{|\pi|}}) = 0.
}
Here we use the the following notation:
\be
\ell(x_{B})&=&\ell_r(x_{j_1}, \cdots, x_{j_r}) \text{ if } B=\{j_1, \cdots, j_r\}\\
\epsilon(\pi, i) &=& \epsilon(\pi) (-1)^{|x_{B_1}|+\cdots+ |x_{B_{i-1}}|}.
\ee
An $L_{\infty}$-algebra with unity (or a unital $L_\infty$-algebra) is a triple $(V, \underline \ell,1_V)$, where $(V,\underline \ell)$
is an $L_{\infty}$-algebra and $1_V$ is a distinguished element in $V^0$ such that
$\ell_n(x_1, \cdots, x_{n-1}, 1_V)=0$ for all $n\geq 1$ and every $x_1, \cdots, x_{n-1} \in V$.
\eed

An ordinary Lie algebra can be viewed as a $L_\infty$-algebra concentrated in degree -1. It follows immediately for degree reasons that $\ell_1$= 0 for all $i \neq 2$, and that $\ell_2$ satisfies the axioms of a Lie bracket.

\bed[$L_{\infty}$-morphism]\label{LM2}
A morphism from an $L_{\infty}$-algebra $(V,\underline \ell)$ to an $L_{\infty}$-algebra $(V', \underline \ell')$ over $\Bbbk$
is a family $\underline \phi= \phi_1, \phi_2, \cdots$, where $\phi_k: S^k(V) \to V'$ is a
$\Bbbk$-linear map of degree 0 for each $k \geq 1$, such that
$$
\eqalign{
&\sum_{\pi \in P(n)} \epsilon(\pi) \ell'_{|\pi|}(\phi(x_{B_1}), \cdots, \phi(x_{B_{|\pi|}}))
\cr
&= \sum_{\substack{\pi \in P(n)\\ |B_i|=n-|\pi|+1}}\!\!\!\! \epsilon(\pi,i) \phi_{|\pi|}(x_{B_1},
\cdots, x_{B_{i-1}}, \ell(x_{B_i}), x_{B_{i+1}}, \cdots, x_{B_{|\pi|}}). 
}
$$
A morphism of unital $L_{\infty}$-algebras from $(V, \underline \ell, 1_V)$ to $(V', \underline \ell',1_{V'})$ over $\Bbbk$ is a morphism $\underline \phi$ of $L_{\infty}$-algebras such that $\phi_1(1_V)=1_{V'}$ and $\phi_n(x_1, \cdots, x_{n-1}, 1_V)
=0$ for all $n\geq 2$ and every $x_1, \cdots, x_{n-1} \in V$.
\eed

\begin{remark}
If one uses an interval partition of $[n]$ instead of $P(n)$, we can prove that this formalism gives an equivalent definition 
to the usual definition of $A_\infty$-algebras and $A_\infty$-morphisms.
\end{remark}

If we let $\pi =B_1 \cup \cdots B_{i-1} \cup B_i \cup B_{i+1} \cup \cdots \cup B_{|\pi|} \in P(n)$, then the condition $|B_i|=n-|\pi|+1$ in the summation implies that the set $B_1, \cdots, B_{i-1}, B_{i+1}, \cdots, B_n$ are singletons. 
Let $\ell_1=K$. For $n=1$, the relation \Lalgebra\ says that $K^2=0$. For $n=2$, the 
relation \Lalgebra\ says that 
\be
K\ell_2(x_1,x_2) + \ell_2(Kx_1, x_2) + (-1)^{|x_1|}\ell_2(x_1, Kx_2) =0.
\ee
For $n=3$, we have 
\be
\ell_2(\ell_2(x_1,x_2), x_3) 
&&+ (-1)^{|x_1|} \ell_2(x_1, \ell_2(x_2, x_3))
+(-1)^{(|x_1|+1)|x_2|}\ell_2(x_2, \ell_2(x_1,x_3))
\\
=&&
-K\ell_3(x_1,x_2,x_3) 
- \ell_3(Kx_1, x_2, x_3)\\
&&
- (-1)^{|x_1|} \ell_3(x_1, Kx_2, x_3) 
- (-1)^{|x_1|+|x_2|}\ell_3(x_1,x_2,Kx_3).
\ee
Because the vanishing of the left hand side is the graded Jacobi identity for $\ell_2$, we see
that $\ell_2$ fails to satisfy the graded Jacobi identity. Thus the failure
of $\ell_2$ being a graded Lie algebra is measured by the homotopy $\ell_3$.

\bec\label{Lcor}
A $\bZ$-graded vector space over $\Bbbk$ is an $L_{\infty}$-algebra with $\underline \ell =
\underline 0$, which we refer to as a trivial $L_{\infty}$-algebra. A cochain complex $(V,K)$
is an $L_{\infty}$-algebra with $\underline \ell = \ell_1$, where $\ell_1=K$. A differential
graded Lie algebra (DGLA) $(V,K,[ ,])$ is an $L_{\infty}$-algebra with $\underline \ell=\ell_1, \ell_2$, where $\ell_1=K$, and $\ell_2 (\cdot,\cdot)= [\cdot,\cdot]$.
\eec

Let $\ell_1=K$ and $\ell'=K'$. For $n=1$, the relation in Definition \ref{LM2} says that
\be
\phi_1 K= K' \phi_1. 
\ee
 For $n=2$, we have
\be
\phi_1(\ell_2(x_1,x_2))- \ell'_2(\phi_1(x_1), \phi_1(x_2)) 
=K'\phi_2(x_1,x_2)-\phi_2(Kx_1,x_2) -(-1)^{|x_1|} \phi_2(x_1, Kx_2).
\ee
Hence $\phi_1$ is a cochain map from $(V,K)$ to $(V', K')$, which fails to be an
 algebra map. The failure of being an algebra map is measured by the homotopy $\phi_2$. For $n=3$, the above says
 \be
 \phi_1(\ell_3(x_1,x_2,x_3))-\ell_3'(\phi_1(x_1),\phi_1(x_2),\phi_1(x_3)) \\
 +\phi_2(\ell_2(x_1,x_2),x_3) 
 +(-1)^{|x_1|}\phi_2(x_1,\ell_2(x_2,x_3)) 
+ (-1)^{(|x_1|+1)|x_2|} \phi_2(x_2, \ell_2(x_1,x_3))  
\\
- \ell_2'(\phi_2(x_1,x_2),\phi_1(x_3)) 
- \ell_2'(\phi_1(x_1), \phi_2(x_2,x_3))
-(-1)^{|x_1||x_2|}\ell_2'(\phi_1(x_2), \phi_2(x_1,x_3)) 
\\
 =K' \phi_3(x_1,x_2,x_3)-\phi_3(Kx_1,x_2,x_3)
 -(-1)^{|x_1|}\phi_3(x_1,Kx_2,x_3)
 -(-1)^{|x_1|+|x_2|}\phi_3(x_1,x_2,Kx_3).
 \ee

Now we define composition of morphisms.
\bed \label{compl}
The composition of $L_{\infty}$-morphisms $\underline \phi:V\to V'$ and $\underline \phi':V' \to V''$ is defined by
\be
\big(\phi' \bullet \phi \big)_n(x_1,\cdots, x_n) = \sum_{\pi \in P(n)} \epsilon(\pi) \phi'_{|\pi|}\big(\phi(x_{B_1}), \cdots, \phi(x_{B_{|\pi|}})\big),
\ee
for every homogeneous elements $x_1, \cdots, x_n \in V$ and all $n \geq 1$.
\eed

Then it can be checked that unital $L_{\infty}$-algebras over $\Bbbk$ and $L_\infty$-morphisms form a category.

\bed
The cohomology $H$ of the $L_{\infty}$-algebra $(V, \underline \ell)$ is the cohomology of the underlying complex $(V, K=\ell_1)$. An $L_{\infty}$-morphism $\underline \phi$ is a quasi-isomorphism if $\phi_1$ induces an isomorphism on cohomology.
\eed

In fact, any $L_\infty$-algebra can be strictified to give a DGLA, i.e. any $L_\infty$-algebra is $L_\infty$-quasi-isomorphic to a DGLA.

\bed
An $L_{\infty}$-algebra $(V,\underline \ell)$ is called minimal if $\ell_1=0$. 
\eed
We recall the following well-known theorem, sometimes called the homotopy transfer theorem.
See the chapter 10 of \cite{LV} for a detailed discussion on the homotopy transfer theorem.

\bet \label{transfer}
There is the structure of a minimal $L_{\infty}$-algebra $(H, \underline \ell^H=\ell_2^H,
\ell_3^H, \cdots)$ on the cohomology $H$ of an $L_{\infty}$-algebra $(V,\underline \ell)$ over $\Bbbk$ together with an $L_{\infty}$-quasi-isomorphism $\underline \phi$ from $(H, \underline \ell^H)$ to $(V, \underline \ell^K)$. Both the minimal $L_{\infty}$-algebra structure and the $L_{\infty}$-quasi-isomorphism are not unique, while $\ell_2^H$ is defined uniquely.
\eet

\bed \label{smoothformal}
An $L_{\infty}$-algebra $(V, \underline \ell)$ is called {formal} if it is quasi-isomorphic to an $L_\infty$-algebra $(V', \ud 0)$ with the zero $L_\infty$-structure.
\eed

\begin{definition}
\label{shlh}
Two $L_\infty$-morphisms $\underline{\phi}$ and $\underline{\tilde{\phi}}$ of unital $L_\infty$-algebras from $(V_L,\underline \ell, 1_V)$ into $(V_L',\underline \ell', 1_{V'})$ 
are $L_\infty$-homotopic, denoted by $\underline{\phi}\sim_{\infty}\underline{\tilde{\phi}}$, if
there is a polynomial family $\underline{\lambda}(\t)=\lambda_1(\t),\lambda_2(\t),\cdots$ in $\t$, where 
\begin{itemize}
\item $\lambda_n(\t) \in \Hom(S^n V, V')^{-1}$ for all $n\geq 1$.
\item $\lambda_n(\t)(v_1,\cdots, v_{n-1}, 1_{V})=0$, $v_1,\cdots, v_{n-1} \in V$, for all $n\geq 1$,
\end{itemize}
and a polynomial family $\underline{\Phi}(\t)=\Phi_1(\t),\Phi_2(\t),\cdots$ of $L_\infty$-morphisms, where
\begin{itemize}
\item $\Phi_n(\t)\in \Hom (S^n V, V')^{0}$ for all $n\geq 1$,
\item $\Phi_1(\t)(1_V)=1_{V'}$ and $\Phi_n(\t)(v_1,\cdots, v_{n-1}, 1_V)=0$, $v_1,\cdots, v_{n-1} \in V$, for all $n\geq 2$,
\end{itemize}
and $\underline{\Phi}(0)=\underline{\phi}$ and $\underline{\Phi}(1)=\underline{\tilde{\phi}}$,
such that $\ud \Phi$ satisfies the following flow equation 
\be
&&\prt{}{\t}{\Phi}_n(\t)({\g},\cdots,{\g})\\
&&=
\sum_{k=1}^n\sum_{j_1+\cdots + j_r =n-k}\frac{1}{k!}\frac{1}{r!}\frac{1}{j_1!\cdots j_r!}
 \ell'_{r+1} \left({\Phi}_{j_1}(\t)(\g,\cdots,\g), \cdots, {\Phi}_{j_r}(\t)(\g,\cdots,\g),{\lambda}_k(\t)(\g,\cdots,\g)\right)\\
&&+\sum_{j_1+ j_2 =n}
\frac{1}{j_1! j_2!} {\lambda}_{j_1+1}(\t)\left(\g,\cdots, \g, \ell_{j_2}(\g,\cdots,\g)\right)
\ee
for all $n\geq 1$ and $\g \in (\mm_A\otimes V)^0$ whenever $A \in\hbox{\it Ob}(\art_{\Bbbk})$.
\end{definition}

It can be checked that $\sim_\infty$ is an equivalence relation (see 4.5.2 of \cite{Kon03} for another form of the above definition).

\begin{lemma} \label{hoco}
Consider  $L_\infty$-morphisms
$\underline{\phi}:V_L \rightarrow V_L'$ and 
$\underline{\phi}':V_L' \rightarrow V_L''$.
Let $\underline{\tilde{\phi}} \sim_\infty \underline{\phi}$ and $\underline{\tilde{\phi}}' \sim_\infty \underline{\phi}'$.
Then
$\underline{\tilde{\phi}}'\bullet \underline{\tilde{\phi}}\sim_\infty \underline{\phi}'\bullet \underline{\phi}$.
\end{lemma}

\begin{proof}
Let $\underline{\Phi}(\t)$, where
$\underline{\Phi}(0)=\underline{\phi}$ and $\underline{\Phi}(1)=\underline{\tilde{\phi}}$, be a polynomial solution to the flow
equation with a polynomial family $\underline{\lambda}(\t)$. 
Let $\underline{\Phi}'(\t)$, where
$\underline{\Phi}'(0)=\underline{\phi}'$ and $\underline{\Phi}'(1)=\underline{\tilde{\phi}}'$, be a polynomial solution to the flow
equation with polynomial family $\underline{\lambda}'(\t)$. 
Then $\underline{\Phi}''(\t) := \underline{\Phi}'(\t)\bullet\underline{\Phi}(\t)$ is a polynomial family\footnote{When we apply Definition \ref{compl}, ${\Phi}'_m, m\geq 1$ are viewed as $\tau$-multilinear functions. For example, if $\Phi_1(x)(\tau)=A(x)\tau^3 +B(x)$ and $\Phi_1'(x)(\tau)=A'(y) \tau^2 +B'(y)$ for $x \in V, y \in V'$, then 
$$
(\Phi'_1 \circ \Phi_1)(x)(\tau)=A'(A(x))\tau^5 +A'(B(x))\tau^2 +B'(A(x))\tau^3 + B'(B(x)).
$$
} 
in $\Hom(V,V'')^0$
such that $\underline{\Phi}''(0) = \underline{\phi}'\bullet\underline{\phi}$ and 
 $\underline{\Phi}''(1) = \underline{\tilde{\phi}}'\bullet\underline{\tilde{\phi}}$. 
 It can be checked that, for all $n\geq 1$ and $\g \in (\mm_A\otimes V)^0$ whenever $A \in \hbox{\it Ob}(\art_{\Bbbk})$,
\be
&&\prt{}{\t}{\Phi}_n''({\g},\cdots,{\g})\\
&&\sum_{k=1}^n\sum_{j_1+\cdots + j_r =n-k}\frac{1}{k!}\frac{1}{r!}\frac{1}{j_1!\cdots j_r!}
 \ell_{r+1}'' \left({\Phi}_{j_1}''(\g,\cdots,\g), \cdots, {\Phi}_{j_r}''(\g,\cdots,\g),{\lambda}_k''(\g,\cdots,\g)\right)   \\
&&+\sum_{j_1+ j_2 =n}
\frac{1}{j_1! j_2!} {\lambda}_{j_1+1}''\left(\g,\cdots, \g, \ell_{j_2}(\g,\cdots,\g)\right)
\ee
where $\underline{\lambda}''(\t)$ is the polynomial family in $\Hom(V,V'')^{-1}$
given by
\be
&&\lambda_n'' \big(\g,\cdots, \g\big)
=\sum_{j_1+\cdots + j_r=n}
\frac{1}{r!}\frac{1}{j_1!\cdots j_r!}\lambda_r'\Big(
 \Phi_{j_1}(\g,\cdots, \g),\cdots,
\Phi_{j_r}(\g,\cdots, \g)
\Big)\\
&&+\sum_{j_1+\cdots + j_r=n}
\frac{1}{r!}\frac{1}{j_1!\cdots j_r!}{\Phi}_r'\Big(
 \Phi_{j_1}(\g,\cdots, \g),\cdots, \Phi_{j_{r-1}}(\g,\cdots, \g),\lambda_{j_r}(\g,\cdots, \g)\Big).
 \ee
This proves the lemma.
 \end{proof}

The above lemma implies that the homotopy category of unital $L_\infty$-algebras is well-defined, where the morphisms in that
category consists of $L_\infty$-homotopy classes of $L_\infty$-morphisms.

Let $(V,K,[,])$ and $(V',K', [,]')$ be an ordered pair of DGLAs over $\Bbbk$. Then a morphism $f:V
\to V'$ of DGLAs is an $L_{\infty}$-morphism $\underline \phi = \phi_1$  such that $\phi_1=f$.
If $f:V\to V'$ is a DGLA morphism, then $\tilde f= f+K's +s K$ is a cochain map homotopic to $f$ by the cochain homotopy $s$. In this case, there is an $L_{\infty}$-morphism $\underline{\tilde \phi}= \tilde \phi_1, \tilde \phi_2, \cdots $ which is homotopic to $\underline \phi= \phi_1=f$ by an $L_{\infty}$-homotopy $\underline \lambda= \lambda_1, \lambda_2, \cdots$ such that $\lambda_1=s$ and $\tilde \phi_1 = \tilde f$.
Let $\tilde f: V \to V'$ be a cochain map and $[\tilde f]$ be the cochain homotopy class of $\tilde f$. Then there is a representative $f$ of $[\tilde f]$ which is a DGLA morphism if and only if $\tilde f$ can be extended to an $L_{\infty}$-morphism $\underline{\tilde \phi}$, i.e. $\tilde\phi_1=\tilde f$ which is $L_{\infty}$-homotopic to a cochain map.

%
%
%

\end{document}